\documentclass[journal]{IEEEtran}

\usepackage{cite}
\usepackage{graphicx}
\usepackage{amsmath}
\usepackage{amssymb}
\usepackage{amsthm}
\usepackage{optidef}
\usepackage{algorithmic}
\usepackage{epstopdf}
\usepackage{subfig}
\usepackage{algorithm}
\usepackage{cuted}
\usepackage{array}
\usepackage{bbm}
\usepackage{xcolor}
\DeclareMathAlphabet{\mathpzc}{OT1}{pzc}{m}{it}

\usepackage[utf8]{inputenc}
\usepackage[english]{babel}
\usepackage{lettrine}
\usepackage{geometry}
\newcommand{\subparagraph}{}
\usepackage[compact]{titlesec}
\titlespacing{\section}{0pt}{1ex}{1ex}
\titlespacing{\subsection}{0pt}{1ex}{1ex}
\usepackage[font=small,skip=2pt]{caption}
\newcommand\blfootnote[1]{%
	\begingroup
	\renewcommand\thefootnote{}\footnote{#1}%
	\addtocounter{footnote}{-1}%
	\endgroup
}

\newtheorem{myth}{Theorem}
\newtheorem{lemma}{Lemma}
\newtheorem{remark}{Remark}
\newtheorem{corollary}{Corollary}[myth]
\newtheorem{clm}{Claim}

\newcommand{\eg}{\textit{e}.\textit{g}.}
\newcommand{\ie}{\textit{i}.\textit{e}.}
\newcommand{\figwidth}{80mm}

\begin{document}
\twocolumn[
\begin{@twocolumnfalse}
\title{Quickest Detection of Deception Attacks on Cyber-Physical Systems with a Parsimonious
Watermarking Policy}

%\author{Arunava Naha$^{1}$, André Teixeira$^{1}$, Anders Ahlén$^{1}$ and Subhrakanti Dey$^{2}$% <-this % stops a space
%	\thanks{*This work is supported by The Swedish Research Council (VR) under grants 2017-04053 and 2018-04396, and by the Swedish Foundation for Strategic Research.}% <-this % stops a space
%	\thanks{$^{1}$Arunava Naha, André Teixeira, and Anders Ahlén are with the Department of Electrical Engineering, Uppsala University, 751 03 Uppsala, Sweden
%		{\tt\small arunava.naha@angstrom.uu.se, andre.teixeira@angstrom.uu.se, and Anders.Ahlen@angstrom.uu.se}}%
%	\thanks{$^{2}$Subhrakanti Dey is with the Department of Electronic Engineering, Hamilton Institute, National University of Ireland, Maynooth, Ireland. He is also with the Department of Electrical Engineering, Uppsala University, 751 03 Uppsala, Sweden
%		{\tt\small Subhra.Dey@signal.uu.se}}%
%}

\author{Arunava Naha$^{a,*}$ , André Teixeira$^{b}$, Anders Ahlén$^{a}$ and Subhrakanti Dey$^{a}$ \\ % <-this % stops a space \\
{\small
$^{a}$Electrical Engineering, Uppsala University, Uppsala, Sweden             \\                              
$^{b}$Department of Information Technology, Uppsala University, Uppsala, Sweden
}
%	\thanks{*This work is supported by The Swedish Research Council under grants 2017-04053 and 2018-04396, and by the Swedish Foundation for Strategic Research.}% <-this % stops a space
%	\thanks{$^{1}$Arunava Naha, and Anders Ahlén are with the Department of Electrical Engineering, Uppsala University, 75103 Uppsala, Sweden.
%		{\tt\small arunava.naha@angstrom.uu.se, and Anders.Ahlen@angstrom.uu.se}}%
%	\thanks{$^{2}$André Teixeira is with the Department of Information Technology, Uppsala University, PO Box 337, SE-75105, Uppsala, Sweden.
%		{\tt\small andre.teixeira@it.uu.se}}%
%	\thanks{$^{3}$Subhrakanti Dey is with the Department of Electronic Engineering, Hamilton Institute, National University of Ireland, Maynooth, Ireland. He is also with the Department of Electrical Engineering, Uppsala University, 751 03 Uppsala, Sweden
%		{\tt\small Subhra.Dey@signal.uu.se}}%
}

\maketitle

	\begin{abstract}
Adding a physical watermarking signal to the control input of a networked control system increases the detection probability of data deception attacks at the expense of increased control cost. This paper proposes a parsimonious policy to limit the average number of watermarking events when the attack is not present, which in turn reduces the control cost. We model the system as a stochastic optimal control problem and apply dynamic programming to minimize the average detection delay (ADD) for fixed upper bounds on false alarm rate (FAR) and an average number of watermarking events (ANW) before the attack. Under practical circumstances, the optimal solution results in a two threshold policy on the posterior probability of attack, derived from the Shiryaev statistics for sequential change detection and assuming the change point is a random variable. We derive asymptotically approximate analytical expressions of ADD and FAR, applying the non-linear renewal theory for non-independent and identically distributed data. The derived expressions reveal that ADD reduces with the increase in the Kullback-Leibler divergence (KLD) between the post- and pre-attack distributions of the test statistics. Therefore, we further design the optimal watermarking that maximizes the KLD for a fixed increase in the control cost. The relationship between the ANW and the increase in control cost is also derived. Simulation studies are performed to illustrate and validate the theoretical results.
	\end{abstract}
\end{@twocolumnfalse}
]

\blfootnote{This work is supported by The Swedish Research Council under grants 2017-04053 and 2018-04396, and by the Swedish Foundation for Strategic Research. $^{*}$Corresponding author A.~Naha. Tel. +46-76-4552158.}

\section{Introduction} \label{sec:intro}
Nowadays, cyber-physical systems (CPS) with embedded software, processors, sensor network, and other physical components are getting deployed for advanced healthcare, smart buildings, smart manufacturing units, intelligent transport systems, defence purposes, smart grids, etc. \cite{Satchidanandan2017}. CPS integrate cyber and physical components by exchanging data over the wireless network and provide autonomy, reliability, accuracy, and real-time control without human involvement \cite{Satchidanandan2017, Alguliyev2018}. Along with their numerous advantages, there is also a growing concern regarding the safety and security of CPS. Due to the use of commodity software and off-the-shelf networking components, unattended operations, and a few other reasons CPS are vulnerable to adversarial attacks on the cyber or/and physical layer \cite{Mo2015}. Cryptography, firewalls, user authentications, digital watermarking, etc. are already in place to protect CPS from cyber attacks. However, such protection mechanisms may not be adequate against physical attacks. For example, during the Stuxnet attack \cite{Langner2011}, attackers issued harmful exogenous control inputs to increase the pressure of the centrifuges beyond the safety limit at a uranium enrichment plant in Iran. To remain stealthy during the attack, attackers also replaced the true observation from the system with previously recorded data. There are a few other examples, such as the attack on a sewage system in Australia \cite{Abrams2008}, the attack on the Davis-Besse nuclear power plant in Ohio, USA \cite{Alvaro1992}, etc., where cyber protection schemes failed to prevent or detect the attacks. Attacks on CPS may cause considerable monetary loss and pose threats to human safety \cite{Satchidanandan2017}.

Attack strategies for the physical layer of CPS can be broadly classified into two groups, data deception attacks and denial of service (DoS) attacks. In data deception attacks, the adversary feeds the system with false data \cite{Satchidanandan2017, Mo2015}. Replay attacks are one kind of data deception attack, where the attacker replaces the true observations with previously recorded data to remain stealthy \cite{Mo2015}. In DoS attacks, the attacker's objective is to disrupt the availability of data. The attacker may achieve that by overpowering the wireless network \cite{Salimi2019}. In an attack scenario, the attacker's objective is to remain stealthy as long as possible and cause maximum damage to the system. The inherent noise and uncertainties in CPS assist the attacker in achieving such an objective. The role of a control system engineer is to detect the attack as soon as possible to minimize the damage. In this paper, we have studied data deception attacks on networked control systems (NCS), where the attacker replaces the true observation with fake data. 

\subsection{Related Work}
\label{subsec:related_work}
Researchers are working on different challenges to secure CPS against attacks on the physical layer, such as the study of different attack strategies \cite{Park2019, Chen2018}, attack resilient state estimation \cite{Fawzi2014, Du2019, Nicola2018}, attack detection strategies \cite{Satchidanandan2017,Mo2015,Pasqualetti2013,Mousavinejad2018,Ge2019,Ko2019a,Mo2014,Fang2020,Satchidanandan2020}, etc. Detection strategies for the attacks on the physical layer of CPS can be broadly divided into two groups, passive and active schemes. Under the passive attack detection scheme, the innovation signal from the state estimator or the observation signal is subjected to various statistical tests  \cite{Pasqualetti2013, Mousavinejad2018, Ge2019}. However, as studied in the literature, passive detection schemes generally have an unsatisfactory probability of detection in the presence of noise and uncertainties \cite{Mo2009}. On the other hand, under the active attack detection scheme, physical watermarking signals are added to the control inputs, and various statistical tests are used to  check the authenticity of the received observations. The physical watermarking scheme was first introduced in \cite{Mo2009} to detect replay attacks by adding an iid watermarking signal to the control input and performing a $\chi^2$ test using the innovation signal from the state estimator. The method in \cite{Mo2009} is improved by designing an optimal watermarking signal in \cite{Mo2014}. Instead of an iid watermarking scheme, the watermarking signal generated from a hidden Markov Model (HMM) is studied in \cite{Mo2015}. A sequential attack detection scheme using the CUSUM statistics evaluated from the joint distribution of the added watermarking and the innovation signal is studied in \cite{naha_replay_attack}. Besides the innovation signal, the observation signal is also used to generate residue signals for the attack detections \cite{Satchidanandan2017}. In \cite{Mo2015, Satchidanandan2017, naha_replay_attack}, watermarking signals are added to the control inputs for all the time instants till the point of attack detection. The addition of physical watermarking to the control input increases the probability of attack detection at the expense of increased control cost \cite{Mo2015}. The relation between the increase in the linear quadratic Gaussian control cost, $\Delta LQG$, and the watermarking signal variance is studied in \cite{Mo2015}. Since the attack is a less frequent event, adding the watermarking signal during the normal operation for a long time can  increase the total control cost significantly \cite{Fang2020} and unnecessarily. In the current paper, we have studied an evidence-based watermarking policy to reduce the increase of control cost  before an attack, and, at the same time, achieve satisfactory detection performance.

Researchers are exploring diverse approaches to reduce the increase in the control cost due to the added watermarking and maintain satisfactory detection performance. In one approach, the authors have added the watermarking periodically to the control inputs and kept a balance between the improvement in the control cost and the increase in the detection delay \cite{Fang2020}. Another approach is to add watermarking directly to the observations \cite{Trapiello2019, Ferrari2017, Ye2019}. In this approach, the authenticity of the observations is first examined at the receiving end, and then the watermarking signal is filtered out before using the observations in the controller. Since the watermarking signal is filtered out, the control cost does not increase. Different kinds of watermarking signals are used in this context, such as sinusoidal \cite{Ferrari2017}, time-varying sinusoidal \cite{Sanchez2019}, random noise \cite{Ye2019}, multiplicative to the observations \cite{Trapiello2019}, etc. However, these methods may fail in the scenario, where the attacker hijacks the sensor node and feeds the fake data before the addition of the watermarking. In general, the physical watermarking-based methods targeting to reduce the increased control cost or more traditional always present watermarking-based methods use batch processing of data, \ie, innovation signal or observation signal. Therefore, those methods do not address the problem of the quickest attack detection. However, we know that early detection of attacks is of paramount importance for CPS to reduce the amount of damage. Therefore, in this paper, we studied the problem of the quickest detection of attacks which uses watermarking parsimoniously to reduce the loss in control performance prior to an attack. The literature on the quickest detection of a change point by sequential analysis of data dates back several decades. A brief  review on the quickest change detection techniques is provided in the following paragraph.        

The quickest change detection methods can be classified into two broad groups depending upon the assumption of the model of the change point \cite{Tartakovsky2014}. In one approach, which is also called the minimax approach, the change point is modelled as deterministic but unknown. The cumulative sum (CUSUM) technique is one of such minimax approaches, which was first introduced by Lorden \cite{lorden1971procedures}. In the other approach, the Bayesian approach, the change point is modelled as a random variable (RV) with some prior distribution. The Bayesian change point detection technique was first introduced by Shiryaev \cite{Shiryaev1963}. The original Shiryaev rule was proposed for the data with different iid distributions, before and after the change point. Finding an optional detection rule for the general non-iid data is difficult \cite{Fuh2019}. In \cite{Series2010}, an optimal detection rule is developed for homogeneous finite-state Markov chains. A slightly different approach is followed in \cite{Tartakovsky2005,Tartakovsky2017}, where the authors proved that the Shiryaev rule, with minor modifications, is an asymptotically optimal quickest change detection rule under the conditions given in (3), (4), (21) and (23) of \cite{Tartakovsky2017}. That means the Shiryaev procedure minimizes the average detection delay (ADD) for a fixed upper threshold on the false alarm rate (FAR) for the non-iid data provided that the threshold $\rightarrow \infty$ and a few other conditions are satisfied. The condition (3) of \cite{Tartakovsky2017} for the optimality is that the prior distribution of the change point must satisfy (\ref{eqn:c1}).
\begin{equation}
	\lim_{k \rightarrow \infty}\frac{\log\text{P}\left\{\Gamma\ge k+1\right\}}{k}=-c, \text{ } c \ge 0, \label{eqn:c1} 
%	&\lim_{k \rightarrow \infty}\frac{\log\text{P}\left\{\Gamma\ge k+1\right\}}{k}=0, \label{eqn:c2}
\end{equation}
where $\Gamma$ is the change point. That means the exponential rate of convergence of the prior distribution must be $c \ge 0$, where $c >0$ indicates the prior distribution has an exponential right tail, and $c=0$ indicates the prior distribution is heavy-tailed \cite{Fuh2019}. An attacker will always try to remain stealthy for a long time because the longer time the attacker remains undetected, the more damage can be caused \cite{Fang2020}. On the other hand, the defender should design a detection mechanism that will detect the attack as soon as possible with an acceptable FAR to reduce the amount of damage. Therefore, we have used the Bayesian approach in this paper, which minimizes the ADD, whereas the other method, the minimax approach, only minimizes the worst-case ADD (computed over all possible attack start points) \cite{Tartakovsky2014}.  Additionally, our work in this paper is inspired by two other prior works \cite{Premkumar2008, Banerjee2012}. The quickest intrusion detection problem is studied in \cite{Premkumar2008}, where only a minimal set of sensors from a sensor network is kept active at a particular time instant. The problem of quickest change detection is also studied in \cite{Banerjee2012} with upper bounds on the average number of sensor data used before the change point and the FAR. In both the problem formulations, the underlying data was assumed to be iid, which is not the case for the system under study in this paper. However, similar to several other works on change-point detections \cite{Premkumar2008, Banerjee2012}, we have also assumed the distribution of the change point, \ie, the attack start point, to be a geometric distribution with parameter $\rho$, which satisfies the condition given in (\ref{eqn:c1}). 

\subsection{Contributions}
\label{subsec:contributions}
In our previous work \cite{naha_replay_attack}, we studied in detail that the worst-case ADD decreases with the increase in $\Delta LQG$, which denotes the increases in the LQG control cost due to the addition of watermarking for the always-present watermarking scheme. Additionally, $\Delta LQG$ is proportional to the watermarking signal power. In other words, an attack can be detected early with higher watermarking signal power, \ie, at the expense of increased control cost. Therefore, in this paper, we propose a method to reduce the average number of watermarking (ANW) events used before the attack start point, which reduces the average watermarking signal power and subsequently $\Delta LQG$. We formulate the task at hand as a stochastic optimal control problem to minimize the ADD for fixed upper bounds on FAR and ANW and apply dynamic programming to solve it. Similar to any other detection technique, there is always a trade-off between ADD and FAR \cite{Tartakovsky2005}. We have studied the structure of the dynamic programming solution, \ie, the solution of the Bellman equation, and found that the optimal policy is a two threshold policy with thresholds $Th^s$ and $Th^d$, $Th^d \ge Th^s$, on the posterior probability of attack $p_k$ under practical circumstances. In other words, if $p_k \ge Th^s$, then we add watermarking to the $(k+1)$-th control input. On the other hand, if $p_k \ge Th^d$, we decide that the attack is present in the system and terminate the process. Our study shows that $Th^s$ primarily controls the ANW, which in turn controls the $\Delta LQG$ value, and $Th^d$ primarily controls the ADD and FAR. Asymptotically approximate analytical expressions of ADD and FAR are derived by applying non-linear renewal theory for non-iid data. The derived expression of ADD indicates that the ADD reduces with the increase of the Kullback-Leibler divergence (KLD) between the post- and pre-attack distributions of the test data. Additionally, the derived analytical expression of KLD for our problem formulation provides the relationship between the KLD and the watermarking signal variance. Therefore, we use this relationship to derive the optimal watermarking signal variance, which maximizes the KLD for a given upper bound on $\Delta LQG$. We have also obtained an expression of $\Delta LQG$ for a given ANW. We have reported a preliminary simulation study on this problem previously for a single-input single-output (SISO) system in \cite{ECC21}. In the current paper, we have performed a more in-depth theoretical analysis of the problem for general multi-input and multi-output (MIMO) system models. Our main contributions are as follows.

\begin{enumerate}
	\item To the best of our knowledge, this is the first time the Bayesian approach is applied for the quickest detection of data deception attacks on NCS with a parsimonious watermarking policy to reduce the control cost. 
	\item We have derived asymptotically approximate analytical expressions of ADD, FAR and $\Delta LQG$ that facilitate the optimal design of the watermarking process. 
	\item We have optimised the watermarking signal variance to maximise KLD, which improves ADD for a fixed upper bound on the $\Delta LQG$.
\end{enumerate}

\label{subsec:organization}
The paper is organized as follows. Section~\ref{sec:sys_model} discusses the system model and the attack strategy considered in this paper. The defence mechanism is explained in Section~\ref{sec:def_strategy}. Section~\ref{sec:ADD_FAR_LQG} provides the analytical expressions of ADD, FAR and the relationship between the ANW and $\Delta LQG$. It also explains the optimization framework for the watermarking signal variance. Section~\ref{sec:num_results} presents and discusses the simulation results. Section~\ref{sec:conclusion} concludes the paper.
\subsection{Notations}
\label{subsec:notations}
We have used capital bold letters, \eg, $\bf{A}$, $\bf{B}$, etc. to specify matrices and small bold letters, \eg, $\bf{x}$, $\bf{y}$, etc. to specify vectors, unless specified otherwise. Some special notations are given in Table~\ref{tab:notations}. 
\begin{table}[h!]
	\begin{center}
		\caption{Notations}
		\label{tab:notations}
		\begin{tabular}{l|l} 
			\hline \hline
			Symbol & Description \\
			\hline
			${\rm I\!R}^{n}$ & The set of $n\times 1$ real vectors \\
			${\rm I\!R}^{m\times n}$ & The set of $m\times n$ real matrices \\
			$\hat {\left\{\cdot \right\}}$  & Estimated quantity   \\
			$\text{E}\left[\cdot\right]$ & Expectation operator \\
			${\hat{\bf{x}}}_{k|k}$ & Estimated state at $k$-th instant using measurements\\
			&  up to $k$-th instant \\
			$\left[\cdot \right]^T$ & Transpose of a matrix or vector \\
			$\mathcal{N}(\mu,{\bf \Sigma})$ & Gaussian distribution with mean $\mu$ and variance $\bf \Sigma$  \\
			$\left \{ \cdot \right \}\cup \left \{\cdot \right \}$ & Union of two sets \\
			$\bf \Sigma \ge 0$  & $\bf \Sigma$ is a positive semi-definite matrix \\
			$\bf \Sigma > 0$  & $\bf \Sigma$ is a positive definite matrix \\
			${\bf x}_{a,k}$, ${\bf u}_{d,k}$,  & $k$-th instant values of ${\bf x}_{a}$, ${\bf u}_{d}$, ${\bf e}_{s}$, etc. \\
			${\bf e}_{s,k}$, etc. & \\
			$ {\left\{\cdot \right\}}^*$& Optimum quantity \\
			$[\cdot]_{ij}$ & $i$-th row and $j$-th column element of a matrix \\
			$\text{P}\left\{\cdot \right\}$ & Probability measure \\
			$\Pi_k$ & Probability of the event $\left\{\Gamma = k\right\}$ \\
			$\text{P}_k$ & Probability measure when the change point $\Gamma = k$ \\
			$\text{P}^{\Pi}\left\{\cdot\right\}$ & Average probability measure, $=\sum_{k=1}^\infty{\Pi}_k\text{P}_k\left\{\cdot\right\}$ \\
			$\text{E}^{\Pi}$ & Expectation with respect to probability measure $\text{P}^{\Pi}$ \\
			$\mathpzc{f}_{1,j}$, $\mathpzc{f}_{2,j}$, $\mu_{1,j}$,  & $j$-th instant values of $\mathpzc{f}_{1}$, $\mathpzc{f}_{2}$, $\mu_{1}$, $\mu_{2}$, $\Sigma_{1} $, $\Sigma_{2}$\\
			$\mu_{2,j}$, $\Sigma_{1,j} $, $\Sigma_{2,j}$ &  \\
			$|\cdot|$ & Determinant of a matrix or absolute value of a scalar \\
			$\bar {\left\{\cdot \right\}}$  & Mean value of a quantity   \\
			$\text{tr}(\cdot)$ & Trace of a matrix \\
			$\left\{{\bf X}\right\}^{k-1}_1$  & $\left\{X_i:1\le i \le k-1\right\}$\\
			$\mathbbm{1}_{\left\{ condition \right\}}$ & Indicator function, 1 if condition is true, 0 otherwise \\
			\hline \hline
		\end{tabular}
	\end{center}
\end{table}
\section{System Model} \label{sec:sys_model}
The system model during normal operations and the model with the data deception attack are discussed in this section.

\subsection{System model during normal operation} \label{subsec:normal_system}
 A schematic diagram of a standard NCS during the normal operation is shown in Fig.~\ref{fig:normal_system}.
\begin{figure}[h!]
	\centering
	\includegraphics[width=50mm]{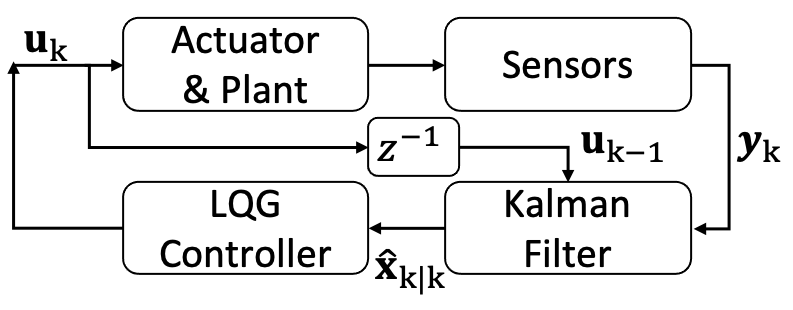}
	\caption{Schematic diagram of the system during normal operation.}
	\label{fig:normal_system}
\end{figure}
We assume a linear time-invariant MIMO plant with the following state update and measurement equations,
\begin{align}
	{\bf{x}}_{k+1}&={\bf A}{\bf{x}}_{k}+{\bf B}{\bf{u}}_{k}+{\bf{w}}_{k}, \label{eqn:states} \\ 
	{\bf{y}}_{k}&={\bf C}{\bf{x}}_{k}+{\bf{v}}_{k}, \label{eqn:output}
\end{align}
where ${\bf{x}}_{k}\in {\rm I\!R}^{n}$ and ${\bf{y}}_{k} \in {\rm I\!R}^{m}$ are the state and measurement vectors, respectively. ${\bf{u}}_{k}\in {\rm I\!R}^{p}$ is the control input vector. The process and observation noise vectors are denoted as ${\bf{w}}_{k} \in {\rm I\!R}^{n} \sim \mathcal{N}(0,{\bf Q})$ and ${\bf{v}}_{k} \in {\rm I\!R}^{m} \sim \mathcal{N}(0,{\bf R})$, respectively, with ${\bf Q} > \bf{0}$ and ${\bf R} > \bf{0}$. Here, ${\bf{A}}\in {\rm I\!R}^{n\times n}$, ${\bf{B}}\in {\rm I\!R}^{n\times p}$, ${\bf{Q}}\in {\rm I\!R}^{n\times n}$, ${\bf{C}}\in {\rm I\!R}^{m\times n}$, and ${\bf{R}}\in {\rm I\!R}^{m\times m}$. Process and observation noises are assumed to be iid and uncorrelated to each other and with the initial state vector. We also assume that the system has been operational for a very long time and is currently in steady-state.

The states of the system are estimated using the Kalman estimator. The sensor measurements are available to a remote estimator/controller, possibly over a wireless link, which may be vulnerable to malicious data deception attacks. In the absence of an attack, the time update and measurement update equations are as follows,
\begin{align}
{\hat{\bf{x}}}_{k|k-1} &={\bf A}{\hat{\bf{x}}}_{k-1|k-1}+{\bf B}{\bf{u}}_{k-1},  \label{eqn:xk_km1} \\ 
{\hat{\bf{x}}}_{k|k} &={\hat{\bf{x}}}_{k|k-1}+{\bf K}\gamma_k, \label{eqn:xk_k}
\end{align}
where ${\hat{\bf{x}}}_{k|k-1} =\text{E}[{\bf{x}}_{k}|\mathcal{I}_{k-1}]$ and ${\hat{\bf{x}}}_{k|k} =\text{E}[{\bf{x}}_{k}|\mathcal{I}_{k}]$ are the Kalman predicted and filtered states, respectively. $\text{E}[\cdot]$ denotes the expectation operator, and $\mathcal{I}_k \triangleq \left\{{\bf u}_0,{\bf u}_1, \cdots {\bf u}_k, {\bf y}_0, {\bf y}_1, \cdots {\bf y}_k \right\}$ is the set of all information up to time $k$. The innovation signal $\gamma_k$ and the steady state Kalman filter gain $\bf{K}$ are given as, 
\begin{align}
\gamma_k&={\bf y}_k-{\bf C}{\hat{\bf{x}}}_{k|k-1} \text{,} \label{eqn:gamma_normal} \\
{\bf K}&={\bf P}{\bf C}^T\left({\bf C}{\bf P}{\bf C}^T+{\bf R}\right)^{-1} \label{eqn:kalman_gain},
\end{align}
where ${\bf P}={\text E}\left[({\bf x}_k-{\hat{ \bf x}}_{k|k-1}) ({\bf x}_k-{\hat{ \bf x}}_{k|k-1})^T   \right]$ is the steady state error covariance. In steady-state ${\bf P}$ becomes the solution to the following algebraic Riccati equation,
\begin{equation}
{\bf P}={\bf A}{\bf P}{\bf A}^T+{\bf Q}-{\bf A}{\bf P}{\bf C}^T\left( {\bf C}{\bf P}{\bf C}^T +{\bf R} \right)^{-1}{\bf C}{\bf P}{\bf A}^T.
\end{equation}
 The estimated states are fed to a state feedback controller which is assumed to be an infinite horizon linear quadratic Gaussian (LQG) controller. The optimal control input ${\bf u}_k^*$ is derived by minimizing the following cost function, 
 \begin{equation}
 J=\lim_{T\to\infty}{\text E}\left[ \frac{1}{2T+1}\left\{ \sum_{k=-T}^T \left( {\bf x}^T_k{\bf W}{\bf x}_k+{\bf u}^T_k{\bf U}{\bf u}_k\right)\right\} \right] \label{eqn:lqg_cost}
 \end{equation}
 Here $\bf {W} \ge {\bf 0}$ and $\bf {U} \ge {\bf 0}$ are weight matrices. The optimal LQG control input turns out to be the following linear function of the estimated states, ${\bf u}^*_k ={\bf L}{\hat {\bf x}}_{k|k}$, where 
 \begin{equation*}
{\bf L} =-\left( {\bf B}^T{\bf S}{\bf B}+{\bf U}\right)^{-1}{\bf B}^T{\bf S}{\bf A}.
 \end{equation*}

%\begin{align}
%	&{\bf u}^*_k ={\bf L}{\hat {\bf x}}_{k|k}  \label{eqn:opt_u}, \ \\
%	&\text{where } {\bf L} =-\left( {\bf B}^T{\bf S}{\bf B}+{\bf U}\right)^{-1}{\bf B}^T{\bf S}{\bf A} \label{eqn:L}.
%\end{align}
 Here, $\bf {S}$ is the solution to the following algebraic Riccati equation,
 \begin{equation}
{\bf S}={\bf A}^T{\bf S}{\bf A}+{\bf W}-{\bf A}^T{\bf S}{\bf B}\left({\bf B}^T{\bf S}{\bf B}+{\bf U}\right)^{-1}{\bf B}^T{\bf S}{\bf A}. \label{eqn:S}
 \end{equation} 

  \subsection{Attack Model} \label{subsec:attack_model}
 We make the following assumptions regarding the capabilities and knowledge of an attacker:
 \begin{enumerate}
 	\item the attacker can access the sensor nodes and replace the true observations with fake data; 
 	\item the attacker has complete knowledge about the system and the controller, \ie, the attacker knows ${\bf A}$, ${\bf B}$, ${\bf C}$, ${\bf Q}$, ${\bf R}$, and ${\bf L}$;
 	\item the attacker can not access or alter the control signal. 
 \end{enumerate}
To launch a data deception attack, the attacker replaces the true observations ${\bf y}_k$ by the fake data ${\bf z}_k$ from $k \ge \Gamma$. A well-studied method to achieve this is to transmit the fake observations with significantly higher power than the true measurements from the sensors. As a result, the wireless control system receiver accepts the fake measurements as legitimate while rejecting the true measurements from the sensor nodes. Such attack models are also known as sensor spoofing attacks \cite{yilmaz2015survey,liu2020secure}. The fake observation data ${\bf z}_k$ is assumed to be generated from the following linear stationary stochastic process, 
\begin{align}
	{\bf z}_{k}&={\bf A}_a{\bf z}_{k-1}+{\bf w}_{a,k-1} \label{eqn:hidden_states_main},
\end{align}
where ${\bf w}_{a,k} \sim \mathcal{N}(0,{\bf Q}_a)$ is the iid noise vector at the $k$-th time instant, and ${\bf Q}_a \in {\rm I\!R}^{m \times m}$. A similar attack strategy is also studied in \cite{li2022measurement}, where the stealthiness of the attack signal is evaluated in terms of the KLD between the distributions of the fake and true observations. The attacker's system matrix $A_a$ and the noise covariance matrix $Q_a$ should be designed in such a way so that the fake data $z_k$ mimics the statistical properties of the true measurement $y_k$. $A_a$ mainly models the correlations between the current and past measurements, and $Q_a$ models the uncertainty. Designing $A_a$ and $Q_a$ in such a way increases the stealthiness of the attack signal. A schematic diagram of the system under the data deception attack is shown in Fig.~\ref{fig:attack_model}.
 \begin{figure}[h!]
	\centering
	\includegraphics[width=60mm]{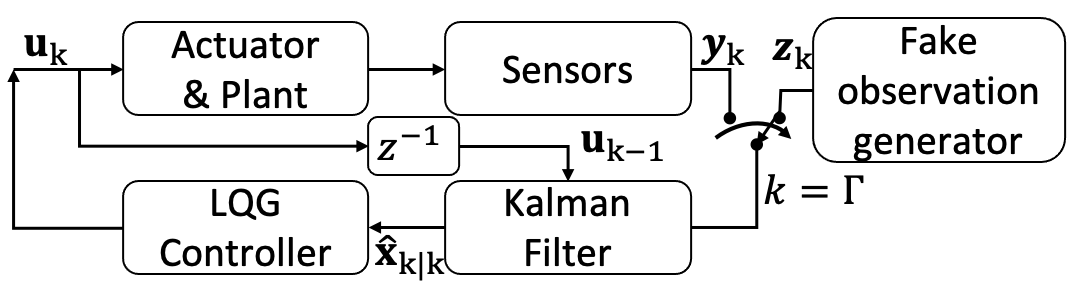}
	\caption{Schematic diagram of the system under attack.}
	\label{fig:attack_model}
\end{figure}
\begin{remark}
In general, for linear control systems, the measurement vector ${\bf y}_k$ can be modelled as a stochastic process that is dependent on its past values with an additive i.i.d noise component, \ie, similar to (\ref{eqn:hidden_states_main}). In other words, the attack model in (\ref{eqn:hidden_states_main}) mimics the linear stationary stochastic model of ${\bf y}_k$, whic­h makes it challenging to detect. In addition to that, the attack model in (\ref{eqn:hidden_states_main}) can make closed-loop control systems unstable, as discussed in \cite{watermarking_tac}, which may cause a significant amount of damage to CPS. Since the attacker's objectives are to cause damage to the CPS and to remain stealthy in doing so, the attack model considered in this paper (\ref{eqn:hidden_states_main}) facilitates the attacker in achieving both the objectives, thus illustrating the significance of such an attack model. Moreover, such an attack model can be used to detect replay attacks after a few modifications, as demonstrated in \cite{reduced_watermarking_ACC}.
\end{remark}

During the attack, \ie, for $k \ge \Gamma$, the innovation signal will take the following form, 
 \begin{align}
{\bf \gamma}_k &= {\bf z}_{k} - {\bf C}{\hat{\bf{x}}}_{k|k-1},
\label{eqn:gamma_attack} 
 \end{align} 
whereas ${\hat{\bf{x}}}_{k|k-1}$ and ${\hat{\bf{x}}}_{k|k}$ will follow the same time update and measurement update equations (\ref{eqn:xk_km1}) and (\ref{eqn:xk_k}), respectively. However, the $\gamma_k$ in (\ref{eqn:xk_k}) for $k \ge \Gamma$ will follow (\ref{eqn:gamma_attack}). Note that, after the attack start point, \ie, $k \ge \Gamma$, the defender does not change the Kalman filter. However, since the attacker replaces $\bf y_k$ by the fake data $\bf z_k$ from $k \ge \Gamma$, the innovation signal $\gamma_k$ automatically takes the form given in (\ref{eqn:gamma_attack}). 

As studied in several works of literature, the distribution of the attack start point can be modelled as exponential distribution for continuous-time systems \cite{jonsson1997quantitative,arnold2014time}. In \cite{jonsson1997quantitative}, the authors collected empirical data from intrusion experiments and found that the attack start points are approximately exponentially distributed. On the other hand, the authors formalized the semantics of attack trees and used them for the probabilistic timed evaluation of attack scenarios in \cite{arnold2014time}. Additionally, the authors studied various practical systems, including the famous Stuxnet attack  \cite{Langner2011}, and derived the distribution of the attack start time to be exponential. Since the exponential and geometric distributions play analogous roles in the continuous and discrete time domains, respectively \cite{prochaska1973note}, we have modelled the attack start point $\Gamma$ to be an RV that follows a geometric distribution with parameter $\rho$, where $0< \rho < 1$. Here $\rho$ is a design parameter reflecting the defender's belief of how often attacks occur. For the proposed method, $\rho$ is a parameter that needs to be set by the defender. A vulnerability analysis of the system can decide the value of $\rho$, see \cite{jonsson1997quantitative,arnold2014time}. From the derived approximate analytical expressions of ADD (\ref{eqn:ADD_2nd}), FAR (\ref{eqn:FAR}) and $\Delta LQG$ (\ref{eqn:DeltaLQG}), we can say that ADD and $\Delta LQG$ will not be affected much by the difference in the chosen $\rho$ and the attacker's true $\rho$,  as long as both $\rho \ll 1$, which is a realistic assumption since attacks are rare events. On the other hand, FAR will increase if the chosen $\rho$ is higher than the attacker's true $\rho$ and vice versa. The defender can thus choose a suitable $\rho$, depending on the specification on the maximum false alarm rate. Finally, we can write the prior probability $\Pi_k \triangleq \text{P}\left\{\Gamma = k\right\}$ in the following form \cite{Banerjee2012}, 
\begin{align}
\Pi_k = \Pi_0\mathbbm{1}_{\left\{k=0\right\}} +\left(1-\Pi_0\right)\rho\left(1-\rho\right)^{k-1}\mathbbm{1}_{\left\{k\ge1\right\}}.
\label{eqn:Pi}
\end{align}
Here, $\Pi_0 \triangleq \text{P}\left\{\Gamma \le 0\right\}$, \ie, $\Pi_0$ is the prior probability of the attack happening before the start of the observation. $\mathbbm{1}_{\left\{condition \right\}}$ is the indicator function, which takes the value 1 if the \textit{condition} is true, or 0 otherwise. In general, $0 \le \Pi_0 < 1$. However, for our problem formulation, we have taken $\Pi_0 = 0$. We assume that the defender does not know the exact value of $\Gamma$, but knows the prior distribution of $\Gamma$.  
\section{Proposed detection strategy}\label{sec:def_strategy}
We perform the following hypothesis test to detect the presence of an attack, 
\begin{description}
	\item[$H_0$:]No attack present
	\item[$H_1$:]Attack present in the system
\end{description}
We parsimoniously add an iid watermarking signal, given as 
\begin{equation}
 {\bf e}_k \sim \mathcal{N}({\bf 0},{\bf \Sigma}_e), \label{eqn:ek}
 \end{equation}
to the optimal LQG control input, ${\bf u}^*_k$,  to improve the detectability of the attack, see (\ref{eqn:add_ek_sk}). 
The decision of adding or not adding the watermarking and the selection of hypothesis for the $k$-th time instant is controlled by the optimal policy ${\bf{u}}_d^*$. Here, the subscript $d$ of ${\bf{u}}_d^*$ indicates that the optimal policy is derived using dynamic programming. The policy ${\bf{u}}_{d,k}$ decides the values of the following two control variables $s_k$ and $d_k$ at the $k$-th time instant, 
\begin{equation}
	s_k=\begin{cases}0, & \text{No watermarking for } (k+1) \text{-th time instant.}\\1, & \text{Watermarking added for } (k+1) \text{-th time instant.}\end{cases}
	\label{eqn:sk}
\end{equation}
\begin{equation}
	d_k=\begin{cases}0, & \text{Hypothesis } H_0 \text{ selected, process continues.}\\1, & \text{Hypothesis } H_1 \text{ selected, process terminated.}\end{cases}
	\label{eqn:d_k}
\end{equation}
\begin{equation}
	{\bf u}_k = {\bf u}^*_k+s_{k-1}{\bf e}_k.
	\label{eqn:add_ek_sk}
\end{equation}
Figure~\ref{fig:reduced_watermarking} illustrates the proposed watermarking and attack detection scheme with a schematic diagram of the system. The components enclosed inside the blue dotted rectangle are assumed to be located at a secure location.  
 \begin{figure}[h!]
	\centering
	\includegraphics[width=80mm]{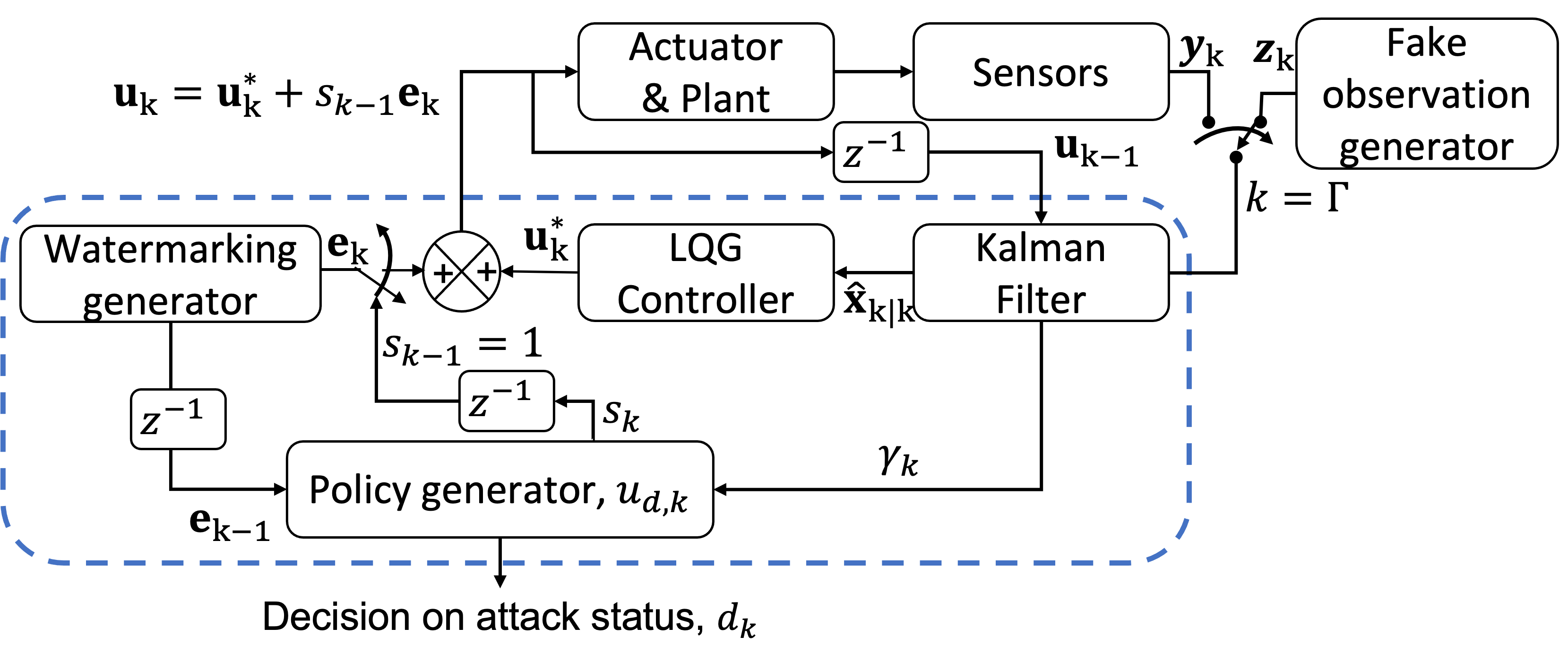}
	\caption{Schematic diagram of the system with the proposed watermarking scheme.}
	\label{fig:reduced_watermarking}
\end{figure}

\subsection{Problem formulation} \label{subsec:problem_formulation}
Our objective is to find the optimal policy ${\bf u}_d^*$ that minimizes the ADD for fixed upper bounds on FAR and ANW. First, we introduce the formal definitions of FAR, ADD and ANW as follows. The definitions of FAR and ADD are similar to \cite{Tartakovsky2005}.  \\
\textbf{False alarm rate (FAR)}: FAR is defined as
\begin{equation}
	FAR \triangleq \text{P}^{\Pi}\left\{\tau < \Gamma\right\}.
	\label{eqn:FAR_def}
\end{equation}
Here, $\text{P}^{\Pi}$ indicates the average probability measure. $\text{P}^{\Pi}\left\{\Omega\right\}= \sum_{k=1}^\infty{\Pi}_k\text{P}_k\left\{\Omega\right\}$, where $\Omega$ is any event and $\text{P}_k$ is the probability measure when the change point $\Gamma = k$. $\tau$ is the time instant when the hypothesis $H_1$ is selected.  \\
\textbf{Average detection delay (ADD)}: ADD is defined as
\begin{equation}
	ADD \triangleq \text{E}^{\Pi}\left[\tau-\Gamma | \tau \ge \Gamma\right].
	\label{eqn:ADD_def}
\end{equation}
Here, $\text{E}^{\Pi}$ denotes the expectation with respect to the probability measure $\text{P}^{\Pi}$. \\
\textbf{Average number of watermarking events (ANW) before attack}: ANW is defined as
\begin{equation}
	ANW \triangleq \text{E}^{\Pi}\left[\sum_{i=1}^{\text{min}\left(\tau,\Gamma-1\right)}s_i \right].
	\label{eqn:ANE_def}
\end{equation}
Here, $s_i$ is the same variable as given in (\ref{eqn:sk}). \\
Now, we formulate the following optimization problem, 
\begin{equation}
	\begin{aligned}
		&\min_{{\bf u}_d} \  ADD, \\
		&\textrm{s.t.}\   FAR \le FAR_{th}, \\
		&ANW \le ANW_{th},
	\end{aligned}
	\label{eqn:cost_fun}
\end{equation}
where $FAR_{th}$ and $ANW_{th}$ are the user-selected thresholds. Then, the constrained optimization problem of (\ref{eqn:cost_fun}) is converted into an unconstrained Lagrangian form as follows.  The unconstrained Lagrangian form adopted in this paper is similar to \cite{Banerjee2012,beutler1985optimal}, except the term $\lambda_eANW$, and it reads as 
\begin{equation}
	J^*=\min_{{\bf u}_d} \ ADD+\lambda_fFAR+\lambda_eANW,
	\label{eqn:J_opt}
\end{equation}
where $\lambda_f \ge 0 $ and $\lambda_e \ge 0$ are the Lagrangian multipliers. A new state variable $\theta_k$ is defined as
%Similar to [ref], we also assume that the optimization problem of (\ref{eqn:J_opt}) can be solved for some values of $\lambda_f$ and $\lambda_e$, such that the constraints of (\ref{eqn:cost_fun}) will be satisfied with equality.
\begin{equation}
	\theta_k \triangleq \begin{cases}0, & \text{No attack, } \\1, & \text{System under attack, } \\ \text{T}_e, & \text{Attack detected by hypothesis testing} \\  & \text{and process terminated.}\end{cases}
	\label{eqn:theta_k}
\end{equation}
Similar to \cite{Banerjee2012}, ADD, FAR and ANW can also be expressed in terms of the control variables, $s_k$ and $d_k$, and the state variable $\theta_k$ as follows,
\begin{align}
	ADD&=\text{E}\left[\sum_{k=1}^\tau \mathbbm{1}_{\left\{\theta_k=1\right\}}\mathbbm{1}_{\left\{d_k=0\right\}}\right], \label{eqn:ADD_pk} \\
	FAR &= \text{E}\left[\sum_{k=1}^\tau \mathbbm{1}_{\left\{\theta_k=0\right\}}\mathbbm{1}_{\left\{d_k=1\right\}}\right], \text{ and} \label{eqn:FAR_pk} \\
	ANW &= \text{E}\left[\sum_{k=1}^\tau \mathbbm{1}_{\left\{\theta_k=0\right\}}\mathbbm{1}_{\left\{s_k=1\right\}} \mathbbm{1}_{\left\{d_k=0\right\}}\right]. \label{eqn:ANW_pk} 
\end{align}
Using (\ref{eqn:ADD_pk})-(\ref{eqn:ANW_pk}), the cost function of (\ref{eqn:J_opt}) can be expressed as 
\begin{equation}
	J^*=\min_{{\bf u}_d}\ \text{E}\left[\sum_{k=1}^\tau g_k\left(\theta_k,s_k,d_k\right)\right],
	\label{eqn:J_opt_thetak}
\end{equation}
where $g_k(\cdot)$ is the per stage cost, expressed as
\begin{equation}
	\begin{aligned}
		&g_k\left(\theta_k,s_k,d_k\right)=\mathbbm{1}_{\left\{\theta_k \ne \text{T}_e\right\}}\left[\mathbbm{1}_{\left\{\theta_k = 1\right\}}\mathbbm{1}_{\left\{d_k =0 \right\}} \right.  \\
		& \left. + \lambda_f\mathbbm{1}_{\left\{\theta_k =0 \right\}}\mathbbm{1}_{\left\{d_k =1 \right\}}+\lambda_e\mathbbm{1}_{\left\{\theta_k =0 \right\}}\mathbbm{1}_{\left\{s_k =1 \right\}}\mathbbm{1}_{\left\{d_k =0 \right\}}\right].
	\end{aligned}
	\label{eqn:gk}
\end{equation}
Here, the first, second and third terms of (\ref{eqn:gk}) come from ADD, FAR and ANW, respectively. For the stochastic optimal control problem defined in (\ref{eqn:J_opt_thetak}), the state  $\theta_k$ is not observable to the defender. Therefore, we replace the state  $\theta_k$ by it's sufficient statistics $p_k$. The sufficient statistics $p_k$, \ie, the posterior probability of attack at $k$-th time instant is defined as, $p_k \triangleq \text{P}\left\{\Gamma\le k|\Psi_k\right\}=\text{E}\left[\mathbbm{1}_{\left\{\theta_k = 1\right\}}|\Psi_k\right]$, where $\Psi_k$ |. The optimization problem in (\ref{eqn:J_opt_thetak}) is then redefined and solved using $p_k$ as discussed in details in the following Sub-section~\ref{subsec:dymanic_programming}.

The accessibility hypothesis discussed in \cite{beutler1985optimal} tells us that under this hypothesis, for every stationary deterministic policy ${\bf u}_{d} \in \mathcal{U}$, any arbitrary state, say $\theta_k$ is accessible from each starting state $\theta_k = \theta_0$. Here, $\mathcal{U}$ is the set of all permissible stationary deterministic policies. Under the accessibility hypothesis, the dynamic programming equation using the cost function of (\ref{eqn:J_opt_thetak}) is solvable by at least one stationary deterministic policy for each $\lambda_f \ge 0$ and $\lambda_e \ge 0$ \cite{beutler1985optimal}. 
%To satisfy the accessibility hypothesis, we have discretized the space of sufficient statistics $p_k$ into a finite set during the numerical simulations. On the other hand, the control space of the stochastic optimization problem under study is inherently discrete and finite. }

\subsection{Finding the optimal policy}
\label{subsec:dymanic_programming}
This section discusses the solution approach taken to solve the optimization problem of (\ref{eqn:J_opt_thetak}) in the following three main steps. 
\subsubsection{Selection of test signals}
Combining (\ref{eqn:states})-(\ref{eqn:gamma_normal}) and (\ref{eqn:gamma_attack}), we can represent the innovation signal as
\begin{align}
	&\text{for } k<\Gamma, \nonumber \\
	&\gamma_k ={\bf C}{\bf A}\left({\bf x}_{k-1}-{\hat {\bf x}}_{k-1|k-1} \right)+{\bf C}{\bf w}_{k-1}+{\bf v}_k \text{, and} \label{eqn:gamma_vs_e}\ \\
	&\text{for } k\ge \Gamma,  \nonumber \\
	&\gamma_k= {\bf z}_k-{\bf C}\left({\bf A}+{\bf B}{\bf L} \right){\hat {\bf x}}_{k-1|k-1}-{\bf C}{\bf B}s_{k-2}{\bf e}_{k-1} \label{eqn:gamma_attack_vs_e}.
\end{align}
So, the innovation signal is dependent on the watermarking signal after the attack, see (\ref{eqn:gamma_attack_vs_e}). On the contrary, the innovation signal is independent of the watermarking signal before the attack, see (\ref{eqn:gamma_vs_e}). It is assumed that the attacker will be replacing the true stationary observation ${\bf y}_k$ with fake but stationary data ${\bf z}_k$ to remain stealthy. In addition to that, as discussed in Sub-section~\ref{subsec:struc_opt_soln}, the optimal policy ${\bf u}^*_d$ is also a stationary one. Therefore, the innovation signal will be stationary but with different distributions before and after the attack, as $k \rightarrow \infty$. Also, from the properties of the Kalman filter, we know the innovation signal is iid before the attack. Additionally, the use of the joint statistics of the innovation signal and the watermarking signal increases the KLD compared to the case where the statistics of the innovation signal alone is used, see Theorem 1 and Remark 1 from \cite{Salimi2019}, where the improvement in KLD has been quantified for a single-input single-output (SISO) system. These reasons motivate us to use the joint distribution of the innovation signal and the watermarking signal to generate the test statistics for attack detections.

\subsubsection{Derivation of test statistics}
We use the Shiryaev statistics because of its asymptotic optimality for a fixed upper bound on FAR as stated in Theorem 1 from \cite{Tartakovsky2017}. The data is assumed to be iid before and after the change point in \cite{Tartakovsky2017}. In contrast to \cite{Tartakovsky2017}, in our study, the innovation signal $\gamma_k$ is iid before the attack and non-iid after the attack. From \cite{Tartakovsky2014}, the Shiryaev statistics $SR_k$ at the $k$-th instant in time for our problem formulation can be written as
	\begin{equation}
	SR_k=\sum_{i=1}^k\prod_{j=i}^k\frac{\mathcal{L}_j}{1-\rho},
	\label{eqn:SRn_1st}
\end{equation}
where $i$ is the candidate change or attack start point, and ${\mathcal{L}_j}$ is the likelihood ratio. The expression for ${\mathcal{L}_j}$ is given in Lemma~\ref{lemma:SRn}. After the change point, $SR_k$ increases on average. In the original Shiryaev procedure, a change is detected once $SR_k$ crosses a predefined threshold for the first time.
\begin{lemma} \label{lemma:SRn}
	The likelihood ratio ${\mathcal{L}_j}$ used to derive the Shiryaev statistics (\ref{eqn:SRn_1st}) considering the joint distribution of the innovation signals ((\ref{eqn:gamma_vs_e}) and (\ref{eqn:gamma_attack_vs_e})) and the watermarking signal (\ref{eqn:ek}), takes the following form,
\begin{equation}
	\mathcal{L}_j=\begin{cases}\mathcal{L}_{a,j}, & j >i \\\mathcal{L}_{b,j}, & j =i\end{cases}, \textit{ and}
	\label{eqn:Lj_2nd}
\end{equation}
\begin{align}
	\mathcal{L}_{a,j}&=\frac{\mathpzc{f}_{1,j}\left(\gamma_j|\left\{\gamma\right\}_1^{j-1},\left\{{\bf e}_s\right\}_1^{j-1} \right)}{\mathpzc{f}_0\left(\gamma_j\right)}, \label{eqn:Laj} \\
	\mathcal{L}_{b,j}&=\frac{\mathpzc{f}_{2,j}\left(\gamma_j|\left\{\gamma\right\}_1^{j-1},\left\{{\bf e}_s\right\}_1^{j-1} \right)}{\mathpzc{f}_0\left(\gamma_j\right)}. \label{eqn:Lbj}
\end{align}
Here, $\mathpzc{f}_{1,j}\left(\cdot|\cdot \right)$, $\mathpzc{f}_{2,j}\left(\cdot|\cdot \right)$ and $\mathpzc{f}_{0}\left(\cdot \right)$ denote the distributions for $j>i$, $j=i$ and $j<i$, respectively, and ${\bf e}_{s,j} = s_{j-1}{\bf e}_j$. $\mathpzc{f}_{1,j}\left(\cdot|\cdot \right)$, $\mathpzc{f}_{2,j}\left(\cdot|\cdot \right)$ and $\mathpzc{f}_{0}\left(\cdot \right)$ take the following forms,
\begin{align}
	\mathpzc{f}_{1,j}\left(\cdot|\cdot \right) &= \mathcal{N}\left( \mu_{1,j},\Sigma_{1,j}|{\bf z}_{j-1},{\hat {\bf x}}_{j-1|j-1},{\bf e}_{s,j-1}\right), \label{eqn:f1j}\ \\
	\mathpzc{f}_{2,j}\left(\cdot|\cdot \right)&=\mathcal{N}\left( \mu_{2,j},\Sigma_{2,j}|{\hat {\bf x}}_{j-1|j-1},{\bf e}_{s,j-1}\right), \label{eqn:f2j} \ \\
	\mathpzc{f}_{0}\left(\cdot \right) &= \mathcal{N}\left( {\bf 0},\Sigma_{0}\right). \label{eqn:f0j}
\end{align}
Here
\begin{align}
	&\mu_{1,j} ={\bf A}_a{\bf z}_{j-1}-{\bf C}\left({\bf A}+ {\bf B}{\bf L} \right){\bf {\hat x}}_{j-1|j-1} - {\bf C}{\bf B}{\bf e}_{s,j-1} \label{eqn:mu_1j} \ \\
		&\mu_{2,j} = -{\bf C}\left({\bf A}+ {\bf B}{\bf L} \right){\bf {\hat x}}_{j-1|j-1} - {\bf C}{\bf B}{\bf e}_{s,j-1} \label{eqn:mu_2j} \ \\
		&\Sigma_{1,j} ={\bf Q}_a \label{eqn:sigma_1j} \ \\
		&\Sigma_{2,j} = {\bf Q}_z \label{eqn:sigma_2j} \ \\
		&\Sigma_{0} = {\bf C}{\bf P}{\bf C}^T+{\bf R} \label{eqn:sigma_0j}
\end{align} 
\end{lemma}
\begin{proof}[Proof]
	The proof of Lemma~\ref{lemma:SRn} is provided in Appendix~\ref{apdx:SRn}.
\end{proof}
Equation (\ref{eqn:SRn_1st}) is same as the original Shiryaev statistics, where $\Pi_0$ is assumed to be 0, see (6.9) from \cite{Tartakovsky2014}. However, the term ${\mathcal{L}_j}$ in (\ref{eqn:SRn_1st}) is derived exclusively for the problem under study, where the test data is iid before the change point and non-iid with stationary distributions after the change point. Furthermore, Lemma 1 shows that the dependency of the test data $\gamma_j$ on the previous values of $\gamma$ and ${\bf e}_{s}$ from the time index $1$ to $j-1$ can be approximated as given in (\ref{eqn:Laj})-(\ref{eqn:mu_2j}), where $\gamma_j$ is only dependent on the immediate past values, \ie, at time index $j-1$, of $\bf z$, $\hat {\bf x}$ and $\bf e_s$. 
\begin{remark}
The likelihood ratios using the distributions of the innovation signal alone, say, $\mathcal{L}_{c,j}$ and $\mathcal{L}_{d,j}$ for $j <i$ and $j = i$, respectively, can be evaluated from Lemma~\ref{lemma:SRn} by using ${\bf e}_{s,j-1}=0$ in (\ref{eqn:f1j})-(\ref{eqn:f2j}). Therefore, we can write $\mathcal{L}_{c,j} = \mathcal{L}_{a,j}\mid_{e_{s,j-1}=0}$ and $\mathcal{L}_{d,j} = \mathcal{L}_{b,j}\mid_{e_{s,j-1}=0}$.
\end{remark}
We have applied the value iteration from \cite{bertsekas1995dynamic} using sufficient statistics $p_k$ to solve (\ref{eqn:J_opt_thetak}), which is an infinite horizon dynamic programming problem with a termination state. The relationship between the Shiryaev statistics $SR_k$ and the posterior probability of attack $p_k$ is given by, see (6.10) from \cite{Tartakovsky2014}, 
\begin{equation}
	p_k=\frac{SR_k}{SR_k+1/\rho}.
	\label{eqn:pk_1st}
\end{equation}
Lemma~\ref{lemma:pk_recur} provides the recursion formula of $p_k$, which is used for the value iteration.
\begin{lemma} \label{lemma:pk_recur}
The posterior probability of attack at the $k$-th time instant, $p_k$, for the test data $\gamma_k$ (iid (\ref{eqn:gamma_vs_e}) and non-iid (\ref{eqn:gamma_attack_vs_e})) and ${\bf e}_k$ (\ref{eqn:ek}), can be updated using the following recursion formula, when the attack start point is geometrically distributed with parameter $\rho$,  
	\begin{align}
		&p_k=\frac{p_{k-1}{\mathcal L}_{c,k}+\left(1-p_{k-1}\right)\rho{\mathcal L}_{d,k}}{\left(1-\rho\right)\left(1-p_{k-1}\right)+p_{k-1}{\mathcal L}_{c,k}+\left(1-p_{k-1}\right)\rho{\mathcal L}_{d,k}}, \nonumber \\
		& \text{when }s_{k-2} = 0 \text{, and} \nonumber \\ 
		&p_k=\frac{p_{k-1}{\mathcal L}_{a,k}+\left(1-p_{k-1}\right)\rho{\mathcal L}_{b,k}}{\left(1-\rho\right)\left(1-p_{k-1}\right)+p_{k-1}{\mathcal L}_{a,k}+\left(1-p_{k-1}\right)\rho{\mathcal L}_{b,k}},  \nonumber \\
		& \text{otherwise}.
		\label{eqn:pk_recurr}
	\end{align}
\end{lemma} 
\begin{proof}[Proof]
	Using (\ref{eqn:SRn_1st}), the recursion formula of the Shiryaev statistics $SR_k$ is derived first, see (\ref{eqn:SRn_recurr}). Then using (\ref{eqn:pk_1st}) in (\ref{eqn:SRn_recurr}), the recursion equations of (\ref{eqn:pk_recurr}) are derived.
	\begin{equation}
		SR_k =\begin{cases}\frac{{\mathcal L}_{c,k}}{1-\rho}SR_{k-1}+\frac{{\mathcal L}_{d,k}}{1-\rho}, & s_{k-2} = 0,\\\frac{{\mathcal L}_{a,k}}{1-\rho}SR_{k-1}+\frac{{\mathcal L}_{b,k}}{1-\rho}, & s_{k-2}  = 1.\end{cases}
		\label{eqn:SRn_recurr}
	\end{equation}
\end{proof}
We use the following simplified notations to represent the recursion formula in (\ref{eqn:pk_recurr}). $p_k = \phi_0\left( p_{k-1}\right)$, if $s_{k-2} = 0$, and $p_k = \phi_1\left(p_{k-1} \right)$, otherwise. Also, the initial value of $p_k$ is taken to be 0.
\subsubsection{Solution of optimization problem (\ref{eqn:J_opt_thetak})} \label{subsub:bellman_eq}
The expected value of the per stage cost function $g_k(\cdot)$ in (\ref{eqn:J_opt_thetak}) is derived by taking expectations on both sides of (\ref{eqn:gk}), and using $p_k=\text{E}\left[\mathbbm{1}_{\left\{\theta_k = 1\right\}}|{\Psi}_k\right]$, see (\ref{eqn:E_gk}).
\begin{equation}
	\begin{aligned}
&\text{E}\left[g_k\left(\theta_k,s_k,d_k\right)|{\Psi}_k\right]=p_k\mathbbm{1}_{\left\{d_k =0 \right\}} + \\ & \lambda_f\left(1-p_k\right)\mathbbm{1}_{\left\{d_k =1 \right\}} +\lambda_e\left(1-p_k\right)\mathbbm{1}_{\left\{s_k =1 \right\}}\mathbbm{1}_{\left\{d_k =0 \right\}}.
\end{aligned}
\label{eqn:E_gk}
\end{equation}
The Bellman equation for the infinite horizon cost function (\ref{eqn:J_opt_thetak}) with the termination state $\text{T}_e$ can be formulated using the sufficient statistics $p_k$ as follows,
	\begin{align}
		&J\left(p_k\right) =\min_{{\bf u}_d\subset\mathcal{U}}\left[p_k\mathbbm{1}_{\left\{d_k =0 \right\}} + \lambda_f\left(1-p_k\right)\mathbbm{1}_{\left\{d_k =1 \right\}} \right. \nonumber \\
		& \left. +\lambda_e\left(1-p_k\right)\mathbbm{1}_{\left\{s_k =1 \right\}}\mathbbm{1}_{\left\{d_k =0 \right\}} + B_0\left(p_k\right)\mathbbm{1}_{\left\{s_k =0 \right\}}\mathbbm{1}_{\left\{d_k =0 \right\}} \right. \nonumber \\
		& \left. + B_1\left(p_k\right)\mathbbm{1}_{\left\{s_k =1 \right\}}\mathbbm{1}_{\left\{d_k =0 \right\}}\right],
		\label{eqn:bellman_eqn}
	\end{align}
where $\mathcal{U} = \left\{1,2,3 \right\}$ is the set of all stationary deterministic permissible policies, see Table~\ref{tab:U}. $B_0\left(p_k\right)=\text{E}\left[J\left(\phi_0\left(p_{k} \right)\right)\right]$, and $B_1\left(p_k\right)=\text{E}\left[J\left(\phi_1\left(p_{k} \right)\right)\right]$. {\color{black} Therefore, $B_0\left(p_k\right)$ and $B_1\left(p_k\right)$ denote the expected total costs from $(k+1)$-th time instant till the termination of the process when $s_k=0$ and $s_k =1$, respectively, and $d_k = 0$. Note that, when $d_k = 1$, the process terminates immediately, so there will be no additional cost after the $k$-th time instant.} In addition to that,  if $d_k = 1$ then the process will immediately terminate, and there will be no use of adding watermarking at the $(k+1)$-th time instant, therefore, the combination $(s_k = 1, d_k =1)$ has been ignored. \vspace{2mm})
\begin{table}[h!]
\begin{center}
\caption{$\mathcal{U}$}
\label{tab:U}
\begin{tabular}{c|c|c} 
			\hline \hline
			${\bf u}_{d,k}$ & $s_k$ & $d_k$ \\
			\hline
	         1 & 0 & 0 \\
	         2& 1 & 0 \\
	         3 & 0 & 1 \\
	         	\hline \hline
\end{tabular}
\end{center}
\end{table}

To satisfy the accessibility hypothesis, we have discretized the space of the sufficient statistics $p_k$ into a finite set during the numerical simulations. On the other hand, the control space of the stochastic optimization problem under study is inherently discrete and finite. Finally, the value iteration is used to solve the Bellman equation (\ref{eqn:bellman_eqn}) and to find the optimal policy ${\bf u}_d^*$.
\subsection{Structural properties of the optimal policy} \label{subsec:struc_opt_soln} {\color{black} 
In this subsection, we will study the structure of the optimal solution found by solving the Bellman equation (\ref{eqn:bellman_eqn}). \newline
\textbf{Assumption A.1: } There exist at least one stationary deterministic policy ${\bf u}_d$, for which both the constraints as given in (\ref{eqn:cost_fun}) will be satisfied.

Assumption A.1 is about the feasibility of the existence of a stationary deterministic policy for the optimization problem in (\ref{eqn:cost_fun}). Now the value iteration reveals the following optimal policy for selecting the control variables $s_k$ and $d_k$ values. 
\begin{equation}
	s_k = \begin{cases}0, & B_0\left(p_k\right)-B_1\left(p_k\right)<\lambda_e\left(1-p_k\right), \\1, & otherwise.\end{cases}
	\label{eqn:ths}
\end{equation} {\color{black}
\begin{equation}
	d_k = \begin{cases}0, & p_k+\lambda_e\left(1-p_k\right)\mathbbm{1}_{\left\{s_k =1 \right\}}+B_0\left(p_k\right)\mathbbm{1}_{\left\{s_k =0 \right\}} \\\ &+B_1\left(p_k\right)\mathbbm{1}_{\left\{s_k =1 \right\}}<\lambda_f\left(1-p_k\right) , \\1, & \text{otherwise}.\end{cases}
	\label{eqn:thd}
\end{equation}}

First, we will prove that the optimal policy is going to be a stationary deterministic policy. As stated in Lemma 3.1 from \cite{beutler1985optimal}, the costs FAR and ANW will be monotone and non-increasing in $\lambda_f$ and $\lambda_e$, respectively.  We can prove that by following the similar steps used to prove Lemma 3.1 in \cite{beutler1985optimal}. From the monotone and non-increasing properties of FAR and ANW, it can be proved that the inequality conditions in the original constrained optimization problem (\ref{eqn:cost_fun}) will be satisfied for finite values of $\lambda_f \ge 0$ and $\lambda_e \ge 0$ for some deterministic policy as stated in Lemma 3.3 from \cite{beutler1985optimal}. 

Finally, as discussed in \cite{beutler1985optimal}, under assumption A.1 or the weaker condition of Lemma 3.1, the stationary deterministic optimal policy found by solving the Bellman equation (\ref{eqn:bellman_eqn}) from the unconstrained optimization problem with the Lagrangian multipliers, $\lambda_e$ and $\lambda_f$, will be the solution of the original constrained problem as given in (\ref{eqn:cost_fun}).

Even though a formal proof is unavailable at this point, we have performed extensive numerical simulations and found that for the following properties of the optimal policy: The optimal policy is a two threshold policy, $Th^s$ and $Th^d$, $Th^d \ge Th^s$. Figures~\ref{fig:lsh_rhs_sk} and \ref{fig:lsh_rhs_dk} provide the insights with thresholds of a two-threshold policy by plotting the left-hand sides (LHS) and right-hand sides (RHS) of (\ref{eqn:ths}) and (\ref{eqn:thd}), respectively, for a relatively small $\lambda_e$ and large $\lambda_f$. We observe that for (\ref{eqn:ths}), the LHS crosses the RHS at two points, but the second crossing happens after $d_k =1 $, \ie, the termination of the process, which results in a two-threshold policy. We have found that for a relatively large $\lambda_e$ or $\rho$ near to unity, the optimal policy may even become a one or three-threshold policy, which is similar to the findings of \cite{Banerjee2012}. The following steps can be followed offline to find the two thresholds. 
\begin{description}
	\item[Step 1]: The search space of $\lambda_e$ and $\lambda_f$ is divided into $N$ equally spaced grid points. 
	\item[Step 2]: For each grid point, we perform the value iterations using (\ref{eqn:bellman_eqn}), and store $J^*(p_k)$ where $p_k$ is also discretized in $ [0,1]$.
	\item[Step 3]: For each grid point, ADD, FAR and ANW are evaluated from Monte-Carlo simulations by deriving the decision variables $s_k$ and $d_k$ from (\ref{eqn:ths})-(\ref{eqn:thd}) using $J^*(p_k)$ from Step 2. 
	\item[Step 4]: Select the best $\lambda^*_e$ and $\lambda^*_f$ combination, which gives minimum ADD and satisfies the constraints on FAR and ANW, see (\ref{eqn:cost_fun}). 
	\item[Step 5]: Apply numerical solvers such as the Trust-Region algorithm, the bisection method, etc., to solve the following two equations for $\bar p$, see (\ref{eqn:B0_Th_s}) and (\ref{eqn:B0_Th_d}). The solutions of (\ref{eqn:B0_Th_s}) are $\bar p = Th^s$ and $\bar p = 1$. Equation (\ref{eqn:B0_Th_d}) is derived from (\ref{eqn:thd}) using $s_k = 0$ and solved for $\bar p \in [Th^s, 1]$. The solutions of (\ref{eqn:B0_Th_d}) is $\bar p = Th^d$ and $\bar p = 1$.  
	\begin{align}
		& B_0\left(\bar p\right)-B_1\left(\bar p\right)=\lambda_e\left(1-\bar p\right) \label{eqn:B0_Th_s} \\
		& \bar p+\lambda_e\left(1-\bar p\right)+B_1\left(\bar p\right)=\lambda_f\left(1-\bar p\right) \label{eqn:B0_Th_d}
	\end{align}
	Finally, the optimal policy ${\bf u}^*_d$ is given as
	\begin{equation}
		{\bf u}_{d,k}^*=\begin{cases}1 \text{, \ie, }(s_k=0, d_k = 0) & p_k < Th^s,  \\2 \text{, \ie, } (s_k=1, d_k = 0) & p_k \ge Th^s,  \\3 \text{, \ie, } (s_k=0, d_k = 1) & p_k \ge Th^d.  \end{cases}
		\label{eqn:ud_last}
	\end{equation}
\end{description}

\begin{figure}[h!]
	\centering
	\includegraphics[width=\figwidth]{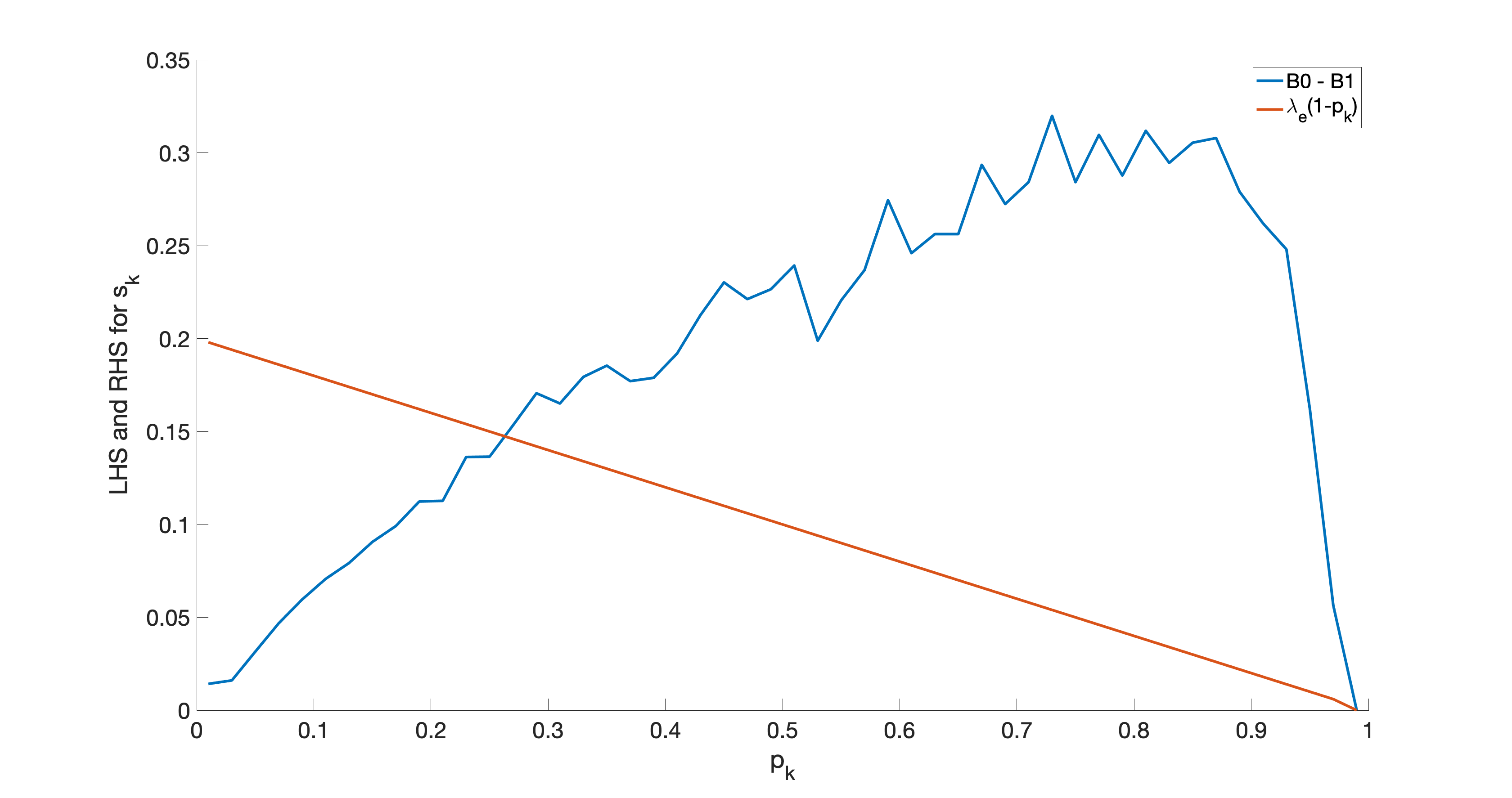}
	\caption{LHS and RHS of (\ref{eqn:ths}) vs. $p_k$ for System-A. $\lambda_e = 0.2$, $\lambda_f = 100$, and $\sigma_{e}^2 = 1.19$. \vspace{5mm}}
	\label{fig:lsh_rhs_sk}
\end{figure}

\begin{figure}[h!]
	\centering
	\includegraphics[width=\figwidth]{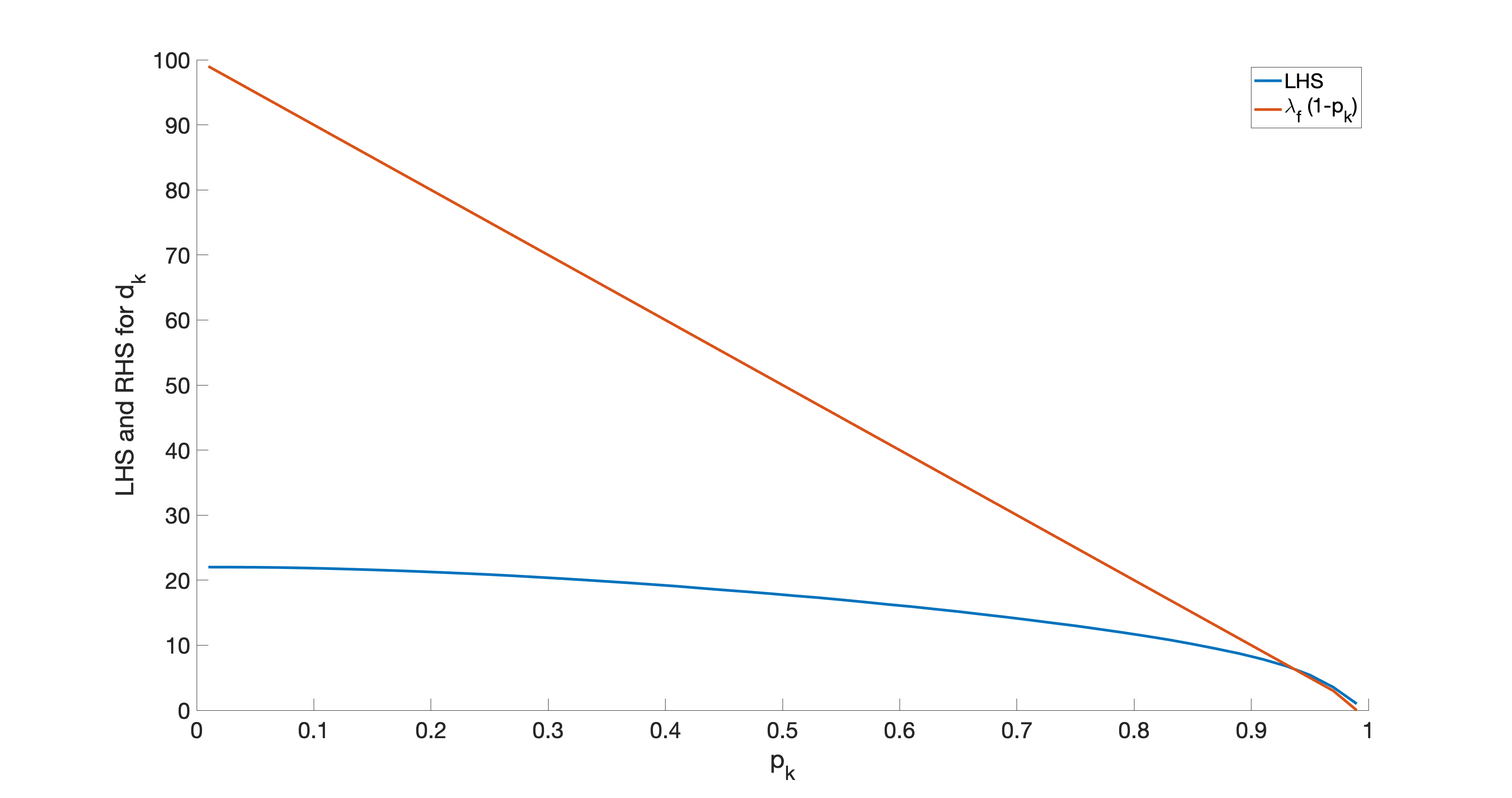}
	\caption{LHS and RHS of (\ref{eqn:thd}) vs. $p_k$ for System-A. $\lambda_e = 0.2$, $\lambda_f = 100$, and $\sigma_{e}^2 = 1.19$.}
	\label{fig:lsh_rhs_dk}
\end{figure}

}
Next, we briefly discuss the computational runtime complexity of the proposed policy. 
\subsection{Computational complexity} \label{subsec:complexity} 
The proposed technique is an online method. At run time, we only need to evaluate $p_k$ (\ref{eqn:pk_recurr}) and compare it with two thresholds at each time step. For our problem formulation, most of the heavy computations, such as matrix inversion and computation of determinants, associated with the evaluation of the likelihood ratio (\ref{eqn:Lj_2nd}) can be derived offline since the variances are fixed, see (\ref{eqn:sigma_1j})-(\ref{eqn:sigma_0j}). The most expensive operations at run-time are a few matrix-vector multiplications with the highest computational complexity of $O(np)$, see (\ref{eqn:mu_1j}) and (\ref{eqn:mu_2j}). 
\section{Derivations of ADD, FAR and $\Delta LQG$} \label{sec:ADD_FAR_LQG}
This section derives the asymptotically approximate analytical expressions of ADD, FAR and $\Delta LQG$ for the given thresholds $Th^s$ and $Th^d$, and a few other parameters to be defined later. 
\subsection{Approximate Expressions of ADD and FAR} \label{subsec:ADD_FAR}
Here we derive the approximate expressions of ADD and FAR, as $Th^d \rightarrow \infty$, applying non-linear renewal theory \cite{siegmund2013sequential,Tartakovsky2005}. First, the Shiryaev statistics $SR_k$ is converted into $LSR_k = \log\left( SR_k\right)$ for the ease of asymptotic analysis. $LSR_k$ can be expressed as a summation of two variables, $S_k$ and $l_k$, as given in the following Lemma~\ref{lemma:LSRn}.
%The following form of the $LSR_n$ is derived by taking log on both sides of (\ref{eqn:SRn_recurr}) and starting from $SLR_0=0$. 
%\begin{equation}
%	\begin{aligned}
%		&LSR_n=\lambda_n+n|\log(1-\rho)| \\
%		&+\log\left(\sum_{i=1}^n(1-\rho)^{i-1}\mathcal{L}_{d,i}\exp(-\lambda_i)\mathbbm{1}_{\left\{LSR_i < Th^S \right\}} \right. \\
%		& \left. +\sum_{i=1}^n(1-\rho)^{i-1}\mathcal{L}_{b,i}\exp(-\lambda_i)\mathbbm{1}_{\left\{LSR_i \ge Th^S \right\}}\right) \\
%		&+\sum_{k=1}^n\log\left(\mathcal{L}_{c,k}\right)\mathbbm{1}_{\left\{LSR_k < Th^S \right\}}-\sum_{k=1}^n\log\left(\mathcal{L}_{a,k}\right)\mathbbm{1}_{\left\{LSR_k < Th^S \right\}},
%	\end{aligned}
%\label{eqn:LSRn}
%\end{equation}
%where 
%\begin{align}
%	&\lambda_n = \sum_{i=1}^n\log\left(\mathcal{L}_{a,i} \right) 	\text{, and}  \label{eqn:lambda_n} \\
%    &Th^S = \log(Th^s) \label{eqn:ThS}
%\end{align}
\begin{lemma} \label{lemma:LSRn}
The logarithm of the Shiryaev statistics, $LSR_k$, generated from the test data, \ie, the innovation signal $\gamma_k$ ((\ref{eqn:gamma_vs_e}) and (\ref{eqn:gamma_attack_vs_e})) and the watermarking signal ${\bf e}_{k}$ (\ref{eqn:ek}), under the two threshold policy $Th^s$ and $Th^d$, can be expressed as the summation of two variables $S_k$ and $l_k$, see (\ref{eqn:LSRn_2nd}). Here $S_k$ (\ref{eqn:Sn_2nd}) is a ladder variable, and $l_k$ (\ref{eqn:ln_2nd}) is a slowly changing variable in the sense defined in \cite{siegmund2013sequential}.
	\begin{equation}
		LSR_k = S_k+l_k.
		\label{eqn:LSRn_2nd}
	\end{equation}
	\begin{equation}
	S_k = Z_k+k|\log(1-\rho)|.
	\label{eqn:Sn_2nd}
\end{equation}
\begin{equation}
\begin{aligned}
	&l_k = \log\left(SR_0+\sum_{i=1}^k(1-\rho)^{i-1}\mathcal{L}_{d,i}\exp(-\lambda_i)\mathbbm{1}_{\left\{LSR_i < Th^S \right\}} \right. \\
	& \left. +\sum_{i=1}^k(1-\rho)^{i-1}\mathcal{L}_{b,i}\exp(-\lambda_i)\mathbbm{1}_{\left\{LSR_i \ge Th^S \right\}}\right) \\
	&+\sum_{j=1}^k\log\left(\mathcal{L}_{c,j}\right)\mathbbm{1}_{\left\{LSR_j < Th^S \right\}}-\sum_{j=1}^k\log\left(\mathcal{L}_{a,j}\right)\mathbbm{1}_{\left\{LSR_j < Th^S \right\}},
\end{aligned}
\label{eqn:ln_2nd}
\end{equation}	
where
\begin{align}
Z_k &= \sum_{i=1}^k\log\left(\mathcal{L}_{a,i} \right),  \label{eqn:Zn} \\ 
\lambda_k &=\sum_{i=1}^k\log\left(\mathcal{L}_{a,i} \right)\mathbbm{1}_{\left\{LSR_i \ge Th^S \right\}}  \cr
&+\sum_{i=1}^k\log\left(\mathcal{L}_{c,i} \right)\mathbbm{1}_{\left\{LSR_i < Th^S \right\}} \text{, and} \label{eqn:lambda_n} \\
Th^S &=\log\frac{Th^s}{\rho\left(1-Th^s\right)} \label{eqn:ThS}.
\end{align}
\end{lemma}
\begin{proof}
	The proof of Lemma~\ref{lemma:LSRn} is provided in Appendix~\ref{apdx:LSRn}.
\end{proof}
\begin{remark}
	The threshold $Th^s$ for $p_k$ is equivalent to the threshold $Th^S$ for $LSR_k$. Similarly, we can define the threshold $Th^D$ as 
	\begin{equation}
		Th^D \triangleq \log\frac{Th^d}{\rho\left(1-Th^d\right)} \label{eqn:ThD}
	\end{equation}
	  for $LSR_k$, which is equivalent to the threshold $Th^d$ for $p_k$. Also, as $Th^d \rightarrow 1$,  $Th^D \rightarrow \infty$.  
\end{remark}
Therefore, Lemma~\ref{lemma:LSRn} enables us to apply non-linear renewal theory to derive the approximate expressions of ADD and FAR by splitting the logarithm of the Shiryaev statistics, $LSR_k$, into a ladder variable $S_k$ and a slowly changing term $l_k$. The definition of a slowly changing variable from \cite{siegmund2013sequential} is also provided in Appendix~\ref{apdx:LSRn}. 
We define the variable $r$ to be the overshoot of the ladder variable $S_{n_d}$ over a large threshold $Th^D$ at $k  = n_d$. Therefore, $r_{n_d} \triangleq S_{n_d}-Th^D$ as $Th^D \rightarrow \infty$, and $n_d = \inf\left\{k\ge1:S_k \ge Th^D \right\}$. According to the non-linear renewal theory \cite{siegmund2013sequential}, the overshoot statistics of $LSR_k$ crossing a large threshold $Th^D$ can be approximated as the statistics of $r_{n_d}$, provided $l_k$ is slowly changing and $Th^D \rightarrow \infty$. The approximate expressions of ADD and FAR derived in this paper are stated in Theorem~\ref{th:ADD_FAR}. 
\begin{myth} \label{th:ADD_FAR}
	For the Shiryaev statistics given in Lemma~\ref{lemma:SRn} and the geometric prior distribution of the change point $\Gamma$ (\ref{eqn:Pi}), under the two threshold policy $Th^s$ and $Th^d$, the asymptotic approximate expressions of ADD and FAR as $Th^D \rightarrow \infty$ will take the following forms, provided the conditions C1-C4 are satisfied. 
\begin{align}
		&\textbf{Conditions:} \nonumber \\
		&\text{C1: } \left\{Z_k: k\ge1 \right\} \text{is nonarithmetic with respect to } \text{P}_0 \text{ and } \text{P}_1. \nonumber \\
		&\text{C2: } \text{E}_1\left[\mid Z_1 \mid ^2 \right] \text{ is finite}. \nonumber \\
		&\text{C3: } l_k \text{ (\ref{eqn:ln_2nd}) is a slowly changing variable in the} \nonumber \\
			&\text{sense defined in \cite{siegmund2013sequential}}. \nonumber \\
		&\text{C4: } 0< \text{E}_1\left[\text{D}\left(\mathpzc{f}^e_{1},\mathpzc{f}_{0}\right)\right] < \infty \text{, and } 0< \text{E}_0\left[\text{D}\left(\mathpzc{f}_{0},\mathpzc{f}^e_{1}\right)\right]	< \infty. \nonumber
	\end{align}
	Then, 
	\begin{equation}
		ADD = \frac{Th^D +{\bar r}-{\bar l}}{\text{E}_1\left[\text{D}\left(\mathpzc{f}^e_{1},\mathpzc{f}_{0}\right) \right]+|\log(1-\rho)|} + o(1),
		\label{eqn:ADD_2nd}
	\end{equation}
\begin{equation}
	\text{and } FAR \approx \frac{\xi}{\rho \exp\left(Th^D\right)}\left( 1+ o(1)\right),  \text{ as } Th^D \rightarrow \infty,
	\label{eqn:FAR}
\end{equation}
where
\begin{align}
	&{\bar r}= \lim_{n_d \rightarrow \infty}\text{E}_1\left[r_{n_d} \right], \label{eqn:rbar} \\
	&{\bar l}= \lim_{k \rightarrow \infty }\text{E}_1\left[l_k \right], \label{eqn:lbar} \\
	&\xi =\lim_{n_d \rightarrow \infty }\text{E}_1 \left[\exp\left(-r_{n_d}\right) \right]. \label{eqn:Xi}
\end{align}
$\text{P}_0$ and $\text{P}_1$ denote the probability measures before and after the attack, respectively. $\text{E}_0$ and $\text{E}_1$ denote the expectations with respect to the probability measures $\text{P}_0$ and $\text{P}_1$, respectively. $\text{E}_1\left[\text{D}\left(\mathpzc{f}^e_{1},\mathpzc{f}_{0}\right) \right]$ is the expected KLD between the distributions $\mathpzc{f}^e_{1,j}\left(\cdot|\cdot \right)$ and $\mathpzc{f}_{2,j}\left(\cdot|\cdot \right)$, and $\mathpzc{f}^e_{1,j}\left(\cdot|\cdot \right) = \mathpzc{f}_{1,j}\left(\cdot|\cdot \right)$ when $s_{j}=1$ for all $j$. Here, the expectation is taken over the joint distribution of the innovation signal and the watermarking signal after the attack start point. Similarly, $\text{E}_0\left[\text{D}\left(\mathpzc{f}_{0},\mathpzc{f}^e_{1}\right)\right]$ is the expected KLD between $\mathpzc{f}_{0}$ and $\mathpzc{f}^e_{1}$, and the expectation is taken over the joint distribution of the innovation signal and the watermarking signal before the attack start point. 
\end{myth}

\begin{proof}
	The proof of Theorem~\ref{th:ADD_FAR} is provided in Appendix~\ref{apdx:ADD_FAR}.
\end{proof}
\begin{remark} \label{rem:kld}
 As given in Lemma~1 and Theorem~2 of \cite{watermarking_tac}, $\text{E}_1\left[\text{D}\left(\mathpzc{f}^e_{1},\mathpzc{f}_{0}\right) \right]$ will take the following form,
\begin{align}
	\text{E}_1\left[\text{D}\left(\mathpzc{f}^e_{1},\mathpzc{f}_{0}\right) \right]  
	=\frac{1}{2}\left\{tr\left({\bf \Sigma}_0^{-1} {\bf \Sigma}_{\widetilde \gamma}\right) -m - \log\frac{\mid{\bf Q}_a\mid}{\mid{\bf \Sigma}_{0}\mid}\right\},
	\label{eqn:opt_kld}
\end{align}
where the covariance matrix ${\bf \Sigma}_{\widetilde \gamma}$ of the innovation signal after the attack start point is given as 
\begin{align}
	{\bf \Sigma}_{\widetilde \gamma}&={\bf E}_{zz}(0)-{\bf C}({\bf A}+{\bf B{\bf L)}}{\bf E}_{xz}(-1) \cr
	&-\left[{\bf C}({\bf A}+{\bf B{\bf L)}}{\bf E}_{xz}(-1) \right]^T+{\bf C}{\bf B}{\bf \Sigma}_e{\bf B}^T{\bf C}^T  \cr
	&+{\bf C}({\bf A}+{\bf B}{\bf L}){\bf \Sigma}_{x^Fz}({\bf A}+{\bf B}{\bf L})^T{\bf C}^T \cr
	&+{\bf C}({\bf A}+{\bf B}{\bf L}){\bf \Sigma}_{x^Fe}({\bf A}+{\bf B}{\bf L})^T{\bf C}^T,
	\label{eqn:sigma_gamma_attack} \\
	&\text{where }{\bf E}_{xz}(-1)= \sum_{i=0}^\infty\mathcal{A}^{i}{\bf K}{\bf A}_a^{i+1}{\bf E}_{zz}\left(0\right) \label{eqn:Exz1_original}
\end{align}
and ${\bf E}_{zz}(0) = \text{E}\left[{\bf z}_k {\bf z}_k^T\right]$. ${\bf \Sigma}_{x^Fz}$ and ${\bf \Sigma}_{x^Fe}$ are the solutions to the following Lyapunov equations,
\begin{align}
	&\mathcal{A}{\bf \Sigma}_{x^Fz}\mathcal{A}^T-{\bf \Sigma}_{x^Fz}+{\bf K}{\bf E}_{zz}(0){\bf K}^T+\mathcal{A}{\bf E}_{xz}(-1){\bf K}^T \cr
	&+\left(\mathcal{A}{\bf E}_{xz}(-1){\bf K}^T\right)^T = 0 \text{, and} \label{eqn:ExFxF_th1p1} \ \\
	&\mathcal{A}{\bf \Sigma}_{x^Fe}\mathcal{A}^T-{\bf \Sigma}_{x^Fe}+\left({\bf I}_n-{\bf K}{\bf C}\right){\bf B}{\bf \Sigma}_e{\bf B}^T\left({\bf I}_n-{\bf K}{\bf C}\right)^T = 0. \cr
	\label{eqn:ExFxF_th1p2}
\end{align}
Here $\mathcal{A} =\left({\bf I}_n-{\bf K}{\bf C}\right)\left({\bf A}+{\bf B}{\bf L}\right)$, which is assumed to be strictly stable. ${\bf I}_n$ is an identity matrix of size $n \times n$.
\end{remark}
Therefore, we can derive approximate values of ADD and FAR using Theorem~\ref{th:ADD_FAR} for the given thresholds, $Th^d$ and $Th^s$, and the system and noise parameters. The denominator of (\ref{eqn:ADD_2nd}) does not depend on the thresholds. Also, according to the renewal theory, the statistics obtained from the overshoot $r_{n_d}$, \ie, $\bar r$ and $\xi$, are not dependent on the exact values of the thresholds as long as  $Th^d$ is large enough. However, from (\ref{eqn:ln_2nd}), we can say that $\bar l$ is dependent on the threshold $Th^s$.
Further approximation of the expression of ADD (\ref{eqn:ADD_2nd}) can be directly obtained from Theorem~\ref{th:ADD_FAR} as stated in Corollary~\ref{corr:ADD_1st}. 
\begin{corollary} \label{corr:ADD_1st}
	The approximate expression of ADD as provided in Theorem~\ref{th:ADD_FAR} can be further simplified as follows,
	\begin{equation}
		ADD \approx \frac{Th^D}{\text{E}_1\left[\text{D}\left(\mathpzc{f}^e_{1},\mathpzc{f}_{0}\right) \right]+|\log(1-\rho)|} \text{, as } Th^D \rightarrow \infty
		\label{eqn:ADD_1st}
		\end{equation}
\end{corollary}
\begin{proof}
${\bar r} \ll Th^D$ and ${\bar l} \ll Th^D$, since $Th^D \rightarrow \infty$. Therefore, by ignoring ${\bar r}$ and ${\bar l}$ from (\ref{eqn:ADD_2nd}), we get (\ref{eqn:ADD_1st}).
\end{proof}
The approximate expression of ADD as provided in Corollary~\ref{corr:ADD_1st} does not depend on the threshold $Th^s$. Therefore, Corollary~\ref{corr:ADD_1st} can also be used to find a suitable value of the threshold $Th^d$ for a given ADD.
\begin{remark}
	Finding analytical expressions for ${\bar r}$, ${\bar l}$, and ${\xi}$ is difficult for the system under consideration. Therefore, we estimate their values by Monte-Carlo (MC) simulation. The values of ${\bar r}$, ${\bar l}$, and ${\xi}$ are not directly dependent on $Th^D$ as long as $Th^D$ is very large, but they depend on $Th^S$. However, to derive the ADD using the second approximate expression as given in Corollary~\ref{corr:ADD_1st}, we do not need the values of ${\bar r}$, ${\bar l}$, and ${\xi}$, but it is less accurate compared to Theorem~\ref{th:ADD_FAR}. 
\end{remark}
\begin{remark}
In order to use the quickest detection scheme, the pre and post-change pdfs must be known. To achieve that, we need to know ${\bf A}_a$ and ${\bf Q}_a$. In practice, it is highly likely that the attacker's system parameters ${\bf A}_a$ and ${\bf Q}_a$ may not be known a priori. In such a case, the attacker's system parameters ${\bf A}_a$ and ${\bf Q}_a$ can be estimated online from the received observations (true or fake) by fitting a vector autoregressive model to the observations \cite{akaike1969fitting}. This estimator will operate in parallel with the attack detection algorithm. Such a parameter estimation scheme will operate before and after the attack. However, before the attack, the estimates of ${\bf A}_a$ and ${\bf Q}_a$ will represent the healthy plant model. We have conducted some preliminary studies using a MISO system where our attack detection algorithm can perform with estimated parameters, albeit with addtional watermarking compared to the 
known parameter case. A  detailed analysis of such a joint estimation and detection scheme is however, beyond the scope of the current manuscript, and interested readers are referred to \cite{xie2021sequential}. Additionally, we also comment that under a replay attack, ${\bf A}_a$ and ${\bf Q}_a$ can be derived from the normal system model as discussed in \cite{naha_replay_attack}.
\end{remark}
\subsection{Approximate Expression of ANW and $\Delta LQG$} \label{subsec:deltaLQG}
Following similar steps as in \cite{Banerjee2012}, the ANW can be approximated as follows,
\begin{equation}
	\begin{aligned}
	ANW \approx &\frac{\text{E}_0\left[t_1\left(Th^S\right) \right]}{\text{E}_0\left[t_1\left(Th^S\right) \right]+\text{E}_0\left[t_2\left(LSR_{Th^S},Th^S\right) \right]} \\
	& \times \text{P}\left\{t\left(Th^S\right) < \Gamma \right\}.
	\end{aligned}
\label{eqn:ANW}
\end{equation}
Here, $t_1\left(Th^S\right)$ denotes the time interval between the time instances when $LSR_k$ starts from $Th^S$, and then crosses the threshold $Th^S$ from above. $t_2\left(LSR_{Th^S},Th^S\right)$ denotes the time interval between the time instances when $LSR_k$ starts from $LSR_{Th^S}$ and crosses the threshold $Th^S$ from below. $t\left(Th^S\right)$ is the first time $LSR_k$ crosses the threshold $Th^S$ from below. An example plot of $LSR_k$ is shown in Fig.~\ref{fig:timing_diagram} to illustrate the variables $t_1(\cdot)$, $t_2(\cdot)$, and $t(\cdot)$. $\text{E}_0[\cdot]$ denotes the expectation with respect to the probability measure before the attack. Deriving analytical expressions for the expectations and the probability values in (\ref{eqn:ANW}) is difficult. Therefore, we perform MC simulation to estimate the ANW for the given thresholds $Th^S$ and $Th^D$. The relationship between the ANW and the increase in the control cost is given in the following theorem. 
\begin{figure}[h!]
	\centering
	\includegraphics[width=\figwidth]{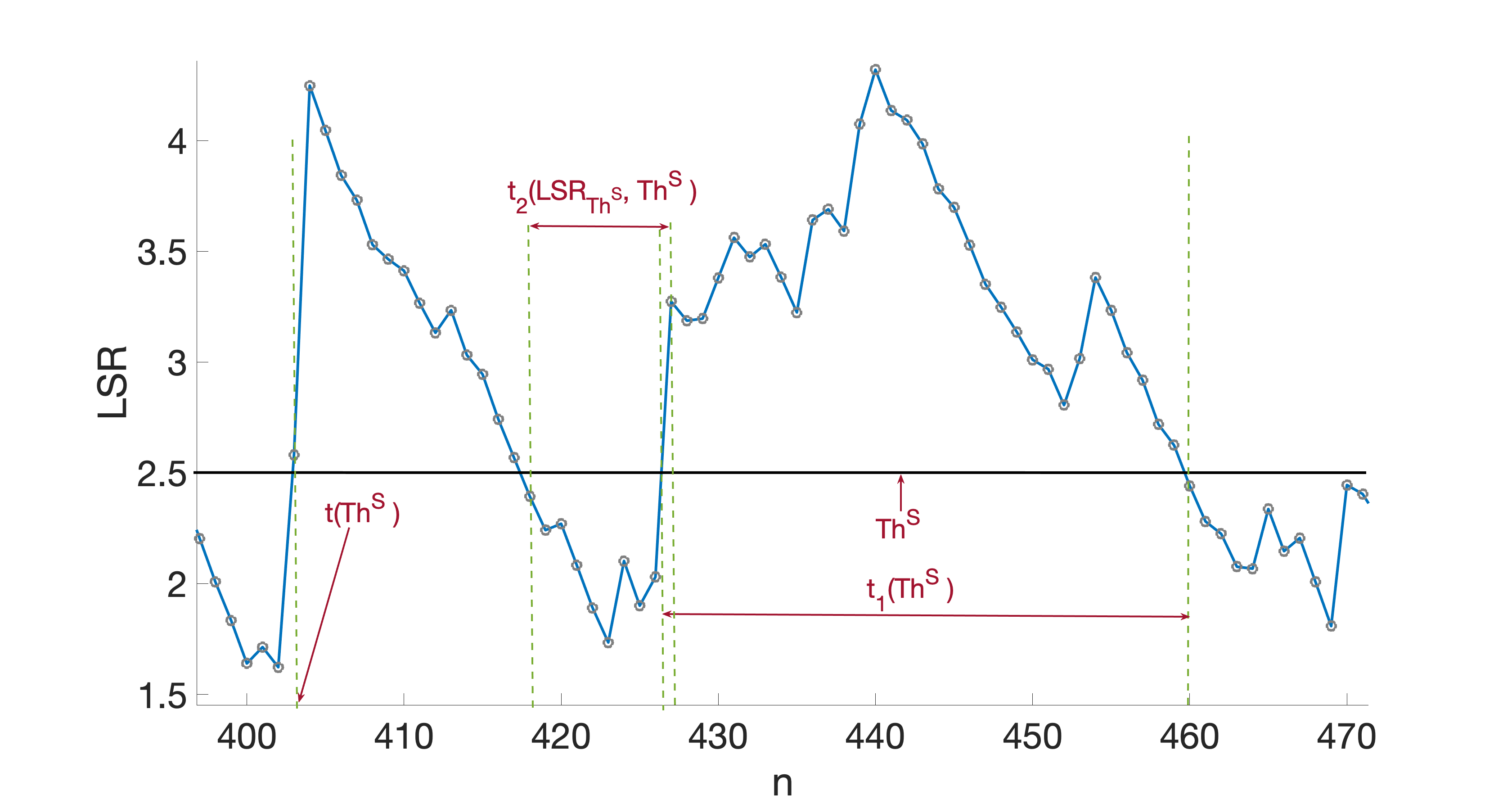}
	\caption{An example plot of $LSR$ vs. time index $n$}
	\label{fig:timing_diagram}
\end{figure}

\begin{myth} \label{th:DeltaLQG}
	For the parsimonious watermarking scheme adopted in this paper, the increase in the LQG control cost, $\Delta LQG$, is related to ANW as
	\begin{equation}
		\Delta LQG = \rho ANW tr\left(H\Sigma_e \right),
		\label{eqn:DeltaLQG} 
	\end{equation}
where
\begin{equation}
	{\bf H} = {\bf B}^T{\bf \Sigma}_L{\bf B}+{\bf U}
	\label{eqn:H}
\end{equation}
and ${\bf \Sigma}_L$ is the solution to the following Lyapunov equation.
\begin{align}
	\left({\bf A}+{\bf B}{\bf L}\right)^T{\bf \Sigma}_L\left({\bf A}+{\bf B}{\bf L}\right)-{\bf \Sigma}_L+{\bf L}^T{\bf U}{\bf L} +{\bf W}=0.
	\label{eqn:Sigma_L}
\end{align}
\end{myth}
\begin{proof}
	The proof of Theorem~\ref{th:DeltaLQG} is provided in Appendix~\ref{apdx:DeltaLQG}.
\end{proof}
Theorem~\ref{th:DeltaLQG} shows that $\Delta LQG$ is proportional to ANW and a linear function of the watermarking signal variance $\Sigma_e$. If watermarking is added at all the time instants, $\rho ANW$ will become unity, and Theorem~\ref{th:DeltaLQG} will coincide with the special case of always present watermarking as stated in Theorem~3 in \cite{watermarking_tac}.
%\begin{remark}
%	We select $Th^s$ from $\left\{Th_{s,i}:i=1,2, \cdots \right\}$ which gives acceptable $\Delta LQG$.
%\end{remark}
\subsection{Comparative Analysis} \label{subsec:comp_analysis} 
The proposed method is compared with the following two methods, PW-$\Sigma_e$: persistent watermarking with fixed watermarking power and PW-$\Delta LQG$: persistent watermarking with fixed $\Delta LQG$. The only difference between the proposed method and PW-$\Sigma_e$ is that the watermarking is always present for the latter, and the watermarking power for both the methods is ${\bf \Sigma}_e$. On the other hand, the only difference between the proposed method and PW-$\Delta LQG$ is that the watermarking is always present for the latter, and the $\Delta LQG$ value is the same for both. The subscripts $P$, $A$ and $B$ denote the proposed method, PW-$\Sigma_e$, and PW-$\Delta LQG$, respectively.
\subsubsection{Comparison with PW-$\Sigma_e$} \label{subsubsec:methodA}
\begin{clm}
	The proposed optimal watermarking policy incurs a lesser increase in LQG cost compared to PW-$\Sigma_e$. 
	\label{clm:LQG_A}
\end{clm} 
The increase in the LQG control cost for PW-$\Sigma_e$ is as follows, see Theorem~3 from \cite{watermarking_tac},
	\begin{equation}
	\Delta LQG_A =  tr\left(H\Sigma_e \right).
	\label{eqn:DeltaLQGA} 
\end{equation}
By comparing the increase in the LQG control cost between the two methods, we can write
\begin{equation}
	\widetilde {\Delta LQG} = \Delta LQG_A - \Delta LQG_P = \left(1-\rho ANW\right)\Delta LQG_A,
	\label{eqn:diff_deltaLQG}
\end{equation}
where $\Delta LQG_P$ denotes the increase in the LQG control cost for the proposed method. Since, $\left(1-\rho ANW\right) < 1$, we can make Claim~\ref{clm:LQG_A}.

\begin{clm}
	The increase in ADD for the proposed optimal policy with respect to PW-$\Sigma_e$ will be small .  
	\label{clm:ADD_A}
\end{clm} 

There will be an increase in the ADD for the proposed method. Other than the mean of the slowly changing term, $\bar l$, in the ADD expression (\ref{eqn:ADD_2nd}), the rest of the components will be the same for the proposed method and PW-$\Sigma_e$. Therefore, the increase in the ADD for the proposed method will be as follows, 
\begin{equation}
	\Delta ADD = ADD_P-ADD_A \approx \frac{{\bar l}_A-{\bar l}_P}{\text{E}_1\left[\text{D}\left(\mathpzc{f}^e_{1},\mathpzc{f}_{0}\right) \right]+|\log(1-\rho)|},
	\label{eqn:deltaADD_A}
\end{equation}
where the subscripts $A$ and $P$ denote the PW-$\Sigma_e$ and the proposed method, respectively. $o(1)$ notation is dropped for simplicity. Here ${\bar l}_P$ is the same as given by (\ref{eqn:ln_2nd}) and (\ref{eqn:lbar}), and ${\bar l}_A$ will take the following form
	\begin{align}
		{\bar l}_A&= \lim_{k \rightarrow \infty }\text{E}_1\left[l_{A,k} \right] \text{, and} \nonumber \\
		l_{A,k} &=  \log\left(SR_0+\sum_{i=1}^k(1-\rho)^{i-1}\mathcal{L}_{b,i}\exp(-Z_i)\right).
		\label{eqn:lbarA}
	\end{align}
Since,  ${\bar l}_A$ and ${\bar l}_P$ both are small quantities compared to $Th^D$, which is assumed to be $\rightarrow \infty$, we can make Claim~\ref{clm:ADD_A}.

\begin{clm}
	The FAR for the proposed optimal policy and PW-$\Sigma_e$ will almost be the same..
	\label{clm:FAR_A}
\end{clm}

Since $Th^D \ge Th^S$, the watermarking will be present for both cases when $Z_n$ crosses the threshold $Th^D$. In other words, the statistics of the overshoot $r_{n_d}$ will be the same for both methods. Therefore, we can make Claim~\ref{clm:FAR_A}. 

\subsubsection{Comparison with PW-$\Delta LQG$} \label{subsubsec:methodB}
  \begin{clm}
	The watermarking signal power for the proposed optimal policy, $\Sigma_{eP}$, will be greater than or equal to the watermarking signal power of PW-$\Delta LQG$, $\Sigma_{eB}$. 
	\label{clm:Sigma_e_B}
\end{clm}

 Since the increase in the LQG control cost is taken to be the same for both the methods, the watermarking signal powers $\Sigma_{eB}$ and $\Sigma_{eP}$ for the method PW-$\Delta LQG$ and the proposed method, respectively, will be different. The relationships between the watermarking signal power and $\Delta LQG$ for both the methods are given as, 
  \begin{equation}
 	\Delta LQG = tr\left(H\Sigma_{eB} \right) = \rho ANW tr\left(H\Sigma_{eP} \right).
 	\label{eqn:sigma_P}
 \end{equation}
  Since $\rho ANW \le 1$, from (\ref{eqn:sigma_P}) we can make Claim~\ref{clm:Sigma_e_B}.

\begin{clm} 
	The ADD for the proposed optimal policy will be less than or equal to the ADD for PW-$\Delta LQG$. 
	\label{clm:ADD_B}
\end{clm}
We use the ADD expression from Corollary~\ref{corr:ADD_1st} to compare the two methods. The difference in the ADD will be due to the difference in $\text{E}_1\left[\text{D}\left(\mathpzc{f}^e_{1},\mathpzc{f}_{0}\right) \right]_B$ and $\text{E}_1\left[\text{D}\left(\mathpzc{f}^e_{1},\mathpzc{f}_{0}\right) \right]_P$ as follows.
\begin{equation}
	\begin{aligned}
	\text{E}_1\left[\text{D}\left(\mathpzc{f}^e_{1},\mathpzc{f}_{0}\right) \right]_P- &\text{E}_1\left[\text{D}\left(\mathpzc{f}^e_{1},\mathpzc{f}_{0}\right) \right]_B  \\
	& \approx \frac{1}{2}\left(\text{tr}\left(\Sigma_{\gamma}^{-1}\left(\Sigma_{\widetilde \gamma{_P}}-\Sigma_{\widetilde \gamma{_A}} \right) \right) \right).
	\label{eqn:Delta_KLD}
	\end{aligned}
\end{equation}
Here the subscripts $B$ and $P$ denote the method PW-$\Delta LQG$ and the proposed method, respectively. By examining (\ref{eqn:sigma_gamma_attack}), we can say $\Sigma_{\widetilde \gamma_{P}}-\Sigma_{\widetilde \gamma_{B}} \ge 0$. Therefore, from (\ref{eqn:Delta_KLD}) we can write $\text{E}_1\left[\text{D}\left(\mathpzc{f}^e_{1},\mathpzc{f}_{0}\right) \right]_P-\text{E}_1\left[\text{D}\left(\mathpzc{f}^e_{1},\mathpzc{f}_{0}\right) \right]_B \ge 0$, and we can further make Claim~\ref{clm:ADD_B}.

\subsection{Optimum $\Sigma_e$} \label{subsec:optimum_sigma_e}
Theorem~\ref{th:ADD_FAR} and Corollary~\ref{corr:ADD_1st} imply that the increase in KLD will reduce ADD. Therefore, we derive the optimum $\Sigma_e$ that will maximize KLD for a given fixed upper limit on $\Delta LQG$, denoted as $\Delta LQG_P$ for the proposed method. The optimization problem is defined as follows. 
\begin{align}
	\max_{{\bf \Sigma}_e} &\  \text{E}_1\left[\text{D}\left(\mathpzc{f}^e_{1},\mathpzc{f}_{0}\right) \right] \label{eqn:opt_prob},  \nonumber \\
	\textrm{s.t.}\  & \Delta LQG_P \le J, \\
	& {\bf \Sigma}_e \ge 0, \nonumber
\end{align}
where $J$ is a user-defined threshold. As given in Remark~\ref{rem:kld}, the KLD expression for the proposed parsimonious watermarking policy is identical with the case where watermarking is always present, \ie, the method PW-$\Sigma_e$. Moreover, $\Delta LQG_P$ (\ref{eqn:DeltaLQG}) is just a scaled version of $\Delta LQG_A$ (\ref{eqn:DeltaLQGA}). Therefore, the condition $\Delta LQG_P \le J$ in (\ref{eqn:opt_prob}) can be replaced by $\Delta LQG_A \le J_A$, where $J_A = J/(\rho ANW)$, without any change in the optimum $\Sigma_e$ value. Now, the optimization problem for the proposed method becomes identical to the optimization problem for the method PW-$\Sigma_e$. According to Theorem~4 from \cite{watermarking_tac}, the optimum $\Sigma_e$ for PW-$\Sigma_e$ will be a rank one positive semi-definite matrix. Therefore, the optimization problem in (\ref{eqn:opt_prob}) can be written as 
\begin{align}
	\max_{{\bf v}_{\lambda}} &\  \text{E}_1\left[\text{D}\left(\mathpzc{f}^e_{1},\mathpzc{f}_{0}\right) \right] \label{eqn:opt_prob_2nd}  \nonumber \\
	\textrm{s.t.}\  & \Delta LQG_A \le J_A ,
\end{align}
where ${{\bf v}_{\lambda}} = \sqrt \sigma_e{{\bf v}_{e}}$, $\sigma_e$ is the non-zero eigenvalue of $\Sigma_e$ and ${{\bf v}_{e}}$ is the corresponding eigenvector. As discussed in \cite{watermarking_tac}, the maximization of $\text{E}_1\left[\text{D}\left(\mathpzc{f}^e_{1},\mathpzc{f}_{0}\right) \right]$ with respect to ${{\bf v}_{\lambda}}$ is same as maximizing the following function,
\begin{align}
	&\max_{{\bf v}_{\lambda}} \  {\bf v}_{\lambda}^T{\bf H}_{KLD}{\bf v}_{\lambda} \label{eqn:opt_prob_3rd}  \nonumber \\
	&\textrm{s.t.}\   \Delta LQG_A \le J_A, 
	\end{align} 
where
\begin{align}
	{\bf H}_{KLD} = {\bf B}^T \left({\bf I}_n - {\bf K}{\bf C}\right)^T{\kappa}_e\left({\bf I}_n - {\bf K}{\bf C}\right){\bf B}+{\bf B}^T{\bf C}^T{\bf C}{\bf B}. 
\end{align}
Here, ${\kappa}_e$ is the solution to the Lyapunov equation
\begin{align}
	{\cal{A}}^T{\kappa}_e{\cal{A}}-{\kappa}_e+\left({\bf A}+{\bf B}{\bf L}\right)^T{\bf C}^T{\bf C}\left({\bf A}+{\bf B}{\bf L}\right)=0.
	\label{eqn:Le}
\end{align}
Since the matrix ${\cal A}$ is assumed to be strictly stable, the Lyapunov equation of (\ref{eqn:Le}) will have a unique solution.
As discussed in \cite{watermarking_tac}, the optimization problem of (\ref{eqn:opt_prob_3rd}) can be solved by various methods available in the literature, such as sequential quadratic programming (SQP) \cite{Boggs1995}, interior point method \cite{Forsgren2002}, simple gradient-based method \cite{watermarking_tac}, etc. Interested readers are referred to \cite{watermarking_tac} for a detailed analysis, but the same has been removed from the current paper due to space constraints.
\section{Numerical Results}  \label{sec:num_results}
This section will illustrate and validate different aspects of the proposed methodology using, System-A: a second-order multi-input single-output (MISO) open-loop unstable system and System-B: a fourth-order MIMO open-loop stable system. Appendix~\ref{apdx:system_params} provides the required parameters for simulations associated with System A and B. 
\subsubsection{Optimal Policy}
Figure~\ref{fig:ud_for_diff_le} shows the optimal decision variable ${\bf u}^*_{d,k}$ vs. $p_k$ plots for three different values of $\lambda_e$ and a fixed $\lambda_f$  for System-A. The watermarking signal variance is taken to be a diagonal matrix with equal signal power, $\sigma_{e}^2$. We  observe that the optimal policy is a two threshold policy, which validates the theory presented in Sub-section~\ref{subsec:struc_opt_soln}. A higher $\lambda_e$ means a stricter constraint on how much watermarking could be added, which gets reflected into higher $Th^s$. On the other hand, a higher $Th^s$ means watermarking will be added for fewer samples. As discussed in Sub-section~\ref{subsubsec:methodA}, since the added watermarking has little effect on the FAR, the change in $\lambda_e$ does not affect the threshold $Th^d$ much.

% {\color{blue} A higher $\lambda_e$ means a stricter constraint on how much watermarking could be added. On the other hand, higher $\lambda_f$ means a stricter constraint on how much FAR could be allowed.} Therefore, the increase in $\lambda_e$ increases the threshold $Th^s$. On the other hand, a higher $Th^s$ reduces the amount of the watermarking added before the attack start point, but at the same time, it increases the ADD as shown later. The change in $\lambda_e$ does not affect the threshold $Th^d$ much.

\begin{figure}[h!]
	\centering
	\includegraphics[width=\figwidth]{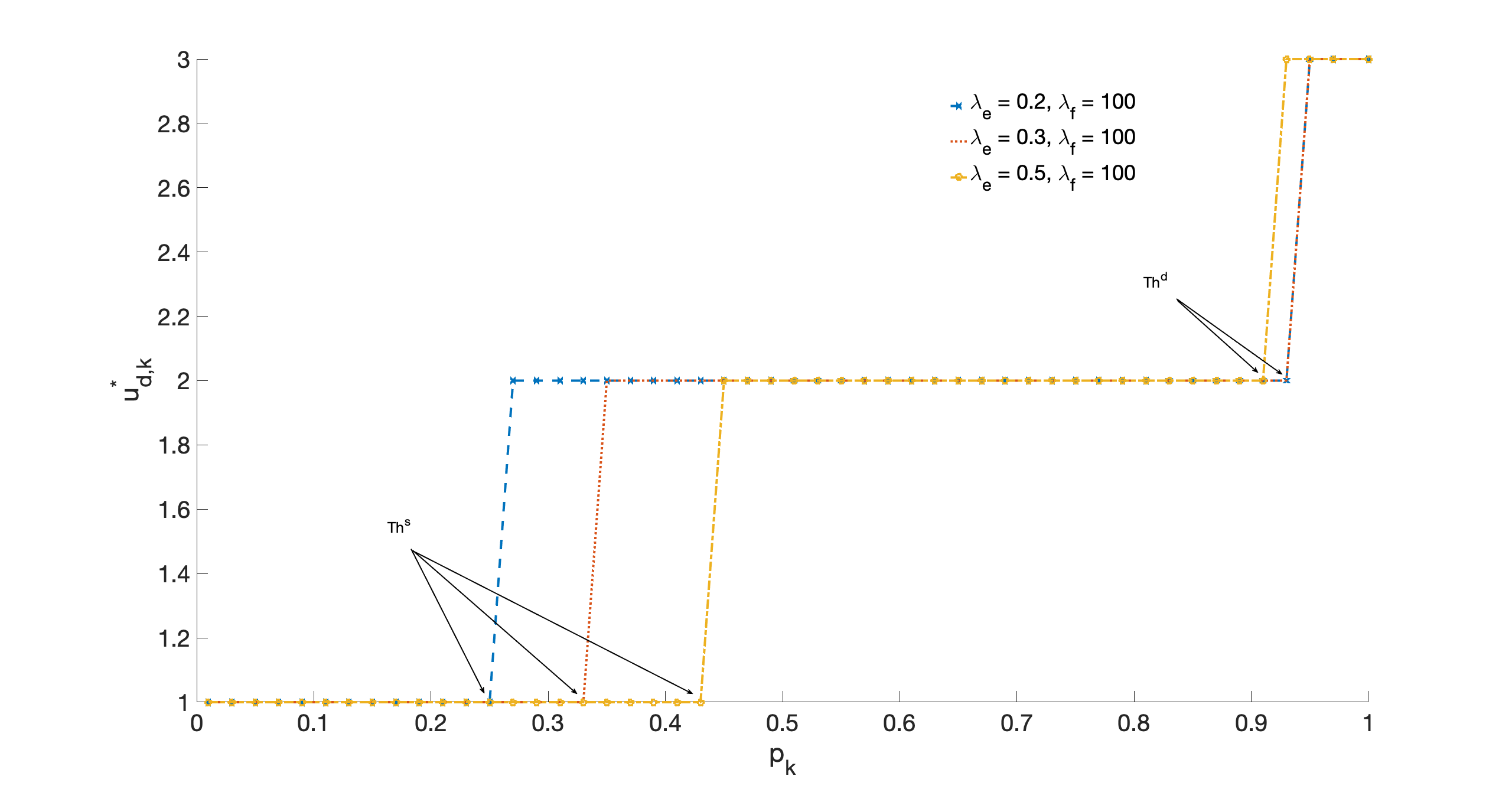}
	\caption{Optimal policy for different $\lambda_e$ and fixed $\lambda_f$ for System-A. $\sigma_{e}^2 = 1.19$.}
	\label{fig:ud_for_diff_le}
\end{figure}

Figure~\ref{fig:ud_for_diff_lf} shows the optimal decision variable ${\bf u}^*_{d,k}$ vs. $p_k$ plots for three different values of $\lambda_f$ and a fixed $\lambda_e$ for System-A. The watermarking signal variance is taken to be a diagonal matrix with equal signal power, $\sigma_e^2$.  A higher $\lambda_f$ means a stricter constraint on how much FAR could be allowed. Therefore, the increase in $\lambda_f$ increases the threshold $Th^d$. {\color{black} However, since $Th^d \ge Th^s$, the change in $\lambda_f$ does not affect the threshold $Th^s$.}

\begin{figure}[h!]
	\centering
	\includegraphics[width=\figwidth]{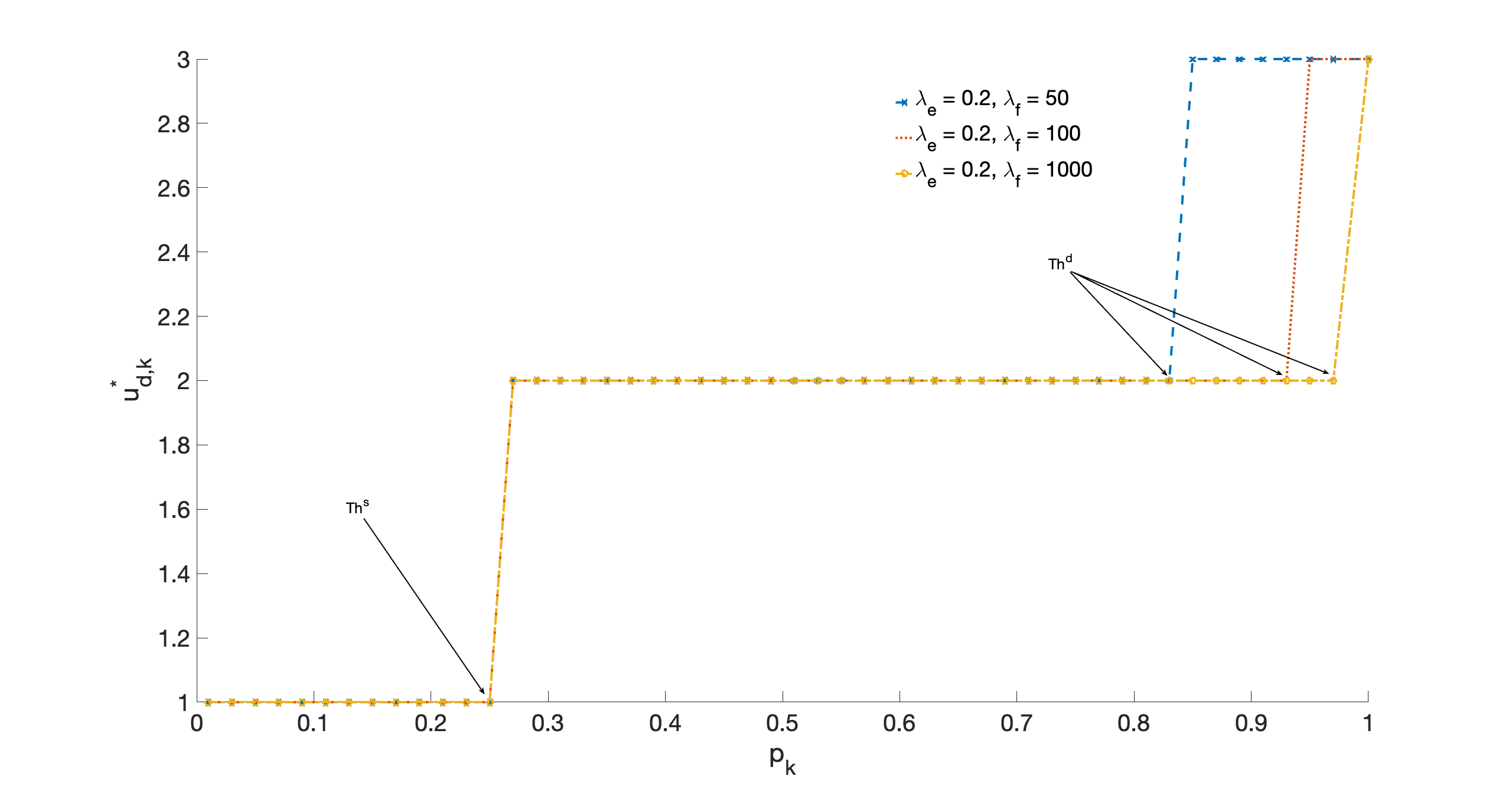}
	\caption{Optimal policy for different $\lambda_f$ and fixed $\lambda_e$ for System-A. $\sigma_{e}^2 = 1.19$.}
	\label{fig:ud_for_diff_lf}
\end{figure}

\subsubsection{Trial Run}
Figure-\ref{fig:tiral_run} illustrates how the control variables $s_k$ and $d_k$, and the sufficient statistics $p_k$ change with the time index $k$ for a sample trial run for System-A. The watermarking signal variance is taken to be a diagonal matrix with equal signal power, $\sigma_{e}^2$. We have also indicated the attack start point as ``change pt" in the plot. This figure provides relevant insights into how the proposed method works. We can observe that only for a very few time instances $p_k \ge Th^s$, and watermarking have been added before the attack. Such parsimonious use of watermarking reduces the control cost before the attack. On the other hand, $p_k$ increases gradually after the attack start point and eventually crosses the threshold $Th^d$. In other words, $p_k \ge Th^s$ and watermarking have been added almost all the time after the change point, resulting in faster detection.

\begin{figure}[h!]
	\centering
	\includegraphics[width=\figwidth]{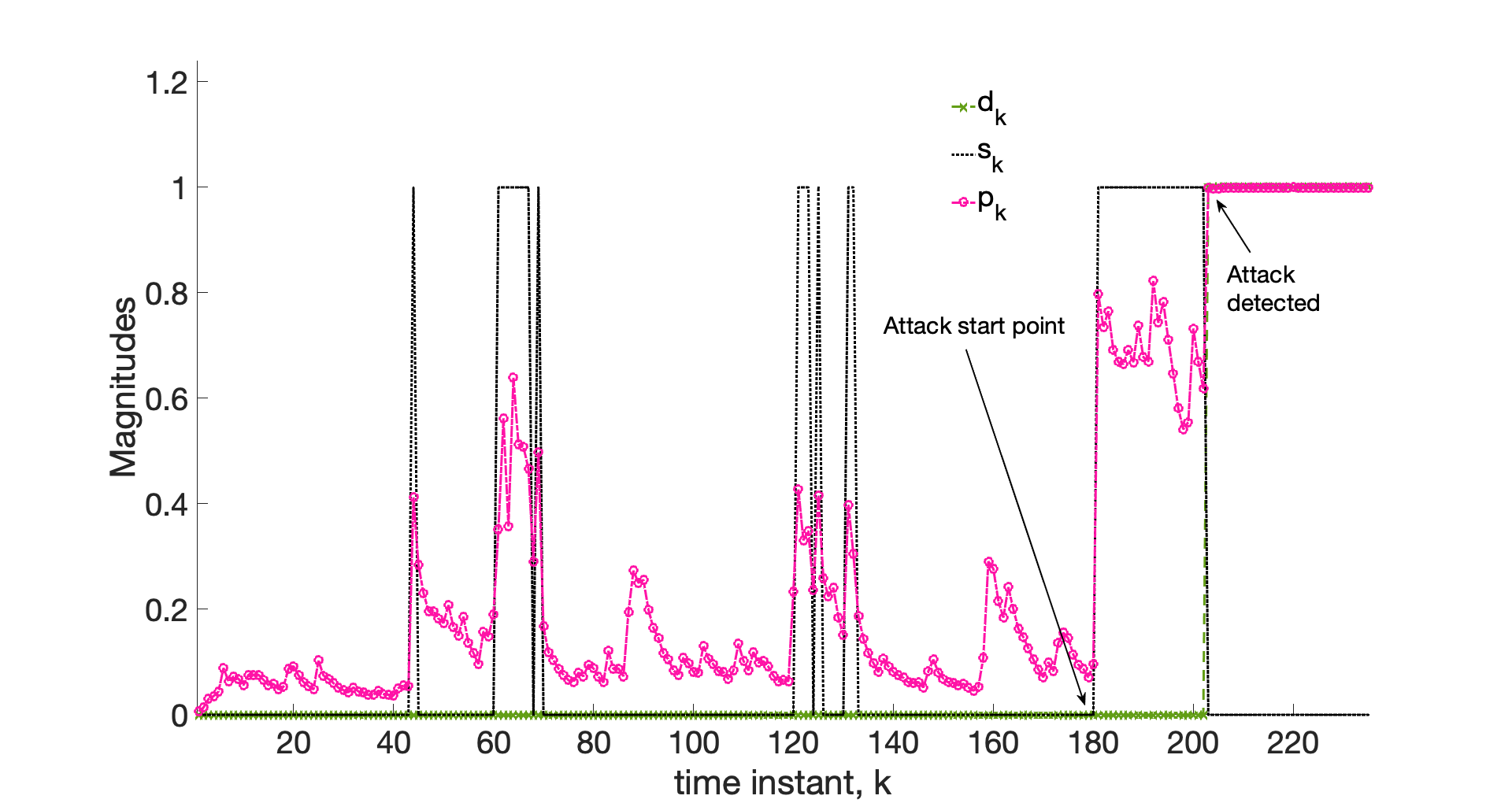}
	\caption{$s_k$, $d_k$ and $p_k$ vs. $k$ for a sample trial run for System-A. $\lambda_e = 0.2$, $\lambda_f = 100$ and $\sigma_{e}^2 = 1.19$.}
	\label{fig:tiral_run}
\end{figure}

\subsubsection{ADD and FAR vs. $\sigma_{e}^2$}
Figure~\ref{fig:ADD_lambda_e} shows the comparison between the plots of ADD and FAR vs. $\sigma_{e}^2$ for two different values of $\lambda_e$ for System-A. $\Sigma_e$ is taken to be a diagonal matrix with equal signal power, \ie, $\sigma_{e}^2$. For each $\sigma_{e}^2$ point, the thresholds $Th^s$ and $Th^d$ are derived using value iterations from dynamic programming. Then, the ADD and FAR are estimated by MC simulations using the derived thresholds. As discussed before, higher $\lambda_e$ reduces the usage of watermarking before the attack by increasing the threshold $Th^s$. The derived approximate expression of ADD (\ref{eqn:ADD_2nd}) reveals that the ADD does not depend on $\lambda_e$ or $Th^s$ directly. However, from (\ref{eqn:ln_2nd}) and (\ref{eqn:lbar}), we can say that the ${\bar l}$ reduces with the reduction in watermarking, which in turn increases ADD. Since ${\bar l}$ is a small quantity compared to $Th^D$, the effect of the change of ${\bar l}$ is small on ADD. To summarize, lower  $\lambda_e$ results in slightly lower ADD. Similarly, the derived approximate expression of FAR (\ref{eqn:FAR}) reveals that the FAR does not depend on $\lambda_e$ or $Th^s$ also. That is why we observe very similar FAR curves for two different values of $\lambda_e$ in Fig.~\ref{fig:ADD_lambda_e}. 

%A higher $\lambda_e$ results in reduced use of watermarking before the attack start point which in turn increases the ADD. However, $\lambda_e$ does not affect the FAR much. 
\begin{figure}[h!]
	\centering
	\includegraphics[width=\figwidth]{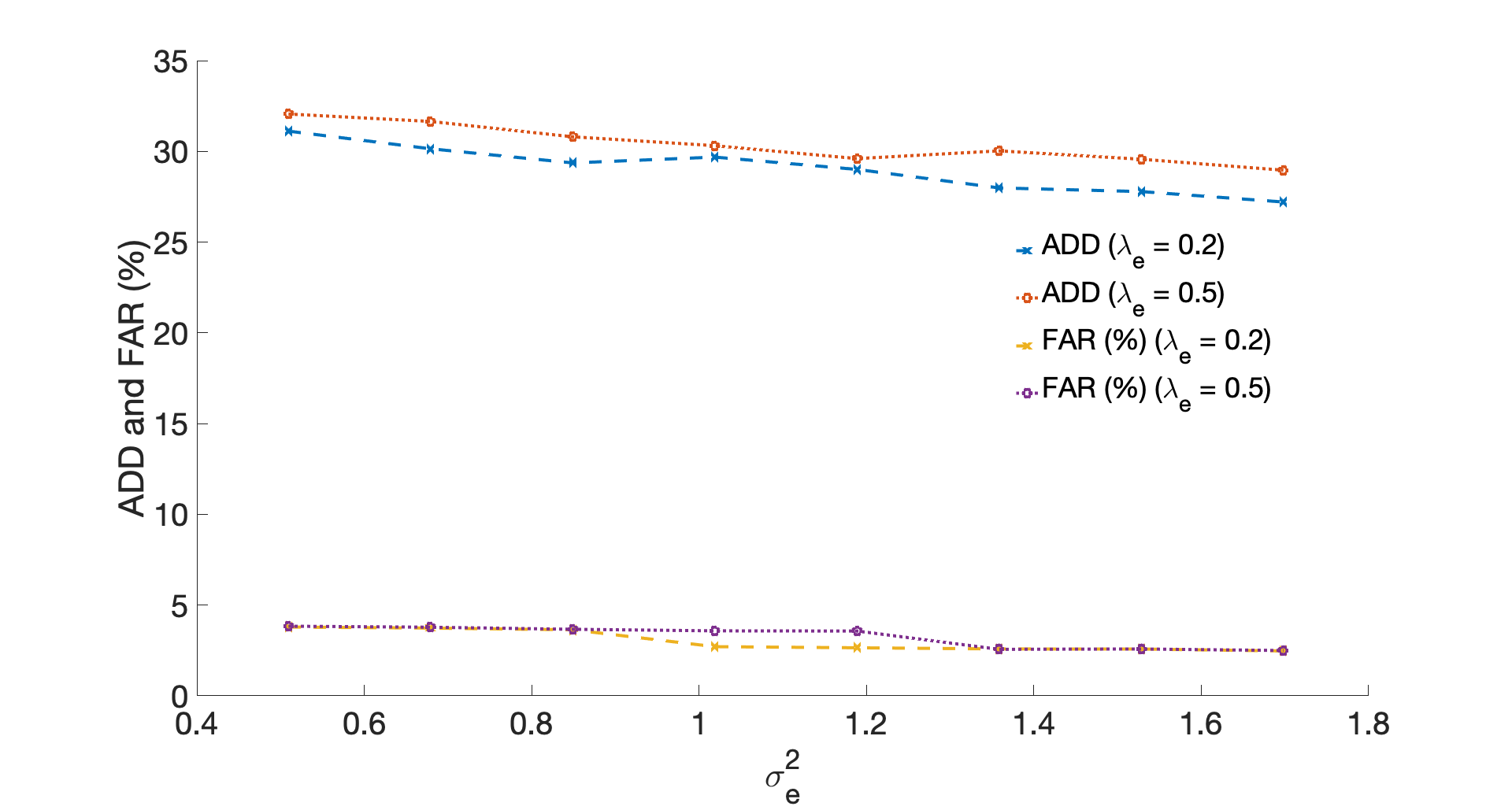}
	\caption{ADD and FAR (\%) vs. $\sigma_{e}^2$ plots for two different $\lambda_e$ and  $\lambda_f = 100$, for System-A.}
	\label{fig:ADD_lambda_e}
\end{figure}

Figure~\ref{fig:ADD_lambda_f} compares the same set of plots as in Figure~\ref{fig:ADD_lambda_e}, but for two different values of $\lambda_f$ and a fixed $\lambda_e$ for System-A. As discussed before, the increase in $\lambda_f$ increases $Th^d$. From (\ref{eqn:ADD_2nd}) and (\ref{eqn:FAR}), we know that ADD and FAR are mainly dependent on the value of $Th^D$. ADD increases with the increase in $Th^D$, whereas FAR reduces. To summarize, ADD increases and FAR decreases with $\lambda_f$.

In both the figures, Fig.~\ref{fig:ADD_lambda_e} and Fig.~\ref{fig:ADD_lambda_f}, ADD reduces with the increase in the watermarking signal power, which is primarily the result of increased KLD (\ref{eqn:opt_kld}). On the other hand, FAR (\ref{eqn:FAR}) does not reduce much with the watermarking signal power since the correlation is weak. Higher watermarking signal power increases the overshoot $r_{n_d}$ to some extent, which in turn reduces $\xi$ (\ref{eqn:Xi}) slightly.

%A higher $\lambda_f$ increases the threshold $Th^d$ for attack detections which in turn increases the ADD and reduces the FAR. 
\begin{figure}[h!]
	\centering
	\includegraphics[width=\figwidth]{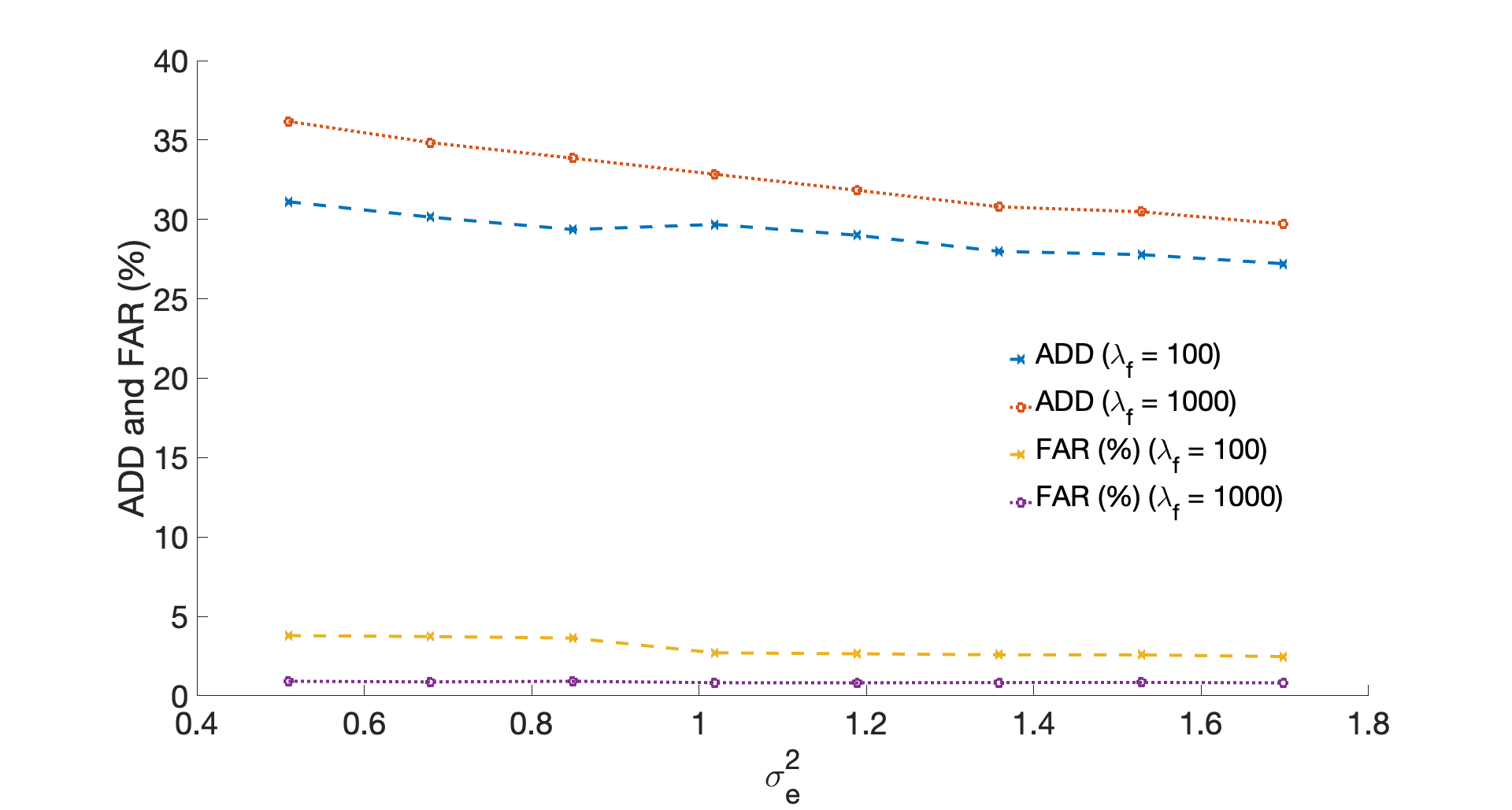}
	\caption{ADD and FAR (\%) vs. $\sigma_{e}^2$ plots for two different $\lambda_f$ and  $\lambda_e = 0.2$, for System-A.}
	\label{fig:ADD_lambda_f}
\end{figure}

\subsubsection{ADD, FAR and $\Delta LQG$ theoretical values}
Figure~\ref{fig:ADD_theory} shows the ADD and FAR vs. $\sigma_{e}^2$ plots for System-A, where ADD and FAR are estimated by MC simulations and also derived using Theorem~\ref{th:ADD_FAR} and Corollary~\ref{corr:ADD_1st}. The watermarking signal variance is taken to be a diagonal matrix with equal signal power, \ie, $\sigma_{e}^2$. For each $\sigma_{e}^2$ point, the thresholds $Th^s$ and $Th^d$ are derived using dynamic programming value iterations. Figure~\ref{fig:ADD_theory_system_b} shows the same set of plots as in Fig.~\ref{fig:ADD_theory} for System-B. The ADD derived using MC simulations does not reduce at the same rate as that of the approximate theoretical ADD with the increase in $\sigma_{e}^2$. {\color{black} The reason is that the derived analytical expression of ADD is asymptotically approximate. On the other hand, we have selected the parameter values for the MC simulations so that ADD remains small for the ease of simulation studies. Within the small delay window after the change point for the MC simulations, the increase in $\sigma_{e}^2$ is not making much difference to the estimated ADD. On the other hand, the simulation study shows that $\xi$ (\ref{eqn:Xi}) does not change much for a small increase in $\sigma_e^2$. Therefore, from (\ref{eqn:FAR}), we can say FAR will only be affected to a small extent due to the increase in $\sigma_e^2$. Therefore, we observe that the simulated FAR and the theoretical FAR are in close agreement in Fig.~\ref{fig:ADD_theory}. We also see that the derived ADD from Theorem~\ref{th:ADD_FAR} is a better match compared to the ADD derived from Corollary~\ref{corr:ADD_1st}. }

\begin{figure}[h!]
	\centering
	\includegraphics[width=\figwidth]{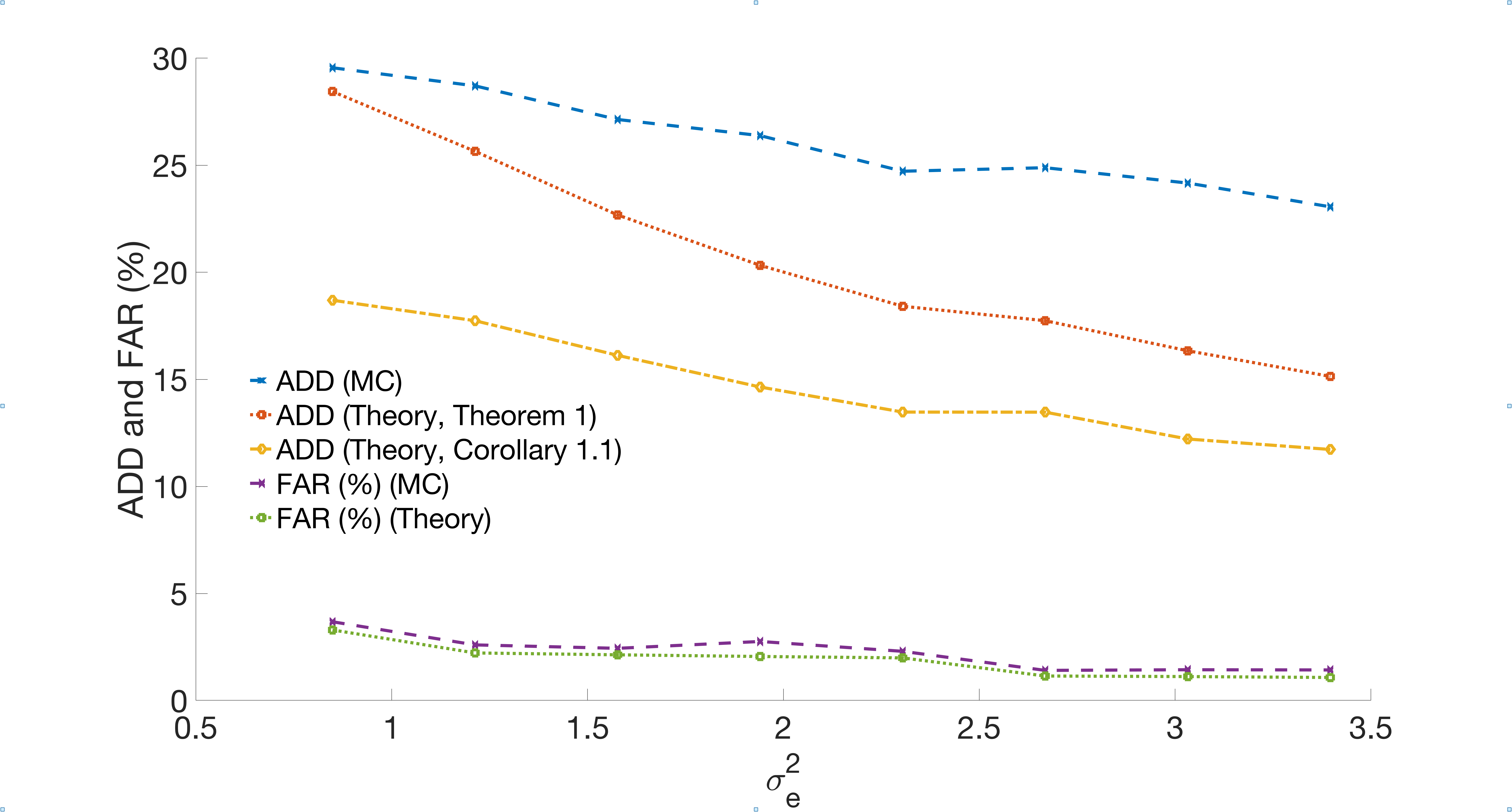}
	\caption{Comparison between the estimated values and theoretical values. ADD and FAR (\%) vs. $\sigma_{e}^2$ plot for System-A. $\lambda_f=100$ and $\lambda_e = 0.2$.}
	\label{fig:ADD_theory}
\end{figure}

\begin{figure}[h!]
	\centering
	\includegraphics[width=70mm]{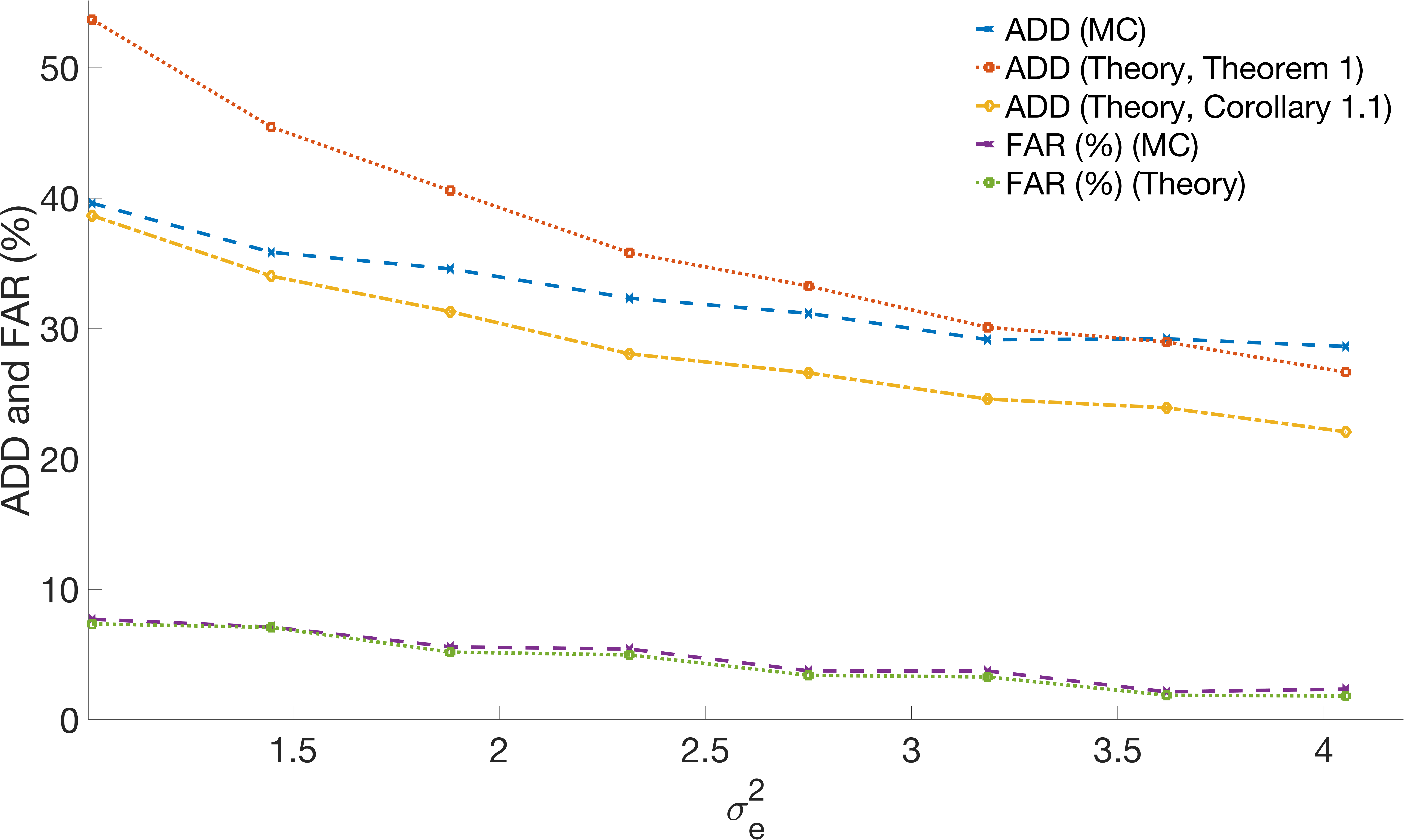}
	\caption{Comparison between the estimated values and theoretical values. ADD and FAR (\%) vs. $\sigma_{e}^2$ plot for System-B. $\lambda_f=100$ and $\lambda_e = 0.2$.}
	\label{fig:ADD_theory_system_b}
\end{figure}

Figure~\ref{fig:Dlqg_theory} shows the $\Delta LQG$ vs. $\sigma_{e}^2$ plot, where $\Delta LQG$ is estimated by MC simulation and also derived using the theory presented in this paper for System-A using the same parameters as Fig.~\ref{fig:ADD_theory}. The watermarking signal variance is taken to be a diagonal matrix with equal signal power, \ie, $\sigma_{e}^2$. From the derived expression of $\Delta LQG$ (\ref{eqn:DeltaLQG}), it is evident that the control cost will increase with the increase in watermarking signal power.

\begin{figure}[h!]
	\centering
	\includegraphics[width=\figwidth]{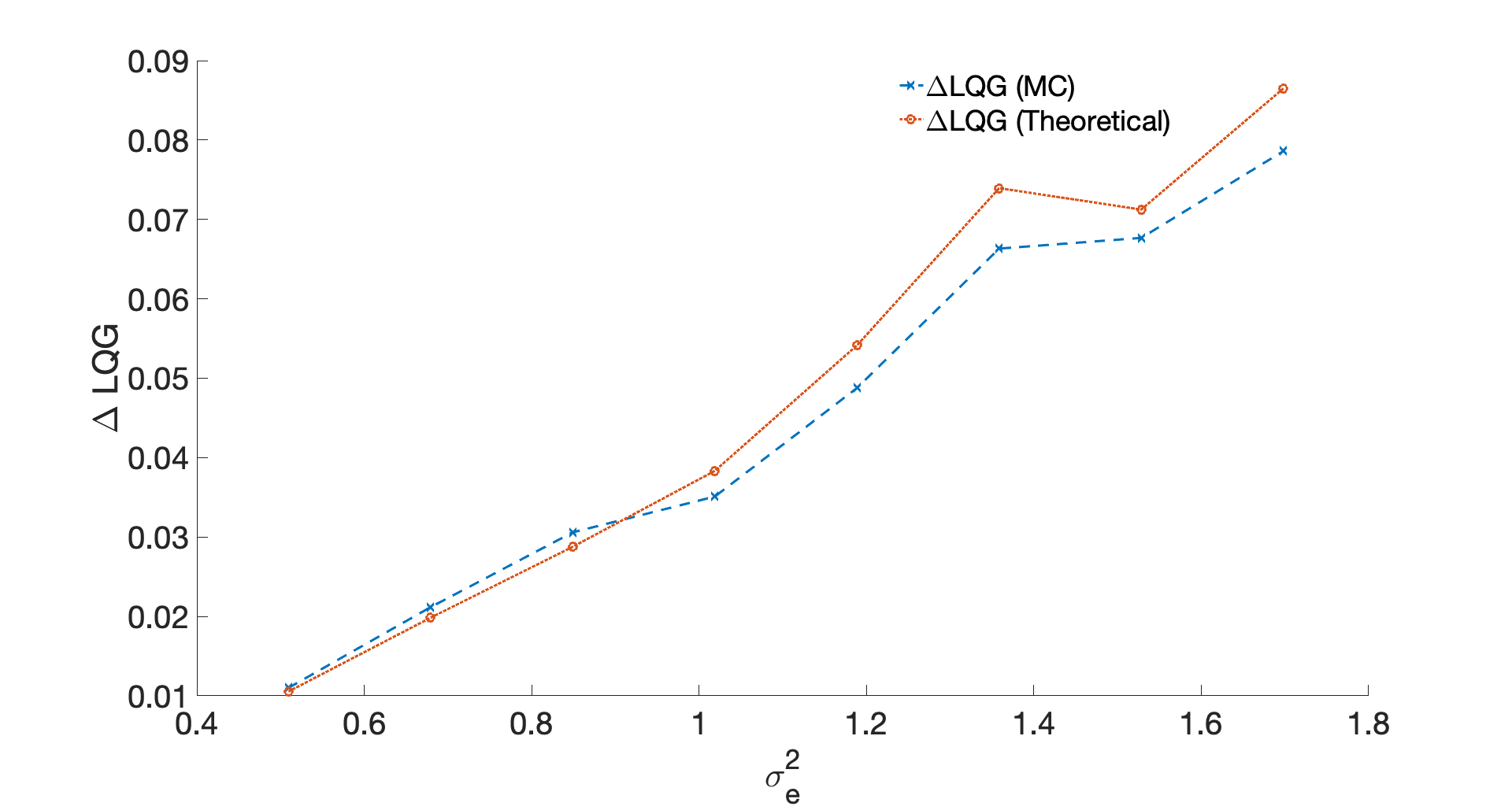}
	\caption{Comparison between the estimated values and theoretical values. $\Delta LQG$ vs. $\sigma_{e}^2$ plot for System-A. $\lambda_f=100$ and $\lambda_e = 0.2$.}
	\label{fig:Dlqg_theory}
\end{figure}

\subsubsection{Comparison with PW-$\Sigma_e$}
Figure~\ref{fig:deltaLQG_A} compares the $\Delta LQG$ vs $\sigma_{e}^2$ plot from the proposed method and PW-$\Sigma_e$ assuming a diagonal $\Sigma_e$ with equal power, $\sigma^2_e$, for System-A. For each $\sigma_{e}^2$ point, the thresholds $Th^s$ and $Th^d$ are derived using dynamic programming value iterations for the proposed method, and the same thresholds are used for PW-$\Sigma_e$ for a fair comparison.  From the derived expression of $\widetilde {\Delta LQG} $ (\ref{eqn:diff_deltaLQG}), we predicted that we would get a large improvement in the control cost since $\rho$ and ANW both are small quantities. Also, the difference will increase with $\sigma_e^2$ as $\Delta LQG_A$ increases with $\Sigma_e$. As predicted from the theory discussed in Sub-section~\ref{subsubsec:methodA}, we observe a large improvement in the control cost (approx. 99\% reduction in $\Delta LQG$) for the proposed method in Fig.~\ref{fig:deltaLQG_A}, which validates our Claim~\ref{clm:LQG_A}.

\begin{figure}[h!]
	\centering
	\includegraphics[width=\figwidth]{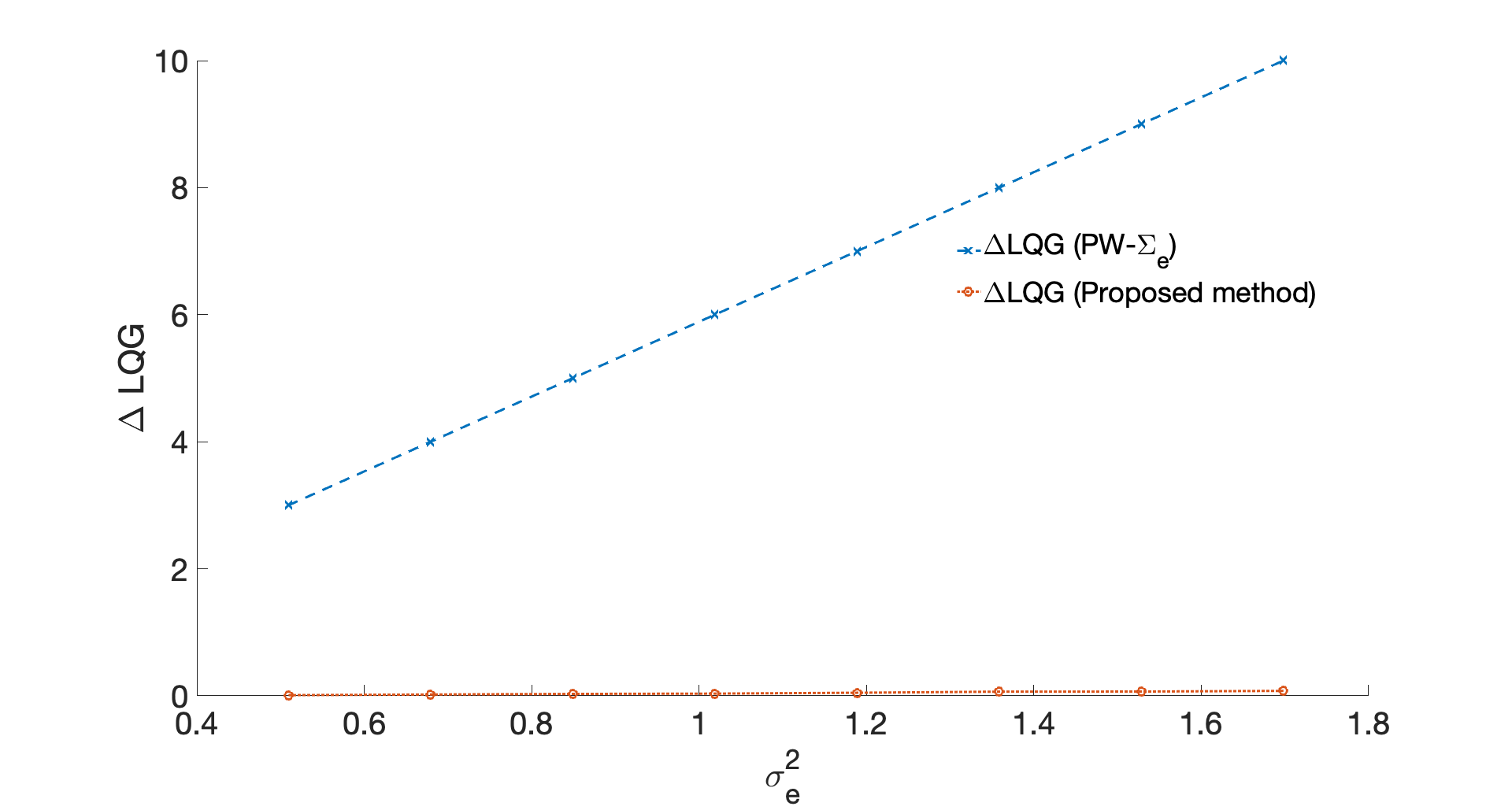}
	\caption{Comparison between proposed method and PW-$\Sigma_e$. $\Delta LQG$ vs. $\sigma_{e}^2$ plot for System-A. $\lambda_f=100$ and  $\lambda_e = 0.2$.}
	\label{fig:deltaLQG_A}
\end{figure}

We have shown ADD and FAR vs. $\sigma_{e}^2$ plots for the proposed method and PW-$\Sigma_e$ using the same parameters as Fig.~\ref{fig:deltaLQG_A} in Fig.~\ref{fig:ADD_A}. From the derived expression of $\Delta ADD$ (\ref{eqn:deltaADD_A}), we can comment that the proposed method will take a longer time on average to detect the attack compared to PW-$\Sigma_e$. The difference is due to the slowly changing terms, ${\bar l}_A$ and  ${\bar l}_P$. Since the magnitude of the slowly changing term usually remains small, the increase in ADD for the proposed method is also small. In Fig.~\ref{fig:ADD_A}, an average increase of 35\% (approx.) in ADD is observed at the same FAR for the proposed method. On the other hand,  as discussed in Sub-section~\ref{subsubsec:methodA}, FAR will be the same for both the methods and the same is observed in Fig.~\ref{fig:ADD_A}. To summarize, Fig.~\ref{fig:ADD_A} supports our Claim~\ref{clm:ADD_A} and Claim~\ref{clm:FAR_A}.

\begin{figure}[h!]
	\centering
	\includegraphics[width=\figwidth]{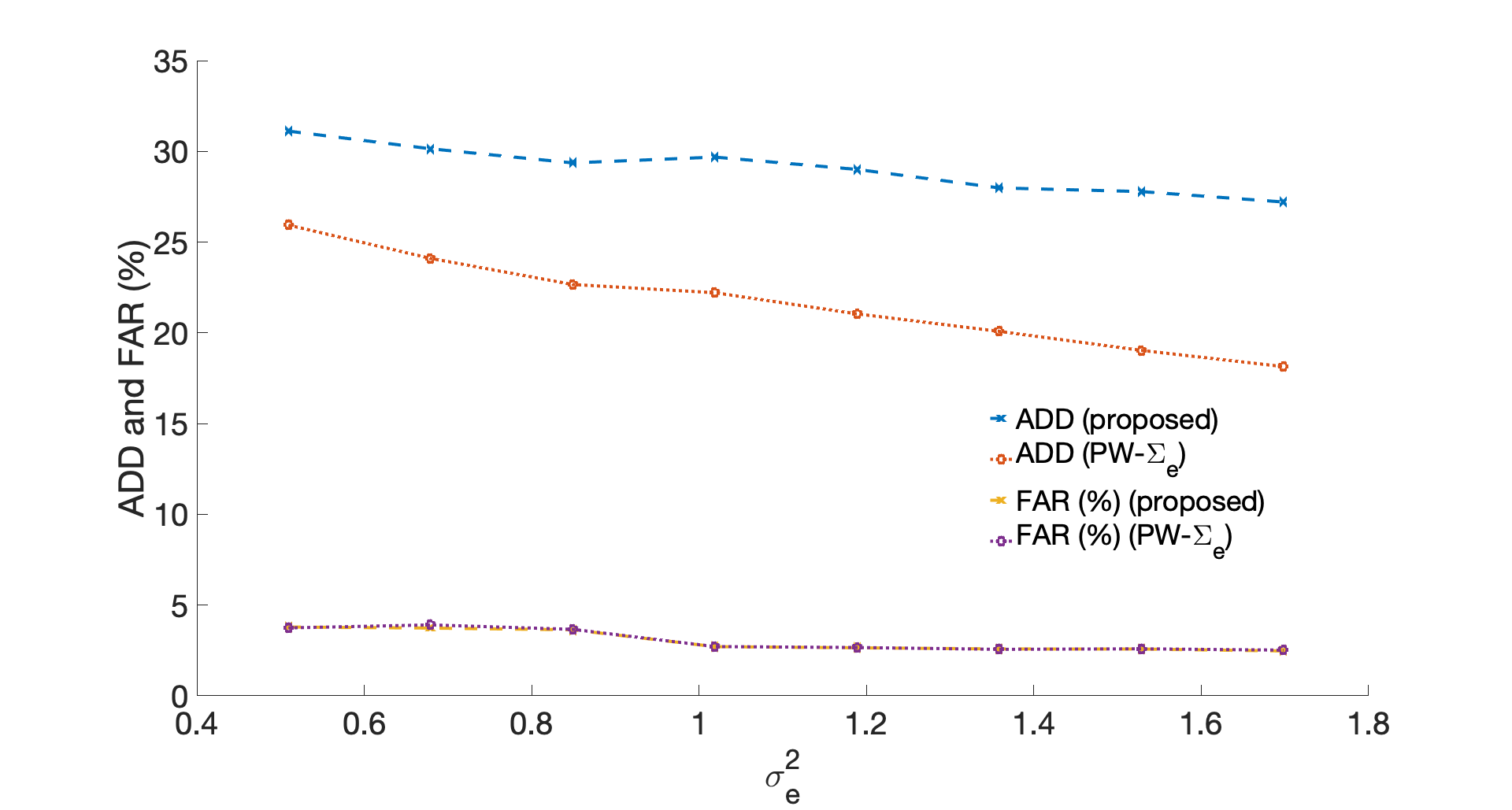}
	\caption{Comparison between proposed method and PW-$\Sigma_e$. ADD and FAR vs. $\sigma_{e}^2$ plot for System-A. $\lambda_f=100$ and  $\lambda_e = 0.2$.}
	\label{fig:ADD_A}
\end{figure}

\subsubsection{Comparison with PW-$\Delta LQG$}
We have shown ADD and FAR vs. $\Delta LQG$ plots derived from MC simulations for the proposed method and PW-$\Delta LQG$ assuming a diagonal $\Sigma_e$ with equal power $\sigma_e^2$ in Fig.~\ref{fig:ADD_B} for System-A. For each $\Delta LQG$ point, the thresholds $Th^s$ and $Th^d$ are derived using dynamic programming value iterations for the proposed method, and the same thresholds are used for PW-$\Delta LQG$ for a fair comparison. {\color{black} In general, for the proposed method, $Th^s$ decreases and $Th^d$ increases with the increase in $\Delta LQG$ or $\sigma^2_e$ for fixed $\lambda_e$ and $\lambda_f$. Since the same thresholds are used for PW-$\Delta LQG$, the ADD increases with $\Delta LQG$ in the plot. }As discussed in Sub-section~\ref{subsubsec:methodB}, since the proposed method uses a higher watermarking signal variance at the same control cost, the KLD for the proposed method is higher compared to PW-$\Delta LQG$. Higher KLD for the proposed method results in lower ADD, and the same characteristic is observed in Fig.~\ref{fig:ADD_B}. Also, the usage of higher watermarking signal power increases the overshoot statistic to a small extent, resulting in a small decrease in FAR. To summarize, Fig.~\ref{fig:ADD_A} supports our Claim~\ref{clm:Sigma_e_B} and Claim~\ref{clm:ADD_B}.

\begin{figure}[h!]
	\centering
	\includegraphics[width=\figwidth]{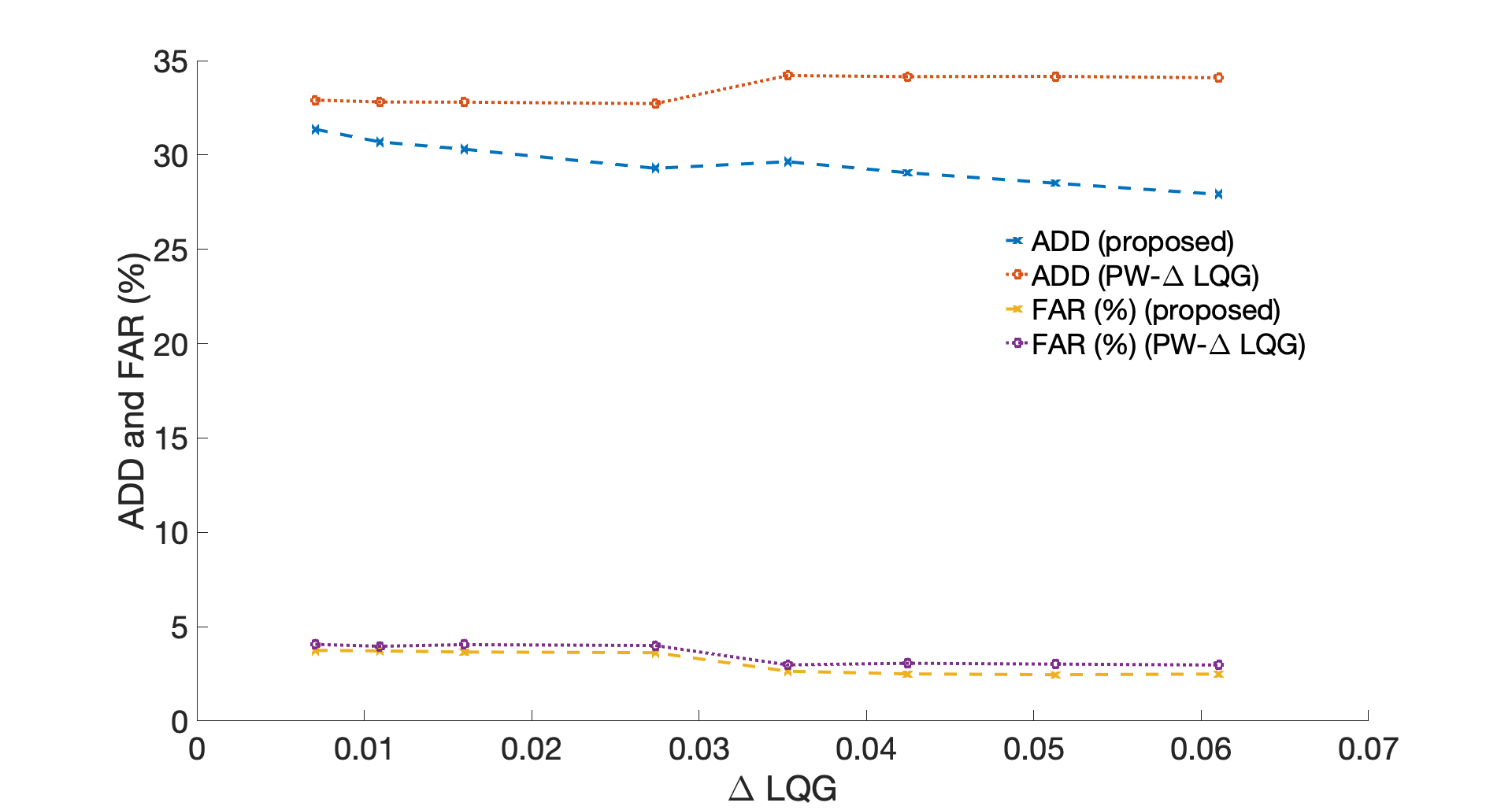}
	\caption{Comparison between proposed method and PW-$\Delta LQG$. ADD and FAR vs. $\Delta LQG$ plot for System-A. $\lambda_f=100$ and  $\lambda_e = 0.3$.}
	\label{fig:ADD_B}
\end{figure}

\subsubsection{Optimum $\Sigma_e$}
	As discussed in Sub-section~\ref{subsec:optimum_sigma_e}, the optimum $\Sigma^*_e$ reduces the KLD for a fixed upper bound on the $\Delta LQG$, which in turn reduces the ADD. We compare the ADD for the optimum $\Sigma_e^*$ and the diagonal $\Sigma_e$ in Fig.~\ref{fig:opt_e} for System-A. For each $\sigma_{e}^2$ point, the thresholds $Th^s$ and $Th^d$ are derived using dynamic programming value iterations for the diagonal $\Sigma_e$ case, and the same thresholds are used for the optimum $\Sigma_e^*$ case for a fair comparison. We observe an average increase of 14\% (approx.) in the estimated ADD for the optimal $\Sigma^*_e$. For the optimal $\Sigma^*_e$, the watermarking signal power is mostly concentrated in one eigenvector direction, which results in higher overshoot and a lower FAR. 
	
	\begin{figure}[h!]
		\centering
		\includegraphics[width=\figwidth]{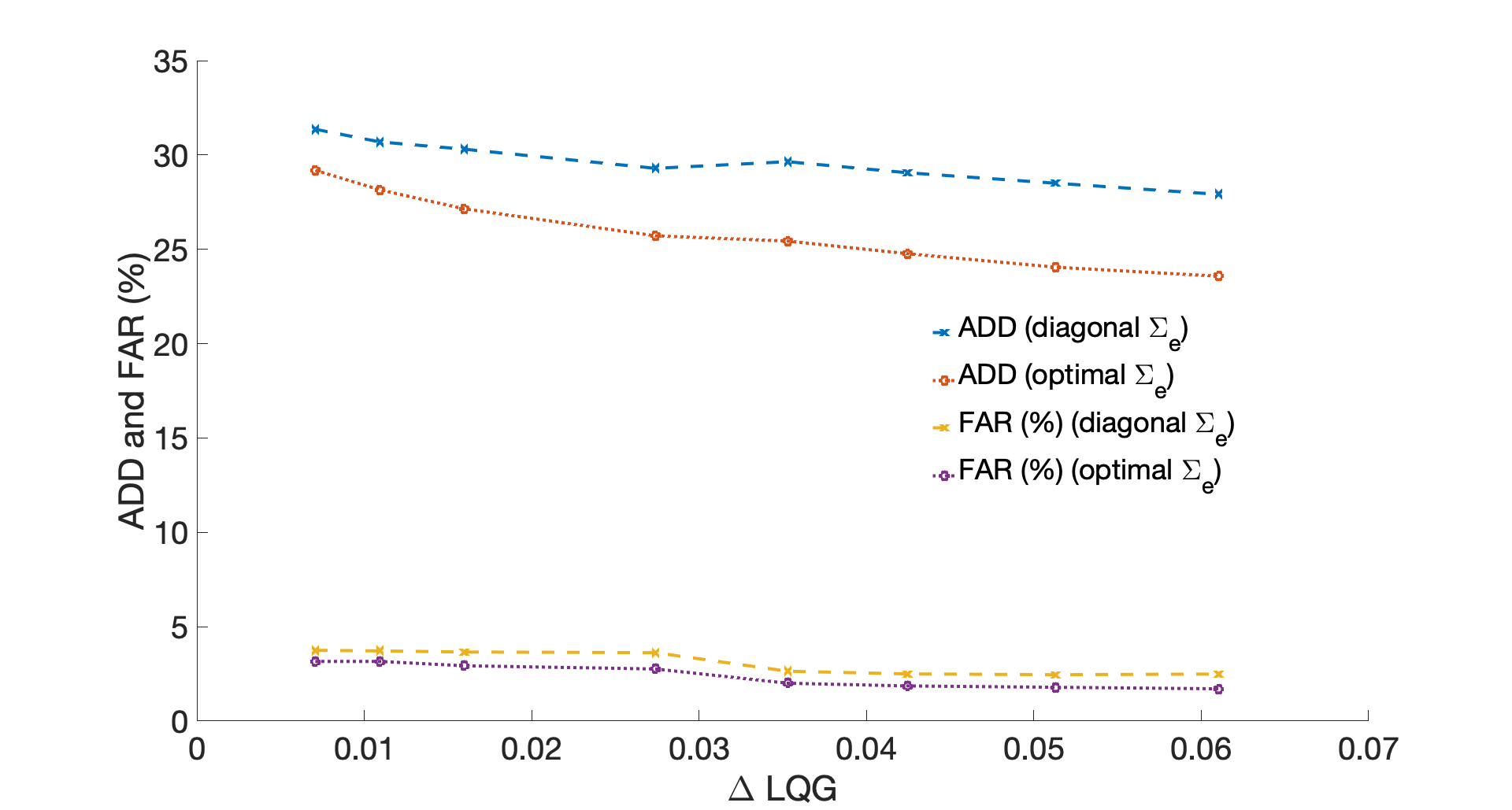}
		\caption{Comparison between diagonal $\Sigma_e$ and optimal $\Sigma^*_e$. ADD and FAR vs. $\Delta LQG$ plot for System-A. $\lambda_f=100$ and  $\lambda_e = 0.3$.}
		\label{fig:opt_e}
	\end{figure}

\subsubsection{Comparison with a periodic watermarking scheme} \label{subsec:compare_peroidic}
We have compared the proposed evidence-based parsimonious watermarking scheme with a periodic watermarking scheme. The periodic watermarking scheme is adopted from \cite{Fang2020} for our problem formulation. To fairly compare both methods, we have evaluated ADD and FAR by MC simulations for the same $\Delta LQG$ values. Under both schemes, the $p_k$ has been evaluated and compared with the same $Th^d$ value for attack detections. Note that $Th^d$ values are different for different $\Delta LQG$ values. However, watermarking has been added under the proposed scheme if $p_k \ge Th^s$. On the other hand, watermarking is added only once in a period for the other method, and the periods are determined separately for each $\Delta LQG$ value. Since the periodic watermarking scheme does use any existing evidence extracted from the set $\Psi_k$ of all available information upto the $k$-th time instant, the watermarking frequency remains the same before and after the attack. However, for the proposed scheme, the watermarking frequency increases significantly after the attack (approx. 50 times), which reduces ADD and FAR, see Fig.~\ref{fig:ADD_periodic}.
\begin{figure}[h!]
	\centering
	\includegraphics[width=\figwidth]{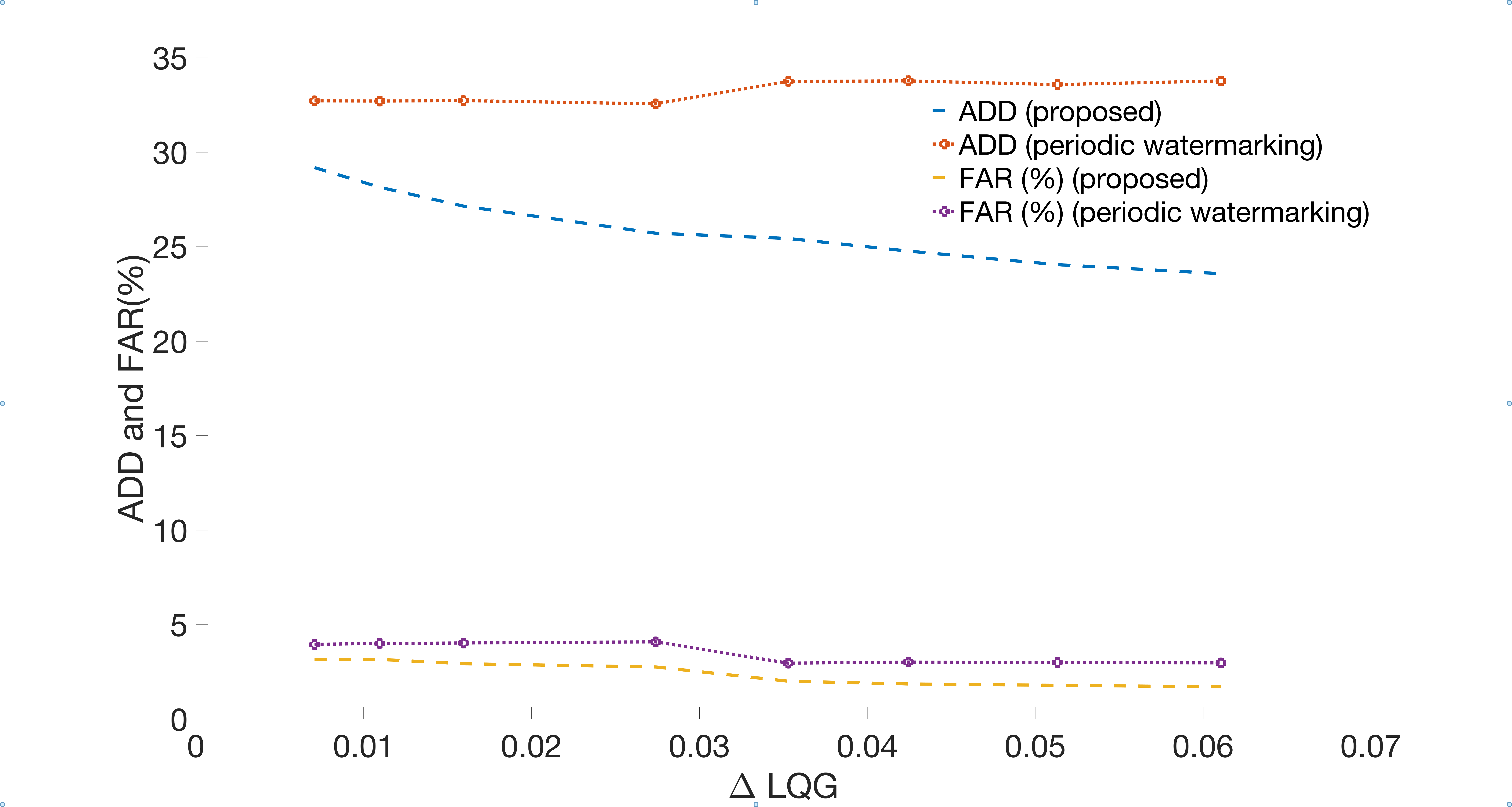}
	\caption{Comparison between proposed method and a periodic watermarking scheme. ADD and FAR vs. $\Delta LQG$ plot for System-A. $\lambda_f=100$ and  $\lambda_e = 0.3$.}
	\label{fig:ADD_periodic}
\end{figure}

\section{Conclusion}
\label{sec:conclusion}
In this paper, we have studied the quickest data deception attack detection problem with constraints on FAR and ANW. Such parsimonious use of watermarking helps to reduce the control cost during normal system operations and maintain a moderate detection performance. First, we have formulated the problem as a stochastic optimal control problem under a Bayesian framework. Then, we have applied dynamic programming to find the optional policy. We have studied the optimal policy structure and found the optimal policy to be a two threshold policy on the posterior probability of attack under a few practical assumptions. We have also derived the asymptotic approximate expressions of ADD and FAR applying non-linear renewal theory. The analytical expression of $\Delta LQG$ and its relationship with ANW is also derived. Theoretical and simulation studies reveal significant improvement in reducing $\Delta LQG$ with a relatively small increase in ADD compared to the method PW-$\Sigma_e$, where watermarking is always present. The proposed method is also compared with PW-$\Delta LQG$, where both the methods have the same $\Delta LQG$ limit and found that the proposed method performs better in terms of ADD and FAR. Furthermore, we have described a technique to find the optimal watermarking signal power that maximises the KLD, which will improve the ADD.

\appendices
\section{Proof of Lemma~\ref{lemma:SRn}} \label{apdx:SRn}
The likelihood ratio, $\mathcal{L}_{a,j}$, of the joint dependent distributions of the innovation signal and the watermarking signal, after and before the attack, takes the following form, 
\begin{equation}
	\mathcal{L}_{a,j} = \frac{\mathpzc{f}_{1,j}\left(\gamma_j,{\bf e}_{s,j-1}|\left\{\gamma\right\}_1^{j-1},\left\{{\bf e}_s\right\}_1^{j-2} \right)}{\mathpzc{f}_{0,j}\left(\gamma_j,{\bf e}_{s,j-1}|\left\{\gamma\right\}_1^{j-1},\left\{{\bf e}_s\right\}_1^{j-2}\right)}.
	\label{eqn:Laj_apdx}
\end{equation}
$\gamma_j$ is iid before the attack. Therefore, applying the chain rule to $\mathpzc{f}_{1,j}(\cdot|\cdot)$, (\ref{eqn:Laj_apdx}) can be written as given in (\ref{eqn:Laj}). Using a similar argument, (\ref{eqn:Lbj}) can be derived. 

To derive (\ref{eqn:f1j}), (\ref{eqn:mu_1j}) and (\ref{eqn:sigma_1j}), we rewrite (\ref{eqn:gamma_attack_vs_e}) for the case, $k > \Gamma$, applying (\ref{eqn:hidden_states_main}) as 
\begin{equation}
	\widetilde\gamma_j ={\bf w}_{a,j-1}+{\bf A}_a{\bf z}_{j-1}-{\bf C}\left({\bf A}+{\bf B}{\bf L}\right){\bf \hat x}_{j-1|j-1}-{\bf C}{\bf B}{\bf e}_{s,j-1}.
	\label{eqn:gamma_attack_apndx}
\end{equation}
Using the recursive state equations and the Kalman time update equation, (\ref{eqn:gamma_attack_apndx}) can be written as
\begin{align}
	&\gamma_j ={\bf w}_{a,j-1}+\left({\bf A}_a{\bf C}-{\bf C}\left({\bf A}+{\bf B}{\bf L}\right)\right)\left(\left({\bf A}+{\bf B}{\bf L}\right)^{j-2}{\hat{\bf{x}}}_{1|1} \right. \nonumber \\
	& \left. + \sum_{r=1}^{j-2}\left({\bf A}+{\bf B}{\bf L}\right)^{r-1}{\bf B}{\bf e}_{s,j-r-1}+\sum_{r=2}^{j-2}\left({\bf A}+{\bf B}{\bf L}\right)^{r-1}{\bf K}{ \gamma}_{j-r}\right) \nonumber \\ 
	&-{\bf C}{\bf B}{\bf e}_{s,j-1}+\left({\bf A}_a-{\bf C}\left({\bf A}+{\bf B}{\bf L}\right){\bf K}\right){{\bf \gamma}}_{j-1} \label{eqn:gamma_attack_apndx_2}.
\end{align}
From the assumptions that $\left({\bf A}+{\bf B}{\bf L}\right)$ is strictly stable and the system started at $j=-\infty$, we can say $\left({\bf A}+{\bf B}{\bf L}\right)^{j-2} \rightarrow {\bf 0}$ as $j \rightarrow \infty$. Therefore, (\ref{eqn:gamma_attack_apndx_2}) will take the following form as $j \rightarrow \infty$, 
 \begin{align}
 	&\gamma_j ={\bf w}_{a,j-1}+\left({\bf A}_a{\bf C}-{\bf C}\left({\bf A}+{\bf B}{\bf L}\right)\right) \nonumber \\
 	& \left( \sum_{r=1}^{j-2}\left({\bf A}+{\bf B}{\bf L}\right)^{r-1}{\bf B}{\bf e}_{s,j-r-1}+\sum_{r=2}^{j-2}\left({\bf A}+{\bf B}{\bf L}\right)^{r-1}{\bf K}{ \gamma}_{j-r}\right) \nonumber \\ 
 	&-{\bf C}{\bf B}{\bf e}_{s,j-1}+\left({\bf A}_a-{\bf C}\left({\bf A}+{\bf B}{\bf L}\right){\bf K}\right){{\bf \gamma}}_{j-1} \label{eqn:gamma_attack_apndx_3}.
 \end{align}
Using (\ref{eqn:gamma_attack_apndx}) and (\ref{eqn:gamma_attack_apndx_3}), we can derive the following, 
\begin{align}
	&{\bf \mu}_{1,j}=\text{E}_1\left[{ {\bf \gamma}_j}|\left\{\gamma\right\}_1^{j-1},\left\{{\bf e}_s\right\}_1^{j-1}\right] \nonumber \\
	&=\text{E}_1\left[{ {\bf \gamma}_j}|{\bf z}_{j-1},{\bf {\hat x}}_{j-1|j-1},{\bf e}_{s,j-1}\right] \cr
	&={\bf A}_a{\bf z}_{j-1}-{\bf C}\left({\bf A}+ {\bf B}{\bf L} \right){\bf {\hat x}}_{j-1|j-1} - {\bf C}{\bf B}{\bf e}_{s,j-1}, \text{and} \\ 
	&\Sigma_{1,j}=cov\left({ {\bf \gamma}_j}|{\bf z}_{j-1},{\bf {\hat x}}_{j-1|j-1},{\bf e}_{s,j-1}\right)={\bf Q}_a.
\end{align}
Using the same approach, (\ref{eqn:f2j}), (\ref{eqn:mu_2j}) and (\ref{eqn:sigma_2j}) can be derived from (\ref{eqn:gamma_attack_vs_e}) for the case, $k = \Gamma$. Taking expectations on both sides of (\ref{eqn:gamma_vs_e}), we get $\text{E}_0\left[ \gamma_j\right]=\bf 0$. (\ref{eqn:sigma_0j}) is derived from (\ref{eqn:gamma_normal}) as, 
\begin{align}
	\gamma_j&={\bf y}_j-{\bf C}{\hat {\bf x}}_{j|j-1}={\bf C}\left({\bf x}_j-{\hat {\bf x}}_{j|j-1}\right)+{\bf v}_j \text{, and} \nonumber \\
	{ \Sigma}_0 &= \text{E}_0\left[\gamma_j\gamma_j^T\right] ={\bf C}{\bf P}{\bf C}^T+{\bf R}. \label{eqn:Sigma_gamma_apn_corr1p1}
\end{align}

\section{Proof of Lemma~\ref{lemma:LSRn}} \label{apdx:LSRn}
The following form of the $LSR_k$ is derived by taking logarithms on both sides of (\ref{eqn:SRn_recurr}), and combining both the conditions in (\ref{eqn:SRn_recurr}) using an indicator function. 
\begin{align}
		&LSR_k=\lambda_k+k|\log(1-\rho)| \nonumber \\
		&+\log\left(LSR_0+\sum_{i=1}^k(1-\rho)^{i-1}\mathcal{L}_{d,i}\exp(-\lambda_i)\mathbbm{1}_{\left\{LSR_i < Th^S \right\}} \right. \nonumber \\
		& \left. +\sum_{i=1}^k(1-\rho)^{i-1}\mathcal{L}_{b,i}\exp(-\lambda_i)\mathbbm{1}_{\left\{LSR_i \ge Th^S \right\}}\right),
\label{eqn:LSRn_apdx}
\end{align}
where
\begin{equation}
	\lambda_k=\sum_{i=1}^k\log\left(\mathcal{L}_{a,i} \right)\mathbbm{1}_{\left\{LSR_i \ge Th^S \right\}}+\sum_{i=1}^k\log\left(\mathcal{L}_{c,i} \right)\mathbbm{1}_{\left\{LSR_i < Th^S \right\}}. \label{eqn:lambda_n_apdx} 
\end{equation}
The threshold $Th^S$ on $LSR_k$ is the same as the threshold $Th^s$ on $p_k$. $Th^S$ is derived directly from (\ref{eqn:pk_1st}) as given in (\ref{eqn:ThS}). We rewrite $\lambda_k$ by adding and subtracting $\sum_{i=1}^k\log\left(\mathcal{L}_{a,i} \right)\mathbbm{1}_{\left\{LSR_i < Th^S \right\}}$ to the right hand side of (\ref{eqn:lambda_n_apdx}) as follows,
\begin{align}
\lambda_k&=Z_k+\sum_{i=1}^k\log\left(\mathcal{L}_{c,i} \right)\mathbbm{1}_{\left\{LSR_i < Th^S \right\}} \nonumber \\ 
&-\sum_{i=1}^k\log\left(\mathcal{L}_{a,i} \right)\mathbbm{1}_{\left\{LSR_i < Th^S \right\}},
\label{eqn:lambda_n_2_apdx}
\end{align}
where $Z_k$ is given in (\ref{eqn:Zn}). Replacing the first $\lambda_k$ in (\ref{eqn:LSRn_apdx}) by (\ref{eqn:lambda_n_2_apdx}) and dividing the terms in $S_k$ and $l_k$, we get (\ref{eqn:LSRn_2nd}). The proof that $l_k$ is slowly changing variable is provided as follows. \newline

The variable $l_k$ will be called slowly changing provided the following two conditions are satisfied, according to \cite{siegmund2013sequential}:
\begin{equation}
	\text{C1: } k^{-1}\max\left\{\mid l_1\mid, \cdots,\mid l_k\mid\right\}\rightarrow0 \text{, } k\rightarrow \infty
	\label{eqn:C1_apdx}
\end{equation}
and for every $\epsilon >0$, there exists $k^*$ and $\delta>0$, such that for all $k \ge k^*$
 \begin{equation}
 	\text{C2: } \text{P}\left\{\max_{1 \le i\le k\delta}\mid l_{k+i}-l_k\mid>\epsilon\right\}<\epsilon.
 	\label{eqn:C2_apdx}
 \end{equation}
$l_k$ from (\ref{eqn:ln_2nd}) is represented as the summation of three terms as follows,
\begin{align}
	&l_k = l_{1,k}+l_{2,k} - l_{3,k}, \label{eqn:ln_apdx} \\ 
	& \text{where} \nonumber \\
	& l_{1,k}= \log\left(LSR_0+\sum_{i=1}^k(1-\rho)^{i-1}\mathcal{L}_{d,i}\exp(-\lambda_i)\mathbbm{1}_{\left\{LSR_i < Th^S \right\}} \right. \nonumber \\
	& \left. +\sum_{i=1}^k(1-\rho)^{i-1}\mathcal{L}_{b,i}\exp(-\lambda_i)\mathbbm{1}_{\left\{LSR_i \ge Th^S \right\}}\right) \label{eqn:l1_apdx},  \\
	& l_{2,k}=\sum_{i=1}^k\log\left(\mathcal{L}_{c,i}\right)\mathbbm{1}_{\left\{LSR_i < Th^S \right\}}, \text{ and} \label{eqn:l2_apdx}  \\
	& l_{3,k}=\sum_{i=1}^k\log\left(\mathcal{L}_{a,i}\right)\mathbbm{1}_{\left\{LSR_i < Th^S \right\}}. \label{eqn:l3_apdx} 
\end{align}
Taking absolute values on both sides of (\ref{eqn:ln_apdx}), we can write, 
\begin{equation}
	\mid l_k\mid \le \mid l_{1,k}\mid+\mid l_{2,k}\mid +\mid l_{3,k}\mid. 
	\label{eqn:ln_abs_apdx}
\end{equation}
After the attack start point, $LSR_k$ will gradually increase on average (from condition $C6$ in Theorem~\ref{th:ADD_FAR}), and it will first cross $Th^S$ and then $Th^D$ as $k \rightarrow \infty$. $LSR_k$ will remain below $Th^S$ for a relatively short period of time compared to the time it takes to cross $Th^D$, since $Th^D \rightarrow \infty$. Therefore, $l_{2,k}$ and $l_{3,k}$ will converge to some finite values, say $L_2$ and $L_3$, respectively, as $k \rightarrow \infty$. Also, $\exp\left(-\lambda_k \right) \rightarrow 0$ since $\lambda_k \rightarrow \infty$ as $k\rightarrow \infty$ from condition $C6$. Therefore, $l_{1,k}$ will also converge to a finite value, say $L_1$, as $k \rightarrow \infty$. Now, from (\ref{eqn:ln_abs_apdx}), we can say $l_k$ will also converge to a finite value $l$, \ie, $l \le L1+L2+L3$ as $k \rightarrow \infty$, which means $l_k$ will satisfy condition C1.

We assume that at $k = k_1$, $LSR_k$ crosses $Th^S$. Therefore, for $k \ge k_1$, we can write
\begin{align}
&l_{2,k+i} = l_{2,k} = \sum_{i=1}^{k_1}\log\left(\mathcal{L}_{c,i}\right)\mathbbm{1}_{\left\{LSR_i < Th^S \right\}}, \\
&l_{3,k+i} = l_{3,k} = \sum_{i=1}^{k_1}\log\left(\mathcal{L}_{a,i}\right)\mathbbm{1}_{\left\{LSR_i < Th^S \right\}}, \\
& l_{1,k+i}= \log\left(LSR_0+\sum_{i=1}^{k_1}(1-\rho)^{i-1}\mathcal{L}_{d,i}\exp(-\lambda_i)\times \right. \nonumber \\
& \left.\mathbbm{1}_{\left\{LSR_i < Th^S \right\}} +\sum_{j=1}^{k+i}(1-\rho)^{j-1}\mathcal{L}_{b,j}\exp(-\lambda_j)\mathbbm{1}_{\left\{LSR_j \ge Th^S \right\}}\right), \label{eqn:l1_apdx_2} \\
& l_{1,k}= \log\left(LSR_0+\sum_{i=1}^{k_1}(1-\rho)^{i-1}\mathcal{L}_{d,i}\exp(-\lambda_i) \right. \nonumber \\
& \left. \mathbbm{1}_{\left\{LSR_i < Th^S \right\}}+\sum_{i=1}^{k}(1-\rho)^{i-1}\mathcal{L}_{b,i}\exp(-\lambda_i)\mathbbm{1}_{\left\{LSR_i \ge Th^S \right\}}\right) \label{eqn:l1_apdx_3}.\\
&\text{Therefore, }l_{k+i}- l_k =  l_{1,k+i} - l_{1,k} \text{ for } k \ge k_1.  \label{eqn:diff_ln_apdx}
\end{align}
As mentioned before, $\exp\left(-\lambda_k \right) \rightarrow 0$ as $k \rightarrow \infty$, therefore, we can say $\text{P}\left\{l_{1,k+i} - l_{1,k}\right\} \rightarrow 0$ for a sufficiently large $k$, say $k^*$, and $k^* \ge k_1$. From (\ref{eqn:diff_ln_apdx}), for $k \ge k^*$, $\text{P}\left\{\left | l_{k+i} - l_{k}\right|> \epsilon\right\} = 0$, which in turn will satisfy condition C2. 
	
\section{Proof of Theorem~\ref{th:ADD_FAR}} \label{apdx:ADD_FAR}
First, we will show that the conditions C1-C4 are satisfied for the problem under study. $Z_k$ is a function of continuous random variables, which take uncountably infinite values, so $Z_k$ is non-arithmetic, thus satisfies the condition C1. 
		
For condition C2, $Z_1$ denotes the log-likelihood ratio (\ref{eqn:Zn}) just after the attack start point. For simplicity, we consider that the attacker is present in the system from the beginning. Now, from (\ref{eqn:Zn}) and (\ref{eqn:Laj}), we can write $Z_1$ as 
		\begin{align}
			Z_1 =& -\frac{1}{2}\log \frac{\left |\Sigma_{1,1} \right|}{\left |\Sigma_{0} \right|}-\frac{1}{2}\left(\widetilde \gamma_1 - {\bf \mu}_{1,1} \right)^T\Sigma_{1,1}^{-1}\left(\widetilde \gamma_1 - {\bf \mu}_{1,1} \right) \nonumber \\ 
			&+\frac{1}{2}{\widetilde \gamma_1}^T\Sigma_{0}^{-1}{\widetilde \gamma_1}
			\label{eqn:Z1}
		\end{align}
From (\ref{eqn:mu_1j}), (\ref{eqn:sigma_1j}), and (\ref{eqn:sigma_0j}), we can say that all the elements of (\ref{eqn:Z1}) are either finite or having Gaussian distributions with finite means and variances, which in turn ensures that $\text{E}_1\left[\mid Z_1 \mid ^2 \right]$ is finite. 
	
Condition C3 is proven in Appendix~\ref{apdx:LSRn}. From the expressions of ${\bf \Sigma}_0$ (\ref{eqn:sigma_0j}) and ${\bf \Sigma}_{\widetilde \gamma}$ (\ref{eqn:Exz1_original}), we can say that under the practical assumptions of the plant model and attacker's system parameters $0 < \text{E}_1\left[\text{D}\left(\mathpzc{f}^e_{1},\mathpzc{f}_{0}\right) \right] <  \infty$ from (\ref{eqn:opt_kld}). In a similar way we can also show that $0< \text{E}_1\left[\text{D}\left(\mathpzc{f}_{0},\mathpzc{f}^e_{1}\right)\right]	< \infty$. Therefore, the condition C4 is valid for the problem under study. The following is the proof of Theorem~\ref{th:ADD_FAR}.

Say, after the attack start point, at $k=\Gamma$ the test statistics $LSR_k$ will cross the threshold $Th^D$ at $k=\tau$ for the first time which is equivalent to the test statistics $p_k$ crossing the threshold $Th^d$. To derive the expression of ADD, we assume, $T_D = \tau - \Gamma$. After adding and subtracting $Th^D$ to (\ref{eqn:LSRn_2nd}) and rearranging the terms, it will take the following form at $k=T_D$,   %$n = k-\Gamma$ for $k \ge \Gamma$ and
	\begin{equation}
	S_{T_D}  = Th^D + \left(LSR_{T_D} - Th^D \right) - l_{T_D} .
	\label{eqn:LSR_TD_apdx}
\end{equation}
According to the nonlinear renewal theory \cite{siegmund2013sequential}, the overshoot statistics of $LSR_{T_D} - Th^D$ can be approximated by the overshoot statistics of $S_{T_D}$, \ie, $r_{n_d} = S_{T_D} - Th^D$, provided $Th^D \rightarrow \infty$. Moreover, the slowly changing term $l_k \rightarrow l$ as $k \rightarrow \infty$, where $l$ is a RV \cite{Tartakovsky2005}. Taking expectations on both sides of (\ref{eqn:LSR_TD_apdx}), we get 
\begin{equation}
\text{E}_{1} \left [S_{T_D}\right] = Th^D + \bar r - \bar l + o(1),
\label{eqn:S_TD_apdx}
\end{equation}
where ${\bar r}= \lim_{n_d \rightarrow \infty} \text{E}_1\left[r_{n_d} \right]$ and ${\bar l}= \lim_{k \rightarrow \infty }\text{E}_1\left[l_k \right]$. 
The following expression of $\text{E}_{1} \left [S_{T_D}\right]$ is derived by taking expectations on both sides of (\ref{eqn:Sn_2nd}) \cite{Tartakovsky2005}, 
\begin{equation}
\text{E}_{1} \left [S_{T_D}\right] = \text{E}_1 \left[T_D\right] \left( \text{E}_1\left[Z_1 \right] + |\log(1-\rho)| \right).
\label{eqn:E_S_TD_apdx}
\end{equation}
Furthermore, $\text{E}_1\left[Z_1 \right]$ can be approximated as $\text{E}_1\left[\text{D}\left(\mathpzc{f}^e_{1},\mathpzc{f}_{0}\right) \right]$ (\ref{eqn:opt_kld}) as explained in \cite{watermarking_tac}. Combining (\ref{eqn:S_TD_apdx}) and (\ref{eqn:E_S_TD_apdx}), and rearranging the term we get (\ref{eqn:ADD_2nd}), where $ADD = \text{E}_1 \left[T_D\right]$. 

A brief derivation of FAR is provided as follows. A detailed one can be found in \cite{Tartakovsky2005}.
\begin{equation}
	\begin{aligned}
	FAR &= \text{E}^{\pi}\left[1-p_{\tau} \right] \\
	& = \text{E}^{\pi}\left[\frac{1}{1+\rho \exp\left(LSR_{\tau} \right)} \right]	\text{ [using (\ref{eqn:pk_1st})]} \\
	&=\text{E}^{\pi}\left[\frac{1}{\exp\left(LSR_{\tau}\right)}\frac{1}{\rho+\exp\left(-LSR_{\tau} \right)} \right]
	\end{aligned}
\label{eqn:FAR_1_apdx}
\end{equation}
False alarm will occur when $LSR_{k}$ crosses $Th^D$ during the normal system operation. Therefore, $\exp\left(-LSR_{\tau} \right) \le \exp\left(-Th^D\right) \rightarrow 0$ as $Th^D \rightarrow \infty$. So, (\ref{eqn:FAR_1_apdx}) can be approximated as 
\begin{equation}
	\begin{aligned}
		FAR &= \frac{1}{\rho}\text{E}^{\pi}\left[\frac{1}{\exp\left(LSR_{\tau}\right)}\right](1+ o(1)), \text{ as }Th^D \rightarrow \infty \\
		&=\frac{1}{\rho}\exp\left(-Th^D \right)\text{E}^{\pi}\left[ \exp\left(Th^D - LSR_{\tau}\right)\right](1+ o(1))
	\end{aligned}
\label{eqn:FAR_2_apdx}
\end{equation}
$\text{E}^{\pi}\left[ \exp\left(Th^D - LSR_{\tau}\right)\right]$ can be approximated by $\xi$ using the overshoot $r_{n_d}$ statistics \cite{siegmund2013sequential} as given in (\ref{eqn:Xi}). Replacing $\text{E}^{\pi}\left[ \exp\left(Th^D - LSR_{\tau}\right)\right]$ by $\xi$ in (\ref{eqn:FAR_2_apdx}), we get (\ref{eqn:FAR}).
%The detailed derivation of (\ref{eqn:opt_kld}) can be found in [ref].
\section{proof of Theorem~\ref{th:DeltaLQG}} \label{apdx:DeltaLQG}
The proposed parsimonious watermarking mechanism can be assumed to be a always present watermarking scheme, where the watermarking signal is ${\bf e}_{s,k} = s_{k-1}{\bf e}_k$. $s_{k-1}$ and ${\bf e}_k$ are assumed to be uncorrelated since they are generated from two independent processes. The variance of ${\bf e}_{s,k}$, ${\bf \Sigma}_{e_s}$, is derived as 
\begin{equation}
		{\bf \Sigma}_{e_s}=\text{E}_0\left[s^2_{k-1}{\bf e}_k{\bf e}^T_k\right]=\text{E}_0\left[s^2_{k-1}\right]{\bf \Sigma}_e, 
		\label{eqn:sigma_es_apdx}
\end{equation}
where ${\bf \Sigma}_e = \text{E}_0\left[{\bf e}_k{\bf e}^T_k\right]$. Since $s_k = 1 \text{ or } 0$, $\text{E}_0\left[s^2_{k-1}\right] = \text{E}_0\left[s_{k}\right]$. Therefore, ${\bf \Sigma}_{e_s}$ takes the following form,
\begin{equation}
	\begin{aligned}
		{\bf \Sigma}_{e_s} &= \text{E}_0\left[s_{k}\right] {\bf \Sigma}_e = \frac{ANW}{\text{E}\left[\Gamma \right]}{\bf \Sigma}_e  \\
		&=\rho ANW {\bf \Sigma}_e \text{ [since } \Gamma \sim Geom\left(\rho\right) \text{]}.
	\end{aligned}
	\label{eqn:sigma_es_2_apdx}
\end{equation}
Now, the increase in the control cost, $\Delta LQG$, is derived using Theorem 3 from \cite{watermarking_tac} for the always present watermarking signal ${\bf e}_{s,k}$ as follows,
\begin{equation}
	\begin{aligned}
		\Delta LQG &= \text{tr} \left({\bf H} {\bf \Sigma}_{e_s}\right) \\
		&=\rho ANW \text{tr} \left({\bf H} {\bf \Sigma}_{e}\right)  \text{, [applying (\ref{eqn:sigma_es_2_apdx})]},
	\end{aligned}
\end{equation}
where $\bf H$ is given in (\ref{eqn:H}).
\section{System Parameters}
\label{apdx:system_params} The following system parameters are used for simulation study.  $\rho = 0.001$.\newline
\textbf{System-A parameters}:
\begin{align*}
{\bf A} &=\begin{bmatrix}0.75 & 0.2 \\0.2 & 1.0 \end{bmatrix}           &  {\bf B} &=\begin{bmatrix}0.9 & 0.5 \\0.1 & 1.2 \end{bmatrix}              &  {\bf C}&=\begin{bmatrix}1.0 & -1.0  \end{bmatrix} \\
{\bf Q} &=diag\begin{bmatrix}1 & 1  \end{bmatrix}           &  {\bf R} &=1              &  {\bf W}&=diag\begin{bmatrix}1 & 2  \end{bmatrix} \\
{\bf U} &=diag\begin{bmatrix}0.4 & 0.7  \end{bmatrix}           &  {\bf A}_a &=0.5             &  {\bf Q}_a &=7.5 
\end{align*}
\textbf{System-B parameters}:
\begin{align*}
{\bf A} =\begin{bmatrix}0.968 &0&0.082 &0 \\ 0&0.978&0&0.064 \\ 0&0&0.917&0 \\ 0&0&0&0.935 \end{bmatrix} &  {\bf B} &=\begin{bmatrix} 0.164&0.004 \\0.002&0.124 \\ 0&0.092 \\ 0.060&0  \end{bmatrix} 
\end{align*}
\begin{align*}
{\bf C} &= \begin{bmatrix} 5 &0 &0 &0 \\  0 &5 &0 &0  \end{bmatrix} &  {\bf R} &=diag\begin{bmatrix}0.5 & 0.5  \end{bmatrix} \\
 {\bf Q} &=diag\begin{bmatrix}0.25 & 0.25 & 0.25 & 0.25  \end{bmatrix}     &  {\bf U} &=diag\begin{bmatrix}2 & 2  \end{bmatrix}      \\
 {\bf W} &=diag\begin{bmatrix}5 & 5 & 1 & 1  \end{bmatrix}        &  {\bf Q}_a &=diag\begin{bmatrix}6 & 6  \end{bmatrix}     \\
 {\bf A}_a &=diag\begin{bmatrix}0.4 & 0.1 & 0.1 & 0.7  \end{bmatrix}           
\end{align*}

\bibliographystyle{IEEEtran}
\bibliography{IEEEabrv,bibfile}

\end{document}